\documentclass[12pt]{amsart}
\usepackage[utf8]{inputenc}
\usepackage{tikz}
\usepackage{amsthm}
\usepackage{leftindex}
\usepackage{amsmath}
\usepackage{amssymb}
\usepackage{amsthm}
\usepackage{amscd}
\usepackage{mathtools}
\usepackage{stmaryrd}
\usepackage[margin=1in]{geometry}
\usepackage{comment}
\usepackage{tikz-cd}
\usepackage{booktabs}
\usepackage{mathdots}
\usepackage{array}
\usepackage{blkarray}
\usepackage{multirow}
\newcolumntype{L}[1]{>{\raggedright\arraybackslash}p{#1}}

\usepackage{biblatex}
\usepackage{hyperref}

\newtheorem{thm}{Theorem}[section]
\newtheorem*{thm*}{Theorem}
\newtheorem{lem}[thm]{Lemma}

\newtheorem{prop}[thm]{Proposition}
\newtheorem{cor}[thm]{Corollary}

\theoremstyle{definition}
\newtheorem{definition}[thm]{Definition}
\newtheorem{assumption}[thm]{Assumption}
\newtheorem{remark}[thm]{Remark}

\newtheorem{example}[thm]{Example}

\numberwithin{equation}{section}

\newcommand{\kk}{\Bbbk}

\newcommand\xto[1]{\xrightarrow{#1}}
\newcommand{\mf}{\mathfrak}
\newcommand{\mb}{\mathbb}
\newcommand{\mc}{\mathcal}
\newcommand{\mbf}{\mathbf}

\newcommand{\margin}[1]{\scalebox{0.7}{#1}}
\DeclareMathOperator\rank{rank}
\DeclareMathOperator\Gr{Gr}
\DeclareMathOperator\SL{SL}
\DeclareMathOperator\GL{GL}
\DeclareMathOperator\SO{SO}
\DeclareMathOperator\OG{OG}
\DeclareMathOperator\Sym{Sym}
\DeclareMathOperator\ad{ad}
\DeclareMathOperator\Spec{Spec}

\DeclareMathOperator\grade{grade}
\DeclareMathOperator\pdim{pdim}
\DeclareMathOperator\Span{span}

\DeclareMathOperator\Hom{Hom}
\DeclareMathOperator\HSI{HSI}
\DeclareMathOperator\NL{NL}
\newcommand{\Id}{\mathrm{Id}}
\newcommand{\Plucker}{\mathrm{Pl\ddot{u}cker}}
\newcommand{\cotimes}{\,\widehat{\otimes}\,}

\newcommand{\colorZ}[1]{{{#1}}}
\newcommand{\colorX}[1]{{{#1}}}

\newcommand{\ssc}{\mathrm{ssc}}

\newcommand{\Rgen}{\widehat{R}_\mathrm{gen}}
\newcommand{\Fgen}{\mathbb{F}^\mathrm{gen}}

\begin{document}
	
	\title{The linkage class of a grade three complete intersection}
	\author{Lorenzo Guerrieri, Xianglong Ni, Jerzy Weyman}

	\maketitle

	\begin{abstract}
		Working over a field of characteristic zero, we give structure theorems for all grade three licci ideals and their minimal free resolutions. In particular, we completely classify such ideals up to specialization. The descriptions of their resolutions extend earlier results by Buchsbaum-Eisenbud, Brown, and S\'{a}nchez. Our primary tool is the theory of higher structure maps originating from the study of generic free resolutions of length three.
	\end{abstract} 
	
	\setcounter{tocdepth}{2}
	\tableofcontents
	
	\section{Introduction}\label{sec:intro}
	
	An ideal $I$ in a commutative local Noetherian ring $R$ is called \emph{perfect} if $\pdim R/I = \grade I$.  The structure theory of grade 2 perfect ideals is well-understood thanks to the Hilbert-Burch theorem, which implies that any such ideal is generated by the maximal minors of a $n \times (n-1)$ matrix.
	
	If $S$ is another commutative local Noetherian ring and $J \subset S$ is a perfect ideal, we say that $J$ is a specialization of $I$ if $\grade J \geq \grade I$ and there is a local homomorphism $\varphi\colon R \to S$ so that $\varphi(I)S = J$. If this is the case, then in fact $\grade J = \grade I$ and $J$ is also perfect. For each $n \geq 2$, let $M^{n \times (n-1)}$ denote a generic $n \times (n-1)$ matrix, let $S_n$ be the polynomial ring in its entries, localized at the ideal of variables, and let $I_n \subset S_n$ be the ideal of maximal minors of $M^{n\times (n-1)}$. Then Hilbert-Burch implies that every grade 2 perfect ideal is the specialization of exactly one of the ideals $I_n$. In particular, if we write $J \approx J'$ when there exists a perfect ideal specializing to both $J$ and $J'$, then $\approx$ defines an equivalence relation on grade 2 perfect ideals.

	It would be desirable to have an analogous classification for perfect ideals of grade $c \geq 3$, but an obstacle arises in this setting. Let $M^{2\times 4}$ be a generic $2 \times 4$ matrix and let $M^{3 \times 3}_s$ be a generic $3 \times 3$ symmetric matrix. The ideals $J = I_2(M^{2\times 4})$ and $J' = I_2(M^{3\times 3}_s)$ both specialize to $(x,y,z) \subset \mb{C}[x,y,z]_{(x,y,z)}$, but there is no ideal which specializes to both $J$ and $J'$. As such, $\approx$ does not define an equivalence relation on grade 3 perfect ideals.
	
	As the consequence of results in \cite{Buchweitz81} and \cite{Herzog80}, at least in the setting of power series rings over a field, this is remedied if one restricts to the smaller class of \emph{licci} ideals. These are the ideals in the linkage class of a complete intersection. The equivalence classes defined by $\approx$ are called \emph{Herzog classes}, and each class admits a generic example that specializes to all members of that class. These generic examples may be viewed as structure theorems for their respective families. In retrospect, the success in analyzing grade 2 perfect ideals can be explained by the fact that all such ideals are licci, but this does not hold for $c \geq 3$ (e.g. the explicit ideals mentioned above are not licci).
	
	Although the existence of structure theorems for licci ideals is guaranteed by this theory, explicit descriptions of Herzog classes and their generic examples have not been computed except in a few special cases. For $c=3$, the pioneering breakthrough was a structure theorem for Gorenstein ideals given by Buchsbaum and Eisenbud in \cite{Buchsbaum-Eisenbud77}. Using linkage, they were also able to deduce a structure theorem for almost complete intersections. This program was carried out further in \cite{Brown87} and \cite{Sanchez89}, where structure theorems for ideals directly linked to almost complete intersections were given.
	
	By definition, any licci ideal can be linked to a complete intersection in some number of steps, so this method of translating structure theorems across linkage appears to be a promising way to inductively classify all licci ideals. One of the main mechanisms for doing so is a classical result showing how to produce a resolution of $R/(K:I)$ given a resolution of $R/I$, as follows.
	
	Let
	\[
	\mb{F} \colon 0 \to P_c \to P_{c-1} \to \cdots \to R
	\]
	be a minimal length projective resolution of $R/I$, and let $\mb{K}$ be the Koszul complex on the map $K \coloneqq R^c \xto{[\alpha_1,\ldots,\alpha_c]} R$ where $\alpha_1,\ldots,\alpha_c \in I$ is a regular sequence. Let $\psi\colon \mb{K} \to \mb{F}$ be any map of complexes covering the quotient $R/K \to R/I$ (where by abuse of notation we also use $K$ to refer to its image in $R$):
	\begin{equation}\label{eq:mapping-cone}
		\begin{tikzcd}
			0 \ar[r] & P_c \ar[r] & P_{c-1} \ar[r] & \cdots \ar[r] & P_1 \ar[r] & R\\
			0 \ar[r] & \bigwedge^c K \ar[r] \ar[u,"\psi_c"] & \bigwedge^{c-1} K \ar[r] \ar[u,"\psi_{c-1}"] & \cdots \ar[r] & K \ar[r] \ar[u,"\psi_1"] & R \ar[u,equals]
		\end{tikzcd}
	\end{equation}
	Finally, we take the dual of the mapping cone of $\psi$ and twist by $\bigwedge^c K$ to obtain a complex
	\[
	R \leftarrow (P_c^* \otimes \bigwedge^c K) \oplus K \leftarrow \cdots \leftarrow (P_2^* \otimes \bigwedge^c K) \oplus \bigwedge^{c-1}K \leftarrow P_1^* \otimes \bigwedge^c K \leftarrow 0.
	\]
	(Note that we have cancelled the split part $R \to R$ at the right end.) The next result originally appeared in \cite{Peskine-Szpiro74}, where it is attributed to Ferrand. It was then extended in scope by Golod in \cite{Golod80}, which is the generality we use here.
	\begin{thm}[\cite{Peskine-Szpiro74},\cite{Golod80}]\label{thm:mapping-cone}
		If $I$ is a grade $c$ perfect ideal in a Noetherian ring $R$, and $\alpha_1,\ldots,\alpha_c \in I$ is a regular sequence, then the complex constructed above is a resolution of $R/((\alpha_1,\ldots,\alpha_c):I)$.
	\end{thm}
	
	The situation is particularly simple for $c=2$. Let $I \subset R$ be a grade 2 perfect ideal and $\alpha_1,\alpha_2 \in I$ a regular sequence. We may produce a resolution
	\begin{equation}\label{eq:intro:res-length-two}
		\mb{F}\colon 0 \to R^{n-1} \to R^n \to R
	\end{equation}
	of $R/I$ where $d_1 = \begin{bmatrix}
		\alpha_1 & \alpha_2 & h_3 & \cdots & h_n
	\end{bmatrix}$. Note that this generating set need not be minimal.
	
	If $K = (\alpha_1,\alpha_2)$ then $R/(K:I)$ admits a resolution
	\[
	0 \to R^{n-2} \xto{d_2'} R^{n-1} \to R
	\]
	where $d_2'$ is dual to the submatrix of $d_2$ obtained by omitting the first two columns, and the ideal $K:I$ is generated by the maximal minors of $d_2'$.
	
	We see that, given a free resolution of a grade 2 perfect ideal $I$, it takes essentially no additional computation to obtain a free resolution of any linked ideal $K:I$. However, the situation is different for grade 3 perfect ideals. In this setting, if $\mb{F}$ is a free resolution of $R/I$, then a free resolution of $R/(K:I)$ may be determined using the differentials of $\mb{F}$ together with a chosen multiplication on $\mb{F}$ making it a differential graded algebra. On the other hand, to determine a multiplication on $\mb{F}'$, one must compute even more supplementary structure on $\mb{F}$; see \cite{AKM}. 
	
	Thus we use more information about the original ideal $I$ to deduce less information about the linked ideal $K:I$. This is problematic, as it limits our ability to inductively study licci ideals. 
	The natural question is whether it is possible to achieve a sort of ``equal exchange'' as we had for $c=2$. That is, given a grade 3 perfect ideal $I$, we would like to compute some supplementary data so that the same supplementary data may be determined for any linked ideal $K:I$ with essentially no additional computation. We would also hope that, if $I$ is licci, then its Herzog class can be read off from this supplementary data. 
	
	The goal of this paper is to show that these hopes can be attained for grade 3 perfect ideals, at least assuming that $R$ has equicharacteristic zero. Theorem~\ref{thm:HST-linked-pair} shows that the ``higher structure maps'' defined in \S\ref{sec:GFR-HST} serve as suitable supplementary data. Then, Theorems \ref{thm:non-licci-locus} and \ref{thm:grade-3-licci-classification} illustrate the utility of these higher structure maps, including how they determine the Herzog class.
	
	The key ingredient in the construction of higher structure maps is a deep connection to representation theory, which we will motivate by example. Returning to the case $c=2$, suppose we have a (not necessarily minimal) resolution $\mb{F}$ of the form \eqref{eq:intro:res-length-two} for a grade 2 perfect ideal $I$. As mentioned in the beginning, the Herzog class of $I$ is determined by its minimal number of generators
	\[
	b_0(I) = \min \{i : \operatorname{Fitt}_i (I) = R\}
	\]
	where the Fitting ideal $\operatorname{Fitt}_i (I)$ is equal to $I_{n-i}(d_2)$ by definition. It is helpful to consider the augmented matrix
	\begin{equation}\label{eq:intro:augmented-matrix}
		\begin{blockarray}{cccc}
			& \margin{$F_2$} & \margin{$F_1$} \\
			\begin{block}{c[ccc]}
				\margin{$F_1$} & d_2 & \Id_n \\
			\end{block}
		\end{blockarray}
	\end{equation}
	whose top exterior power
	\begin{equation}\label{eq:intro:Plucker}
		[\bigwedge^{n-1} F_2 \otimes \bigwedge^{n-1} F_1^* \oplus \bigwedge^{n-2} F_2 \otimes \bigwedge^{n-2} F_1^* \oplus \cdots \oplus R] \xto{w} R
	\end{equation}
	encodes the minors of all sizes of $d_2$ simultaneously. We view the source of $w$ as being graded with $\bigwedge^{n-1} F_2 \otimes \bigwedge^{n-1} F_1^*$ in bottom degree zero, so that $\operatorname{Fitt}_i(I)$ is generated by the entries of $w$ restricted to components in degree $\leq i-1$. As such, writing $k$ for the residue field of $R$, the first nonzero component of $w\otimes k$ is in degree $b_0(I) - 1$.
	
	Here is a geometric restatement of the preceding. The matrix \eqref{eq:intro:augmented-matrix} is surjective, so it determines a map $f\colon \Spec R \to \Gr(n,F_2 \oplus F_1)$. The corresponding Pl\"ucker coordinates are given by \eqref{eq:intro:Plucker}. There is a parabolic subgroup $P \subset \GL(F_2 \oplus F_1)$ defined as
	\[
	P = \{ g\in \GL(F_2 \oplus F_1) : g(F_2) = F_2 \}
	\]
	acting on $\Gr(n,F_2 \oplus F_1)$. The complement of the open $P$-orbit is a codimension 2 Schubert variety $X$, cut out by the vanishing of Pl\"ucker coordinates in $\bigwedge^{n-1} F_2 \otimes \bigwedge^{n-1} F_1^*$. In particular, $f^{-1} X = \Spec R/I$.
	
	As the action of $P$ preserves $X$, the local defining ideals of $X \subset \Gr(n,F_2 \oplus F_1)$ at any two points in the same $P$-orbit are equivalent after a change of coordinates. At the $k$-point determined by $w \otimes k$, this ideal specializes to $I$. The $P$-orbit of this $k$-point is determined by the lowest nonzero component in $w\otimes k$, which in turn is determined by $b_0(I)$.
	
	In summary, we see that the Herzog classes of grade 2 perfect ideals on $\leq n$ generators correspond to the non-maximal $P$-orbits in $\Gr(n,2n-1)$, and that a generic example for each class may be obtained as the local defining ideal of $X$ at any point in the corresponding orbit.
	
	All this may seem like an excessive treatment of a relatively basic situation, but it turns out that this framework generalizes nicely to grade 3 licci ideals. However, while the preceding took place in the realm of Grassmannians, i.e. the representation theory of type $A$ Lie algebras, the generalization to grade 3 will involve the representation theory of Kac-Moody Lie algebras associated to $T$-shaped diagrams.
	
	Accordingly, we will begin with a brief summary of Kac-Moody Lie algebras and their representations in \S\ref{sec:RT-bg}. Then, in \S\ref{sec:example-res}, we will use representation theory to construct a family of length three resolutions which we will later establish as the generic free resolutions of all grade 3 licci ideals. In particular, we demonstrate how to recover the free resolutions from \cite{Buchsbaum-Eisenbud77} and \cite{Brown87}. In \S\ref{sec:Rgen} we recall the construction of the generic ring for free resolutions of length three, and collect a few important results regarding it. We use this to define higher structure maps, and we show in \S\ref{sec:HST-linkage} that they transform elegantly under linkage in a manner which extends results from \cite{AKM}. This culminates in \S\ref{sec:classify}, where we classify all grade 3 licci ideals and show that they are resolved by the complexes from \S\ref{sec:example-res}.
	
	We hope that many of the results in this paper remain valid for grade $c > 3$ as well, but we currently lack the framework to make such generalizations. We conclude with some commentary on this in \S\ref{sec:conclusion}. Throughout this paper, we will always work under the following assumption:
	\begin{assumption}\label{ass:base-field}
		All rings considered throughout are $\mb{C}$-algebras.
	\end{assumption}
	Our commutative algebra results remain valid over any field of characteristic zero by simple base change arguments, but this setting is the most familiar for the representation theory and Schubert varieties that we will use.
	
	\subsection*{Acknowledgements}
	This material is partially based upon work supported by the National Science Foundation under Grant No. DMS-1928930 and by the Alfred P. Sloan Foundation under grant G-2021-16778, while the authors were in residence at the Simons Laufer Mathematical Sciences Institute (formerly MSRI) in Berkeley, California, during the Spring 2024 semester. The first and third authors are supported by the grants MAESTRO NCN-UMO-2019/34/A/ST1/00263 - Research in Commutative Algebra and
	Representation Theory, NAWA POWROTY - PPN/PPO/2018/1/00013/U/00001 - Applications of Lie algebras to Commutative Algebra, and OPUS grant National Science Centre, Poland grant UMO-2018/29/BST1/01290. The first author is also supported by the Miniatura grant
	2023/07/X/ST1/01329 from NCN (Narodowe Centrum Nauki), which funded his visit to SLMath in April 2024.
	
	The authors would like to thank Ela Celikbas, Lars Christensen, David Eisenbud, Sara Angela Filippini, Craig Huneke, Witold Kraskiewicz, Shrawan Kumar, Andrew Kustin, Jai Laxmi, Claudia Polini, Steven Sam, Jacinta Torres, Bernd Ulrich, and Oana Veliche for interesting discussions pertaining to this paper and related topics. We thank the anonymous referee for catching a number of mistakes and for helpful suggestions that improved the clarity of the paper.

	\section{Representation theory background}\label{sec:RT-bg}
	We summarize the necessary results on representation theory and Schubert varieties, working in the general setting of Kac-Moody Lie algebras where possible, since this will be needed for studying $\Rgen$ in \S\ref{sec:Rgen}.
	
	\subsection{Lie algebras and representations}
	
	\subsubsection{Construction}\label{bg:lie-construction}
	Fix integers $p,q,r \geq 1$, and let $T=T_{p,q,r}$ denote the graph
	\[\begin{tikzcd}[column sep = small, row sep = small]
		x_{p-1} \ar[r,dash] & \cdots \ar[r,dash] & x_1 \ar[r,dash] & u \ar[r,dash]\ar[d,dash] & y_1 \ar[r,dash] & \cdots \ar[r,dash] & y_{q-1} \\
		&&& z_1 \ar[d,dash]\\
		&&& \vdots \ar[d,dash]\\
		&&& z_{r-1}
	\end{tikzcd}\]
	Let $n = p+q+r-2$ be the number of vertices. From the above graph, we construct an $n\times n$ matrix $A$, called the \emph{Cartan matrix}, whose rows and columns are indexed by the nodes of $T$:
	\[
	A = (a_{i,j})_{i,j \in T}, \quad a_{i,j} = \begin{cases}
		2 &\text{if $i = j$,}\\
		-1 &\text{if $i,j \in T$ are adjacent,}\\
		0 &\text{otherwise.}
	\end{cases}
	\]
	$T$ is a Dynkin diagram if and only if $1/p + 1/q + 1/r > 1$; in this case we say it is of \emph{finite type}. We next describe how to construct the associated Lie algebra $\mf{g}$, following \cite[Chapter I]{KumarBook}.
	
	Let $\mf{h} = \mb{C}^{2n - \rank A}$, and pick independent sets $\Pi = \{\alpha_i\}_{i\in T} \subset \mf{h}^*$ and $\Pi^\vee = \{\alpha_i^\vee\}_{i\in T} \subset \mf{h}$ satisfying the condition
	\[
	\langle \alpha_i^\vee,\alpha_j \rangle = a_{i,j}.
	\]
	The $\alpha_i$ are the \emph{simple roots} and the $\alpha_i^\vee$ are the \emph{simple coroots}. If $1/p + 1/q + 1/r = 1$, then $T = E_{n-1}^{(1)}$ is of \emph{affine type} and $\rank A = n-1$. Otherwise $\rank A = n$, and $\Pi, \Pi^\vee$ are bases of $\mf{h}^*,\mf{h}$ respectively.
	
	The Lie algebra $\mf{g} \coloneqq \mf{g}(T)$ is generated by $\mf{h}$ together with elements $e_i,f_i$ for $i\in T$, subject to the defining relations
	\begin{gather*}
		[e_i,f_j] = \delta_{i,j} \alpha_i^\vee,\\
		[h,e_i] = \langle h, \alpha_i \rangle e_i, [h,f_i] = -\langle h,\alpha_i \rangle f_i  \text{ for } h \in \mf{h},\\
		[h,h'] = 0 \text{ for } h,h' \in \mf{h},\\
		\ad(e_i)^{1-a_{i,j}}(e_j) = \ad(f_i)^{1-a_{i,j}}(f_j) \text{ for } i \neq j,
	\end{gather*}
	where $\ad(x)$ denotes the adjoint action of $x$ (i.e. $\ad(x)(y) = [x,y]$ for $x,y \in \mf{g}$).
	
	Under the adjoint action of $\mf{h}$, the Lie algebra $\mf{g}$ decomposes into eigenspaces as $\mf{g} = \bigoplus \mf{g}_\alpha$, where
	\[
	\mf{g}_\alpha = \{x \in \mf{g} : [h,x] = \alpha(h)x \text{ for all } h \in \mf{h}\}.
	\]
	This is the \emph{root space decomposition} of $\mf{g}$.
	
	\begin{assumption}\label{ass:not-affine-type}
		For simplicity, we henceforth assume that $T$ is \emph{not} of affine type. Thus the Cartan matrix is invertible, and $\mf{g}$ is generated by $e_i$ and $f_i$ for $i \in T$ since $\alpha_i^\vee$ is a basis of $\mf{h}$. This is just for convenience of exposition; the theory we discuss remains valid in the affine case with minor adjustments.
	\end{assumption}
	\subsubsection{Gradings on $\mf{g}$}\label{bg:lie-grading1}
	
	Let $Q \subset \mf{h}^*$ be the root lattice $\bigoplus_{i\in T} \mb{Z}\alpha_i$. If $\mf{g}_\alpha \neq 0$, then necessarily $\alpha \in Q$. If such an $\alpha$ is nonzero, we say it is a \emph{root}, and denote the set of all roots by $\Delta$. Hence the Lie algebra $\mf{g}$ is $Q$-graded. By singling out a vertex $t \in T$, this $Q$-grading can be coarsened to a $\mb{Z}$-grading by considering only the coefficient of $\alpha_t$. We refer to this as the $t$-grading. If we express a root $\alpha \in \Delta$ as a linear combination of $\{\alpha_t\}_{t\in T}$, it is a fact that either all of the coefficients are nonnegative, in which case we say $\alpha$ is a positive root and write $\alpha > 0$, or all of the coefficients are nonpositive, in which case we say $\alpha$ is a negative root and write $\alpha < 0$. We denote the sets of positive and negative roots by $\Delta^+$ and $\Delta^-$ respectively, so that $\Delta = \Delta^+ \amalg \Delta^-$.
	
	The sum of all $t$-gradings for $t\in T$ is called the \emph{principal gradation} on $\mf{g}$. The degree zero part in the principal gradation is the Cartan subalgebra $\mf{h}$. For $\alpha \in \Delta \cup \{0\}$, we write:
	\begin{itemize}
		\item $\alpha >_t 0$ (resp. $\alpha \geq_t 0$) if the coefficient of $\alpha_t$ in $\alpha$ is positive (resp. nonnegative),
		\item $\alpha \geq 0$ if $\alpha > 0$ or $\alpha = 0$,
	\end{itemize}
	and similarly for $\alpha <_t 0, \alpha \leq_t 0, \alpha < 0, \alpha \leq 0$.
	
	Using these notions, we define a few important subalgebras of $\mf{g}$:
	\begin{align*}
		\mf{p}_t^+ &= \bigoplus_{\alpha \geq_t 0} \mf{g}_\alpha & \mf{p}_t^- &= \bigoplus_{\alpha \leq_t 0} \mf{g}_\alpha\\
		\mf{b}^+ &= \bigoplus_{\alpha \geq 0} \mf{g}_\alpha & \mf{b}^- &= \bigoplus_{\alpha \leq 0} \mf{g}_\alpha\\
		\mf{n}^+ &= \bigoplus_{\alpha > 0} \mf{g}_\alpha & \mf{n}^- &= \bigoplus_{\alpha < 0} \mf{g}_\alpha\\
		\mf{n}_t^+ &= \bigoplus_{\alpha >_t 0} \mf{g}_\alpha & \mf{n}_t^- &= \bigoplus_{\alpha <_t 0} \mf{g}_\alpha
	\end{align*}
	Fixing a $t$, each subalgebra contains the ones below it in the same column.
	
	Write $h_i \in \mf{h}$ for the basis dual to the simple roots $\alpha_i \in \mf{h}^*$. The degree zero part of $\mf{g}$ in the $t$-grading is
	\[
	\mf{g}^{(t)} \times \mb{C}h_t
	\]
	where $\mf{g}^{(t)}$ is the subalgebra generated by $\{e_i, f_i\}_{i \neq t}$ and $\mb{C}h_t$ is the one-dimensional abelian Lie algebra spanned by $h_t$. The decomposition of $\mf{g}$ into $t$-graded components is just its decomposition into eigenspaces for the adjoint action of $h_t$:
	\[
	\mf{g} = \bigoplus_{j  \in \mb{Z}} \ker(\ad(h_t) - j).
	\]
	
	\begin{example}\label{ex:inclusion-of-sl}
		For the Dynkin diagram $A_n$, with vertices labeled as
		\[
		\begin{tikzcd}[column sep = small, row sep = small]
			1 \ar[r,dash] & 2 \ar[r,dash] & \cdots \ar[r,dash] & n,
		\end{tikzcd}
		\]
		the associated Lie algebra is $\mf{sl}_{n+1} \coloneqq \mf{sl}(\mb{C}^{n+1})$. Let $\epsilon_{ij}$ denote the $(n+1)\times(n+1)$ matrix whose entries are all 0 except for a single 1 in the $i$-th row and $j$-th column. Then it is customary to use the Lie algebra generators $e_i = \epsilon_{i,i+1}$ and $f_i = \epsilon_{i+1,i}$ for $i = 0,\ldots,n$.
		
		An ordered sequence of vertices $t_1,\ldots,t_n$ forming a subgraph $A_n \subset T$ yields an inclusion $\mf{sl}_{n+1} \hookrightarrow \mf{g}$ by sending the generators $e_i, f_i$ of $\mf{sl}_{n+1}$ to the corresponding elements $e_{t_i}, f_{t_i} \in \mf{g}$.
	\end{example}
	
	\subsubsection{Representations}\label{bg:reps}
	For more details regarding the representation theory of $\mf{g}$, we refer to \cite[Chapter II]{KumarBook}. Here we only summarize the main points we need. Let $V$ be a representation of $\mf{g}$. For $\lambda \in \mf{h}^*$, define the \emph{$\lambda$-weight space of $V$} to be
	\[
	V_\lambda = \{v \in V : h v = \lambda(h)v \text{ for all }h \in \mf{h}\}.
	\]
	If $V_\lambda \neq 0$, then we say $\lambda$ is a \emph{weight} of $V$. A nonzero vector $v \in V_\lambda$ is a \emph{highest weight vector} if $\mf{n}^+ v = 0$. If such a $v$ generates $V$ as a $\mf{g}$-module, then we say $V$ is a \emph{highest weight module} with highest weight $\lambda$.
	
	Let $\mc{U}$ denote the universal enveloping algebra functor. Representations of $\mf{g}$ are equivalent to modules over $\mc{U}(\mf{g})$. Given $\lambda \in \mf{h}^*$, the \emph{Verma module} $M(\lambda)$ is defined to be
	\[
	M(\lambda) = \mc{U}(\mf{g}) \otimes_{\mc{U}(\mf{b}^+)} \mb{C}_\lambda.
	\]
	Here $\mb{C}_\lambda$ is the one-dimensional $\mf{b}^+$-module where $\mf{h}$ acts by $\lambda$ and $\mf{n}^+$ acts trivially. All the weights of $M(\lambda)$ are in $\lambda + Q$. If $v \in V_\lambda$ is a highest weight vector, then there is a map $M(\lambda) \to V$ sending $1 \mapsto v$. If $V$ is a highest weight module then this map is surjective.
	
	Every Verma module $M(\lambda)$ has a unique maximal proper submodule $J(\lambda)$, namely the sum of all submodules which do not contain $v$. It follows that $L(\lambda) = M(\lambda)/J(\lambda)$ is an irreducible highest weight module with heighest weight $\lambda$, and any such module is isomorphic to $L(\lambda)$.
	
	Let $\omega_i \in \mf{h}^*$ be the basis dual to $\alpha_i^\vee \in \mf{h}$. Explicitly, $\omega_i$ is the linear combination of $\alpha_i$ given by the $i$-th column of $A^{-1}$. These are the \emph{fundamental weights}, and the representations $L(\omega_i)$ are called \emph{fundamental representations}. Their nonnegative integral span is the collection of \emph{dominant weights}.
	
	One can alternatively work with lowest weights instead of highest weights, interchanging the roles of positive and negative parts of the Lie algebra in all of the preceding. The irreducible representation with lowest weight $-\lambda$ is $L(\lambda)^\vee$, where $(-)^\vee$ represents the ``restricted'' dual. That is, $V^\vee \coloneqq \bigoplus V_\lambda^*$ for a weight module $V$ with finite-dimensional weight spaces. One has $V^\vee \subseteq V^*$, with equality when $V$ is finite-dimensional.
	
	\subsubsection{Weight grading on representations}\label{sec:grading}
	The decomposition of $L(\lambda)$ into weight spaces gives an $\mf{h}^*$-grading on $L(\lambda)$. Moreover, all the weights of $L(\lambda)$ are in the translate $\lambda + \bigoplus_{i\in T} \mb{Z} \alpha_i$ of the root lattice.
	
	In \S\ref{bg:lie-grading1} it was described how singling out a vertex $t\in T$ allows us to impose a $\mb{Z}$-grading on $\mf{g}$ by considering only the coefficient of $\alpha_t$ in the $\mf{h}^*$-grading. This works for representations $L(\lambda)$ as well: if $v \in L(\lambda)$ is a highest weight vector then $h_t v = \langle h_t, \lambda \rangle v$ and the eigenvalues for the action of $h_t$ on $L(\lambda)$ are $\langle h_t, \lambda \rangle, \langle h_t, \lambda \rangle-1,\ldots$, terminating iff $L(\lambda)$ is finite-dimensional. The eigenspaces give the $t$-graded components. Each one is a representation of the subalgebra $\mf{g}^{(t)} \times \mb{C}h_t \subset \mf{g}$. In particular, $v$ is a highest weight vector for the top graded component, thus this component is the representation of $\mf{g}^{(t)}$ with highest weight $\sum_{i \neq t} c_i \omega_i$ if $\lambda = \sum_{i\in T} c_i \omega_i$.
	
	\subsubsection{Exponential action and Baker-Campbell-Hausdorff}\label{bg:lie-exp}
	Let $\bigoplus_{i>0} \mb{L}_i$ be a strictly positively graded Lie algebra, e.g. $\mf{n}_t^\pm$ for some $t\in T$. Its bracket naturally extends to one on $\mbf{L} = \prod_{i>0} \mb{L}_i$. Suppose $R$ is a commutative ring on which elements $X\in \mbf{L}$ act by locally nilpotent $R_0$-linear derivations, where $R_0 \subset R$ is a subring. Here ``locally nilpotent'' means that for any $f \in R$, we have $X^N f = 0$ for $N \gg 0$. Then for any $X \in \mbf{L}$, the exponential
	\[
	\exp X = \Id + X + \frac{1}{2!} X^2 + \frac{1}{3!}X^3 + \cdots
	\]
	defines an $R_0$-algebra automorphism of $R$.
	
	Moreover, given Lie algebra elements $X,Y \in \mbf{L}$, the Baker-Campbell-Hausdorff formula gives $Z \in \mbf{L}$ such that $\exp Z = \exp X \exp Y$:
	\[
	Z = X + Y + \frac{1}{2}[X,Y] + \frac{1}{12}[X,[X,Y]] - \frac{1}{12}[Y,[X,Y]] + \cdots.
	\]
	Note that $Z$ is well-defined because only finitely many terms on the right will contribute in each degree.
	
	While many proofs of this are analytic, a purely algebraic proof applicable to our setting is given in \cite[Chapter XVI, Section 2]{HochschildBook}. We will not need the explicit expression for $Z$, only that such an expression exists in terms of iterated commutators. Our main use of the formula is the following.
	
	\begin{lem}\label{lem:BCH-split-up-positively-graded-Lie-algebra}
		For each $i$, let $A_i$ and $B_i$ be such that $\mb{L}_i = A_i \oplus B_i$. Let $\mbf{A} = \prod A_i$ and $\mbf{B} = \prod B_i$. Then for every $X \in \mbf{L}$, there exist unique elements $X_A \in \mbf{A}$ and $X_B \in \mbf{B}$ such that $\exp(X) = \exp(X_A)\exp(X_B)$.
	\end{lem}
	\begin{proof}
		Let $Y_A$ and $Y_B$ denote generic elements of $\mbf{A}$ and $\mbf{B}$. Explicitly: for each $i$, let $\{a_{ij}\}_j$ be a basis of $A_i$, introduce corresponding indeterminates $\alpha_{ij}$, and let $Y_A = \sum \alpha_{ij} a_{ij}$. Define $b_{ij}$, $\beta_{ij}$, and $Y_B = \sum \beta_{ij} b_{ij}$ similarly.
		
		Using Baker-Campbell-Hausdorff, we obtain an expression for $Y$ such that $\exp(Y) = \exp(Y_A)\exp(Y_B)$. Crucially, the coefficient of $a_{ij}$ in iterated commutators of $Y_A$ and $Y_B$ only involve $\alpha_{rs}$ and $\beta_{rs}$ with $r < i$ because $\mb{L}$ is strictly positively graded. Thus the coefficient of $a_{ij}$ in $Y$ has the form
		\[
			\alpha_{ij} + (\text{expression only involving $\alpha_{rs}$ and $\beta_{rs}$ with $r < i$})
		\]
		and similarly the coefficient of $b_{ij}$ in $Y$ has the form
		\[
			\beta_{ij} + (\text{expression only involving $\alpha_{rs}$ and $\beta_{rs}$ with $r < i$}).
		\]
		Thus, equating these expressions to the corresponding coefficients in $X$, we see that we can uniquely solve for $\alpha_{ij}$ and $\beta_{ij}$ in order of increasing $i$, thereby determining $X_A$ and $X_B$.
	\end{proof}
	
	\subsection{Weyl group and subgroups}\label{bg:Weyl-group-and-subgroups}
	Let $W$ denote the Weyl group associated to $T$. It is generated by the  \emph{simple reflections} $\{s_i\}_{t\in T}$. Explicitly,
	\[
	W = \langle \{s_i\}_{i\in T} \mid (s_i s_j)^{m_{ij}} = 1\rangle
	\]
	where $m_{ij} = 1$ if $i=j$, $m_{ij} = 2$ if $i,j \in T$ are not adjacent, and $m_{ij} = 3$ if $i,j \in T$ are adjacent. The group $W$ is finite if and only if $T$ is a Dynkin diagram.
	
	The Weyl group acts on $\mf{h}^*$: the simple reflections act via
	\[
	s_i(\lambda) = \lambda - \langle \alpha_i^\vee, \lambda\rangle \alpha_i.
	\]
	For any $t \in T$, the expression
	\[
	\exp(f_t)\exp(-e_t)\exp(f_t)
	\]
	defines an automorphism of any representation on which the actions of $e_t$ and $f_t$ are locally nilpotent. This includes the adjoint representation of $\mf{g}$ and $L(\lambda)$ for any dominant integral $\lambda$. By abuse of notation, we will also denote this automorphism by $s_t$, although it is really a (non-unique) lift thereof. For any element of the Weyl group, we may define an analogous automorphism by expressing $\sigma$ as a product of simple reflections. The resulting automorphism is not unique, but this is not important for our purposes.
	
	A \emph{word} for $\sigma$ is a sequence of simple reflections whose product is $\sigma$. It is \emph{reduced} if there are no shorter words for $\sigma$. The \emph{length} $\ell(\sigma)$ is the length of a reduced word for $\sigma$. There is a partial order, called the \emph{(strong) Bruhat order} on $W$, defined so that $\sigma \geq \sigma'$ if a reduced word for $\sigma$ contains a reduced word for $\sigma'$ as a (not necessarily consecutive) substring. If this holds for some reduced word for $\sigma$, it in fact holds for all reduced words.
	
	For $t \in T$, let $W_t \subset W$ denote the subgroup generated by all simple reflections other\footnote{In the literature, the subgroup generated by a subset $\{s_i\}_{i \in I}$ of simple reflections is often denoted by $W_{P_I}$, so our $W_t$ would be written as $W_{P_{T \setminus \{t\}}}$. Our notation here is just for brevity.} than $s_t$. Every element of $W/W_t$ has a unique minimal length representative in $W$; we denote the set of all such representatives by $W^t$. If $r \in T$, then every element of $W_r \backslash W / W_t$ has a unique minimal length representative, and we denote the set of all such representatives by $\leftindex^r W^t$. There is a bijection between $\leftindex^r W^t$ and $\leftindex^t W^r$, given by sending $\sigma$ to $\sigma^{-1}$. For $\sigma \in W$, we may use $[\sigma]$ instead of $W_r \sigma W_t$ to denote its double coset if $r,t$ are clear from context.
	
	The subgroup $W_t$ is the stabilizer of the fundamental weight $\omega_t$ under the action of $W$, hence $W/W_t$ and $W^t$ may be identified with the orbit $W\omega_t$. These correspond to the extremal weight spaces in the representation $L(\omega_t)$. This can also be used to compute minimal length representatives as follows. Given $\tau \in W$, choose a sequence of simple reflections $s_{i_1},\ldots,s_{i_n}$ so that the coefficient of $\omega_{i_j}$ in $s_{i_{j-1}} \cdots s_{i_1}\tau \omega_t$ is strictly negative for all $j$, and all coefficients in $s_{i_{n}} \cdots s_{i_1}\tau \omega_t$ are nonnegative. Note that this means $s_{i_{n}} \cdots s_{i_1}\tau \omega_t = \omega_t$ because they are dominant weights in the same $W$-orbit. Then $\sigma = s_{i_1} \cdots s_{i_n}$ is a minimal length representative of $\tau W_t$ and $(s_{i_1},\ldots,s_{i_n})$ is a reduced word for $\sigma$.
	
	The sets $W_r \backslash W / W_t$ and $\leftindex^r W^t$ may be identified with
	\begin{equation}\label{eq:double-coset-reps-in-orbit}
	\left\{\sum_{i \in T} b_i \omega_i \in W \omega_t : b_i \geq 0 \text{ for }i \neq r\right\}.
	\end{equation}
	These are the highest weights of extremal $\mf{g}^{(r)}$-representations inside $L(\omega_t)$, so we will often identify such representations with $\leftindex^r W^t$ accordingly. To get the minimal length representative of the double coset $W_r \tau W_t$, one does the same procedure as for $\tau W_t$, but first replacing $\tau$ with $\rho \tau$ where $\rho \in W_r$ and $\rho\tau \omega_t$ belongs to the set \eqref{eq:double-coset-reps-in-orbit}.

	\subsection{$G/P$ and Schubert varieties}\label{sec:bg-schubert}
	If $\mf{g}$ is of finite type, there is a unique simply connected Lie group $G$ associated to the Lie algebra $\mf{g}$, and the representations of $G$ correspond to those of $\mf{g}$. For a fundamental weight $\omega_t$, the action of $G$ on the highest weight line in $\mb{P}(L(\omega_t))$ has stabilizer $P_t^+$, the subgroup of $G$ corresponding to the maximal parabolic subalgebra $\mf{p}_t^+$ as defined in \S\ref{bg:lie-grading1}. Hence the orbit of this highest weight line can be identified with the homogeneous space $G/P_t^+$. For Dynkin type $A_n$ with the standard labeling of vertices, this construction produces the Grassmannian $\Gr(t,n+1)$. Accordingly, the reader may think of $G/P_t^+$ as a ``generalized Grassmannian.''
	
	This theory generalizes to the Kac-Moody setting, but there are some subtleties. Let $L(\omega_t)^\vee$ be the restricted dual of $L(\omega_t)$, i.e. the irreducible representation with lowest weight $-\omega_t$. Let $\widehat{L}(\omega_t) = (L(\omega_t)^\vee)^*$, which contains $L(\omega_t)$, but is strictly larger unless $L(\omega_t)$ is finite-dimensional. Concretely, $L(\omega_t)$ is the sum of its weight spaces, and taking the product instead gives $\widehat{L}(\omega_t)$.
	
	\subsubsection{Definition of the homogeneous space $G/P$}
	In the literature, there are two distinct objects that one might call $G/P_t^+$ in the Kac-Moody setting. The one studied in \cite[Chapter VII]{KumarBook} is an ind-variety, which is the union of finite-dimensional Schubert cells. It can be embedded in $\mb{P}(L(\omega_t))$. A different object is considered in \cite{Kashiwara89}, which is instead the union of infinite-dimensional opposite Schubert cells and resides in the larger $\mb{P}(\widehat{L}(\omega_t))$. The latter is better suited for our purposes, though the distinction goes away if we restrict to a finite-dimensional cell (as we will for our main results later).
	
	Let $\mf{A} = \bigoplus_{n\geq 0} L(n\omega_t)^\vee$. Since $L((n+1)\omega_t)^\vee$ appears with multiplicity 1 inside of $L(n\omega_t)^\vee \otimes L(\omega_t)^\vee$, we may use this to define a $\mf{g}$-equivariant multiplication on $\mf{A}$ which makes it a graded $\mb{C}$-algebra generated by $L(\omega_t)^\vee$:
	\[
	\bigoplus_{n\geq 0} L(n\omega_t)^\vee = (\Sym L(\omega_t)^\vee) / I_\Plucker.
	\]
	The ideal $I_\Plucker$ is comprised of subrepresentations vanishing on a highest weight vector $v \in L(\omega_t)$. We define $G/P_{t}^+ = \operatorname{Proj} \mf{A} \subset \mb{P}(\widehat{L}(\omega_t))$.
	
	\begin{remark}
		Kashiwara's construction in \cite{Kashiwara89} is different from the preceding, but to see that the objects coincide, see \cite[Remark 2.1.ii]{Kashiwara-Shimozono09}. (Although the referenced remark is about the complete flag variety $G/B$, we obtain $G/P_t^+$ by using $\omega_t$ instead of the regular dominant weight $\omega$ used there.) That remark implies that the defining ideal of Kashiwara's $G/P_t^+$ must contain $I_\Plucker$. In fact it must be equal, since any larger $\mf{g}$-equivariant ideal would contain some entire graded component of $\mf{A}$ and thus cut out the empty set.
	\end{remark}
	
	\subsubsection{Pl\"ucker coordinates and Schubert cells}\label{subsec:Plucker-coords-Schubert-cells}
	Let $v$ be a highest weight vector in $L(\omega_t)$. It determines a point in $\mb{P}(\widehat{L}(\omega_t))$, which by abuse of notation we will also denote $v$. This is the \emph{Borel-fixed point} of $G/P_t^+$. For $\sigma \in W^t$, the point $\sigma v$ is a \emph{torus-fixed point} of $G/P_t^+$.
	
	Let $\hat{\mf{n}}^- = \prod_{\alpha < 0} \mf{g}_\alpha$ and $\hat{\mf{n}}^+ = \prod_{\alpha > 0} \mf{g}_\alpha$. Each $\sigma \in W^t$ determines a Schubert cell $C_\sigma = \exp(\hat{\mf{n}}^+)\sigma v$ and opposite Schubert cell $C^\sigma = \exp(\hat{\mf{n}}^-)\sigma v$. These are affine spaces; explicitly we have the identifications
	\[
		\bigoplus_{\alpha > 0, \sigma^{-1}\alpha < 0} \mf{g}_\alpha = \prod_{\alpha > 0, \sigma^{-1}\alpha < 0} \mf{g}_\alpha \xto{\sim} C_\sigma, \quad \prod_{\alpha < 0, \sigma^{-1}\alpha < 0} \mf{g}_\alpha \xto{\sim} C^\sigma.
	\]
	via $Y \mapsto \exp(Y)\sigma v$. The former has dimension $\ell(\sigma)$ by \cite[Proposition 7.4.16]{KumarBook} and the latter has codimension $\ell(\sigma)$ by \cite[Corollary 4.5.8]{Kashiwara89}.
	
	Elements of $L(\omega_t)^\vee$ are called \emph{Pl\"ucker coordinates}. Let $p_e$ denote a lowest weight vector of $L(\omega_t)^\vee$ and let $p_\sigma = \sigma p_e$. The set $\{p_\sigma : \sigma\in W^t\}$ is the set of \emph{extremal Pl\"ucker coordinates}; they are defined up to scale. The representation $L(\omega_t)^\vee$ may have weights other than those in the $W$-orbit of $-\omega_t$, and thus there may be non-extremal Pl\"ucker coordinates. (The representation is \emph{miniscule} if all weights belong to the same $W$-orbit, in which case all Pl\"ucker coordinates are extremal.)
	
	Fix a $w \in W^t$ and consider the set $S=\{p_\sigma : \sigma \in W^t, \sigma \not\geq w\}$. For $\sigma\in W^t$ with $\sigma \geq w$, all elements of $S$ vanish on $C^\sigma$. On the other hand, if $\sigma \not\geq w$, then the zero locus of $S$ is disjoint from $C^\sigma$ because $p_\sigma \neq 0$ on $C^\sigma$. Thus the zero locus of $S$ is the union $\bigcup_{\sigma \geq w} C^\sigma$, which is equal to the closure of $C^\sigma$ by \cite[Proposition 4.5.11]{Kashiwara89}. This closure is called the \emph{opposite Schubert variety} $X^w$.
	
	For $w \leq \sigma$, the \emph{Kazhdan-Lusztig variety} $\mc{N}_\sigma^w$ is defined to be the intersection $X^w \cap C_\sigma$. It has codimension $\ell(w)$ inside of $C_\sigma$ by \cite[Lemma 7.3.10]{KumarBook}. Note that while Kashiwara's $X^w$ is not the same as the one in Kumar's book, they become the same after restricting to $C_\sigma$, since they are cut out by the same Pl\"ucker coordinates.

	\section{A family of resolutions of length three}\label{sec:example-res}
	In this section we construct a family of length three resolutions which generalize the well-known Buchsbaum-Eisenbud resolution for the submaximal pfaffians of a generic $(2n+1)\times(2n+1)$ skew-symmetric matrix.
	
	\subsection{Construction of the resolutions}\label{sec:liccires-construction}
	Let $r_1 \geq 1$, $r_2\geq 2$, and $r_3\geq 1$ be positive integers. Define $f_0 = r_1$, $f_1 = r_1 + r_2$, $f_2 = r_2 + r_3$, and $f_3 = r_3$. In this subsection we define the differentials of certain free resolutions $\mb{F}$ of the form
	\[
	\mb{F} \colon 0 \to R^{f_3} \xto{d_3} R^{f_2} \xto{d_2} R^{f_1} \xto{d_1} R^{f_0}
	\]
	over polynomial rings $R$. We defer the proofs that $d^2 = 0$ and that $\mb{F}$ is acyclic to \S\ref{sec:liccires-acyclicity}. The base ring $R$ and the differentials $d_i$ depend on an additional parameter $\sigma$, which we now explain.
	
	\subsubsection{The coordinate ring $R$ of $C_\sigma$}\label{sec:liccires-base-ring}
	Fix the graph $T = T_{r_1+1,r_2-1,r_3+1}$ with labels
	\[
	\begin{tikzcd}[column sep=small, row sep=small]
		x_{r_1} \ar[r,dash] & \cdots \ar[r,dash] & \colorX{x_1} \ar[r,dash] & u\ar[r,dash]\ar[d,dash] & y_1\ar[r,dash] &\ar[r,dash] \cdots\ar[r,dash] & y_{r_2-2}\\
		&&& \colorZ{z_1}\ar[d,dash]\\
		&&& \vdots\ar[d,dash]\\
		&&& z_{r_3}
	\end{tikzcd}
	\]
	and use the setup and notation from \S\ref{sec:bg-schubert}. Let $\sigma\in \leftindex^{z_1}{W}^{x_1}$, and let $R_\sigma$ be the coordinate ring of the Schubert cell $C_\sigma$ inside of $G/P_{x_1}^+$, which we henceforth abbreviate as $G/P$. Since we will be fixing a $\sigma$ throughout, we will simply denote $R_\sigma$ by $R$ for brevity in this section, though we will reintroduce the subscript when keeping track of $\sigma$ is inmportant (e.g. \S\ref{sec:classify}). Explicitly, if we let
	\[
	\mf{n}_\sigma = \bigoplus_{\alpha>0, \sigma^{-1}\alpha <0} \mf{g}_\alpha
	\]
	then $R = \Sym (\mf{n}_\sigma)^*$. Note that even if $\mf{g}$ is infinite-dimensional, in which case $G/P$ is an ind-variety, this cell $C_\sigma$ is a finite-dimensional affine space $\mb{A}^{\ell(\sigma)}$. The variables corresponding to coordinates on $\mf{g}_\alpha$ are given multidegree $-\alpha$. In this manner, the ring $R$ has a grading by the root lattice $Q = \bigoplus_{i\in T} \mb{Z}\alpha_i$, where each variable is ``negatively graded'' in the sense that all multidegrees are nonpositive.
	
	We assumed $\sigma$ to be a minimal length representative of its double coset, which guarantees that if $\alpha > 0$ and $\sigma^{-1} \alpha < 0$ for some root $\alpha$, then $\alpha >_{z_1} 0$ and $\sigma^{-1}\alpha <_{x_1} 0$. The former implies that if we coarsen the multigrading to a $\mb{Z}$-grading via the projection $Q \to \mb{Z}\alpha_{z_1}$, then all variables have strictly negative degree.
	
	Let $v \in G/P \subset \mb{P}(\widehat{L}(\omega_{x_1}))$ be the Borel-fixed point. The conditions $\alpha > 0$ and $\sigma^{-1}\alpha < 0$ in the definition of $\mf{n}_\sigma$ imply that $\exp(\mf{n}_\sigma)$ has a well-defined action on highest and lowest weight representations of $\mf{g}$ respectively. In particular, it is well-defined on $L(\omega_{x_1})$, $L(\omega_{x_1})^\vee$, and $\widehat{L}(\omega_{x_1}) = (L(\omega_{x_1})^\vee)^*$. We will parametrize $C_\sigma$ as $\exp(\mf{n}_\sigma)\sigma v$. Explicitly, let $Y$ denote a generic element of $\mf{n}_\sigma$, with coefficients in $R$. In other words, if $\{a_i\}$ is a basis of $\mf{n}_\sigma$ and $\{a_i'\}$ its dual basis,
	\[
	Y = \sum a_i \otimes a_i' \in \mf{n}_\sigma \otimes \mf{n}_\sigma^* \subset \mf{n}_\sigma \otimes R.
	\]
	Then our parametrization of $C_\sigma$ is given by the composite
	\begin{equation}\label{eq:C_sigma-parametrization}
	\mb{C} \otimes R \to \widehat{L}(\omega_{x_1}) \otimes R \xto{\exp(Y)\sigma} \widehat{L}(\omega_{x_1}) \otimes R
	\end{equation}
	where the first map is the heighest weight line corresponding to $v$. This is dual to
	\[
	L(\omega_{x_1})^\vee \otimes R \xto{(\exp (Y)\sigma)^{-1}} L(\omega_{x_1})^\vee \otimes R \to \mb{C} \otimes R
	\]
	where the second map is projection onto the lowest weight space, which describes how the Pl\"ucker coordinates restrict to polynomials in $R$. Note that we can also use $L(\omega_{x_1})$ in place of $\widehat{L}(\omega_{x_1})$ in \eqref{eq:C_sigma-parametrization} since $C_\sigma$ resides in the smaller projective space $\mb{P}(L(\omega_{x_1}))$ anyway.
	
	\subsubsection{The differentials of $\mb{F}$}\label{sec:gen-licci-differentials}
	Let $F_j = \mb{C}^{f_j}$ for $0 \leq j \leq 3$. Following Example~\ref{ex:inclusion-of-sl}, we view the Lie algebras $\mf{sl}(F_j)$ as subalgebras of $\mf{g}$:
	\begin{itemize}
		\item $\mf{sl}(F_0)$ corresponds to the ordered sequence of vertices $x_2,\ldots,x_{r_1}$,
		\item $\mf{sl}(F_1)$ corresponds to the ordered sequence of vertices $y_{r_2-2},\ldots,y_1,u,x_1,\ldots,x_{r_1}$,
		\item $\mf{sl}(F_2)$ corresponds to the ordered sequence of vertices $y_{r_2-2},\ldots,y_1,u,z_1,\ldots,z_{r_3}$, and
		\item $\mf{sl}(F_3)$ corresponds to the ordered sequence of vertices $z_2,\ldots,z_{r_3}$.
	\end{itemize}
	In particular, $\mf{g}^{(x_1)} = \mf{sl}(F_0) \times \mf{sl}(F_2)$ and $\mf{g}^{(z_1)} = \mf{sl}(F_1) \times \mf{sl}(F_3)$.
	
	Let $\omega = \sum_{i\in T} c_i \omega_i$ be a dominant integral weight of $\mf{g}$ (i.e. $c_i \geq 0$ for all $i$) and let $L(\omega)$ be the associated irreducible representation with highest weight $\omega$. For $t \in T$, the top graded component of $L(\omega)$ in the $t$-grading is the irreducible representation of $\mf{g}^{(t)}$ with highest weight $\sum_{i \neq t} c_i \omega_i$. Applying this for $t \in \{x_1,z_1\}$ to the three fundamental representations associated with the extremal vertices $x_{r_1}$, $y_{r_2-2}$, and $z_{r_3}$, we obtain
	\begin{align*}
		L(\omega_{x_{r_1}}) &= \cdots \oplus F_0^* \text{ in the $x_1$-grading and} \\
		&= \cdots \oplus F_1^*\text{ in the $z_1$-grading,}\\
		L(\omega_{y_{r_2-2}}) &= \cdots \oplus F_2\text{ in the $x_1$-grading and} \\
		&= \cdots \oplus F_1\text{ in the $z_1$-grading,}\\
		L(\omega_{z_{r_3}}) &= \cdots \oplus F_2^*\text{ in the $x_1$-grading and}\\
		&= \cdots \oplus F_3^*\text{ in the $z_1$-grading.}
	\end{align*}
	To be precise, the identification of each $\mf{g}^{(t)}$-representation is only up to a nonzero scalar in $\mb{C}$. We fix such identifications; any other choice will only alter the subsequent definitions of $d_i$ by scaling.
	
	Using the above, we define the differential $d_1$ to be dual to
	\[
	F_0^* \otimes R \to L(\omega_{x_{r_1}}) \otimes R \xto{\exp(Y)\sigma} L(\omega_{x_{r_1}}) \otimes R \to F_1^* \otimes R,
	\]
	the differential $d_2$ to be
	\[
	F_2 \otimes R \to L(\omega_{y_{r_2-2}}) \otimes R \xto{\exp(Y)\sigma} L(\omega_{y_{r_2-2}}) \otimes R \to F_1 \otimes R,
	\]
	and the differential $d_3$ to be dual to
	\[
	F_2^* \otimes R \to L(\omega_{z_{r_3}}) \otimes R \xto{\exp(Y)\sigma} L(\omega_{z_{r_3}}) \otimes R \to F_3^* \otimes R.
	\]
	In each composite, the first map is the inclusion of the top $z_1$-graded component, and the last map is the projection onto the top $x_1$-graded component. We define $\mb{F}$ as
	\[
	0 \to F_3 \otimes R \xto{d_3} F_2 \otimes R \xto{d_2} F_1 \otimes R \xto{d_1} F_0 \otimes R.
	\]
	\begin{example}\label{ex:split-exact-complex}
		If $\sigma = e \in W$ is the identity, then the Schubert cell $C_\sigma$ is a point, $R = \mb{C}$, and $Y = 0$. The differentials $d_1^*$, $d_2$, and $d_3^*$ are
		\begin{gather*}
			F_0^* \hookrightarrow L(\omega_{x_{r_1}}) \twoheadrightarrow F_1^*\\
			F_2 \hookrightarrow L(\omega_{y_{r_2-2}}) \twoheadrightarrow F_1\\
			F_2^* \hookrightarrow L(\omega_{z_{r_3}}) \twoheadrightarrow F_3^*
		\end{gather*}
		and $\mb{F}$ is a split exact complex of $\mb{C}$-vector spaces.
	\end{example}
	
	\begin{example}\label{ex:1nn1-resolution}
		Let $n \geq 3$ be an odd integer. We apply this construction with the parameters $r_1 = 1$, $r_2 = n-1$, $r_3 = 1$, and demonstrate that it recovers the well-known Buchsbaum-Eisenbud resolution from \cite{Buchsbaum-Eisenbud77} for the $(k-1) \times (k-1)$ pfaffians of a generic $k\times k$ skew-symmetric matrix, where $k \leq n$ is odd. As we will see, the value of $k$ depends on the choice of the parameter $\sigma$. In particular, $\sigma = e$ corresponds to $k = 1$ and we obtain a resolution of the unit ideal---i.e. a split complex---as was demonstrated in the previous example.
		
		The diagram $T$ is $D_n$, and we label the vertices as
		\[
		\begin{tikzcd}[column sep=small, row sep=small]
			\colorX{n-1} \ar[r,dash]& n-2\ar[r,dash] \ar[d,dash]& n-3\ar[r,dash] & \cdots \ar[r,dash]& 1\\
			& \colorZ{n}
		\end{tikzcd}
		\]
		following Bourbaki.
		
		The simple Lie algebra associated to this Dynkin diagram is $\mf{so}_{2n}$. Using the vertex ${n}$, we may decompose the Lie algebra as
		\[
		\mf{so}_{2n} = \bigwedge^2 F_1^* \oplus \mf{gl}(F_1) \oplus \bigwedge^2 F_1
		\]
		where $F_1 = \mb{C}^n$. Concretely, elements of $\mf{so}_{2n}$ are skew-symmetric endomorphisms of the self-dual space $F_1 \oplus F_1^*$.
		
		On the other hand, we could also use the vertex ${n-1}$ instead of the vertex ${n}$. With this perspective we get another decomposition of the standard representation as $F_2 \oplus F_2^*$ for another $n$-plane $F_2$, where $F_1 \cap F_2$ is $(n-1)$-dimensional (corresponding to the overlapping $A_{n-2}$ diagram consisting of the vertices 1 through $n-2$). Let $e_1,\ldots,e_n$ be a basis for $F_1$ and $e_1',\ldots,e_n'$ its dual basis. Then one can arrange for $e_1,\ldots,e_{n-1},e_n'$ to be a basis for $F_2$ and $e_1',\ldots,e_{n-1}',e_n$ its dual.
		
		The space $F_2 \oplus F_2^*$ comes equipped with an evident quadratic form. The subspace $F_2$ is isotropic; its $\SO(2n)$-orbit in $\Gr(n,2n)$ gives the \emph{isotropic/orthogonal Grassmannian} $\OG(n,2n)$. This is one of two isomorphic components comprising the entire set of isotropic $n$-planes. (Taking $n=2$ as an example, there are two families of lines on a smooth quadric surface in $\mb{P}^3$.)
		
		Here the Borel-fixed point $v\in \OG(n,2n)$ is represented by $F_2$. Its stabilizer ${P_{n-1}^+}\subset \SO(2n)$ corresponds to the subalgebra $\mf{gl}(F_2) \oplus \bigwedge^2 F_2$, and consists of all automorphisms of $F_2 \oplus F_2^*$ of the form
		\[
		(a,\varphi) \mapsto (g(a+f(\varphi)),\varphi\circ g^{-1})
		\]
		where $g \in \GL(F_2)$ is arbitrary and $f\colon F_2^* \to F_2$ is skew-symmetric.
		
		We analogously have ${P_{n}^-} \subset SO(2n)$ consisting of all automorphisms of $F_1 \oplus F_1^*$ of the form
		\[
		(a,\varphi) \mapsto (ga,(\varphi+f(a))\circ g^{-1})
		\]
		where $g \in GL(F_1)$ is arbitrary and $f\colon F_1 \to F_1^*$ is skew-symmetric. 
		
		The orthogonal Grassmannian $G/{P_{n-1}^+}$ decomposes into $\lceil \frac{n}{2} \rceil$ orbits under the action of ${P_{n}^-}$. The orbit containing a given isotropic $n$-plane $F$ is determined by $\dim (F \cap F_1^*)$, which can be any odd\footnote{The other component of the moduli of isotropic $n$-planes contains those with even-dimensional intersection with $F_1^*$.} number from 1 to $n$. The various choices of $\sigma$ correspond to the following isotropic $n$-planes, which are representatives of each ${P_{n}^-}$-orbit:
		\begin{align*}
			v &= \Span(e_1,\ldots,e_{n-1},e_n') = F_2\\
			\sigma_1 v &= \Span(e_1,\ldots,e_{n-3},e_n',e_{n-1}',e_{n-2}')\\
			\vdots\\
			\sigma_{(n-1)/2} v &= \Span(e_n',\ldots,e_1') = F_1^*
		\end{align*}
		We have
		\[
		G/{P_{n-1}^+} = \coprod_{i=0}^{(n-1)/2} {P_{n}^-} \sigma_i v = {P_{n}^-} v \amalg {X}
		\]
		where the complement of the open orbit ${P_{n}^-} v$ is a Schubert variety ${X}$ consisting of all isotropic $F$ such that $\dim (F \cap F_1^*) \geq 3$.
		
		Using the ordered basis $e_1,e_2,\ldots,e_n,e_n',\ldots,e_2',e_1'$ for the ambient space, the cell $C_{\sigma_i}$ can be parametrized as
		\[
		C_{\sigma_i} = \left\{\begin{bmatrix}
			I_{n-1-2i} & 0 & 0 & 0\\
			0 & Y_{2i+1} & I_{2i+1} & 0 
		\end{bmatrix}\right\}
		\]
		where $I_k$ is a $k\times k$ identity matrix, and $Y_{2i+1}$ is a generic $(2i+1)\times(2i+1)$ skew-symmetric matrix, but written so that it is antisymmetric across the antidiagonal (instead of the diagonal).
		
		Let's examine the specific example $\sigma = \sigma_{(n-1)/2}$, so that $C_\sigma$ is parametrized as
		\[
		C_{\sigma} = \left\{\begin{bmatrix}
			Y_{n} & I_{n} 
		\end{bmatrix}\right\}
		\]
		where $Y = Y_n$ is a generic $n\times n$ skew-symmetric matrix. The left half of this block matrix is the dual of
		\[
		F_2 \otimes R \hookrightarrow L(\omega_1) \otimes R \xto{(\exp Y)\sigma} L(\omega_1) \otimes R \to F_1 \otimes R.
		\]
		Next, consider the Pl\"ucker embedding $G/{P_{n-1}^+} \hookrightarrow \mb{P}(L({\omega_{n-1}}))$, where $L({\omega_{n-1}})$ is one of the half-spin representations of $\mf{so}_{2n}$. The coordinates in this projective space are given by $L({\omega_{n-1}})^\vee$, which has a $\mb{Z}$-grading induced by ${n} \in D_n$ as follows:
		\[
		L({\omega_{n-1}})^\vee = F_1 \oplus \bigwedge^3 F_1 \oplus \bigwedge^5 F_1 \oplus \cdots
		\]
		Its symmetric square $S_2 L(\omega_{n-1})$ contains the irreducible representation $L(2\omega_{n-1})$, which is also present in $\bigwedge^n L(\omega_1)$. This reflects the fact that the Pl\"ucker coordinates on $\mb{P}(L({\omega_{n-1}}))$ are square-roots of certain maximal minors of $\begin{bmatrix}
			Y_n & I_n
		\end{bmatrix}$; namely the coordinates in $\bigwedge^j F_1$ are square-roots of minors involving $n-j$ columns from the left block and the complementary $j$ columns from the right block.
		
		So for the differential $d_1$, whose dual is the composite
		\[
		\mb{C} \otimes R \hookrightarrow L(\omega_{n-1}) \otimes R \xto{(\exp Y)\sigma} L(\omega_{n-1}) \otimes R \to F_1^* \otimes R,
		\]
		we recover the $(n-1)\times(n-1)$ pfaffians of $Y$. A similar calculation shows that $d_3$ consists of these pfaffians as well.
	\end{example}
	
	\subsection{Multigrading by the root lattice}\label{sec:liccires-multigrading}
	From our setup, the ring $R$ is graded by the root lattice $Q = \bigoplus_{t\in T} \mb{Z}\alpha_t$. We show that, by giving the free modules in $\mb{F}$ appropriate multidegrees, we can arrange for the differentials $d_i$ to be homogeneous of degree zero. This grading can be inferred combinatorially without explicitly knowing $\exp Y$.
	
	\subsubsection{$Q$-grading on $\mb{F}$}\label{sec:liccires-multigrading-Q}
	From highest to lowest, the sequence of weights in $F_0^*$ is obtained by applying the reflections $s_{x_{r_1}},\ldots,s_{x_2}$ sequentially to $\omega_{x_{r_1}}$:
	\[
	Q_0' \coloneqq (\omega_{x_{r_1}}, \omega_{x_{r_1-1}} - \omega_{x_{r_1}}, \ldots,  \omega_{x_1} -  \omega_{x_2}).
	\]
	The decreasing sequence of weights in $F_1^*$ is obtained by applying the reflections $s_{x_{r_1}},\ldots,s_u,\ldots,s_{y_{r_2-2}}$ sequentially to $\omega_{x_{r_1}}$:
	\[
	Q_1' \coloneqq (\omega_{x_{r_1}}, \omega_{x_{r_1-1}} - \omega_{x_{r_1}},\ldots, \omega_{y_1}-\omega_{y_2},\omega_u- \omega_{y_1},\omega_{z_1} + \omega_{y_1} - \omega_u,\ldots,\omega_{z_1} - \omega_{y_{r_2-2}}).
	\]
	Similarly the decreasing sequence of weights in $F_2 \subset L(\omega_{y_{r_2-2}})$ is
	\[
	Q_2 \coloneqq (\omega_{y_{r_2-2}},\omega_{y_{r_2-3}} - \omega_{y_{r_2-2}},\ldots,\omega_u - \omega_{y_1},\omega_{x_1} + \omega_{z_1} - \omega_u,\ldots,\omega_{x_1}-\omega_{z_{r_3}})
	\]
	and the decreasing sequence of weights in $F_1 \subset L(\omega_{y_{r_2-2}})$,is
	\[
	Q_1 \coloneqq (\omega_{y_{r_2-2}},\omega_{y_{r_2-3}} - \omega_{y_{r_2-2}},\ldots,\omega_u - \omega_{y_1},\omega_{z_1} + \omega_{x_1} - \omega_u,\ldots,\omega_{z_1}-\omega_{x_{r_1}}).
	\]
	Finally the decreasing sequence of weights in $F_2^* \subset L(\omega_{z_{r_3}})$ is
	\[
	Q_2' \coloneqq (\omega_{z_{r_3}},\omega_{z_{r_3-1}}-\omega_{z_{r_3}},\ldots,\omega_{z_1}-\omega_{z_2},\omega_u-\omega_{z_1},\omega_{x_1}+\omega_{y_1}-\omega_u,\ldots,\omega_{x_1}-\omega_{y_{r_2-2}})
	\]
	and the decreasing sequence of weights in $F_3^* \subset L(\omega_{z_{r_3}})$ is
	\[
	Q_3' \coloneqq (\omega_{z_{r_3}},\omega_{z_{r_3-1}}-\omega_{z_{r_3}},\ldots,\omega_{z_1}-\omega_{z_2}).
	\]
	With our grading on $R$ and our definition of $Y$, the automorphism $\exp Y$ of each representation is homogeneous of degree zero by construction. It is the action of $\sigma$ which does not respect the grading.
	
	To work around this, we can alternatively view the maps $d_1^*$, $d_2$, and $d_3^*$ as
	\begin{gather}\label{eq:stitching-multigrading}
		\begin{split}
			\sigma F_0^* \otimes R \hookrightarrow L(\omega_{x_{r_1}}) \otimes R \xto{\exp Y} L(\omega_{x_{r_1}}) \otimes R \twoheadrightarrow F_1^* \otimes R,\\
			\sigma F_2 \otimes R \hookrightarrow L(\omega_{y_{r_2-2}}) \otimes R \xto{\exp Y} L(\omega_{y_{r_2-2}}) \otimes R \twoheadrightarrow F_1 \otimes R,\\
			\sigma F_2^* \otimes R \hookrightarrow L(\omega_{z_{r_3}}) \otimes R \xto{\exp Y} L(\omega_{z_{r_3}}) \otimes R \twoheadrightarrow F_3^* \otimes R.
		\end{split}
	\end{gather}
	Now all the maps in each composite are homogeneous of degree zero. In the following theorem, when we say to take $F_i$ as being generated in a sequence of degrees $(\lambda_1,\ldots,\lambda_{f_i})$, we mean this in the decreasing weight order on $F_i$. So for instance, the highest weight vector of $F_i$ is in degree $\lambda_1$ and the lowest weight vector of $F_i$ is in degree $\lambda_{f_i}$.
	
	In the following statement, if $\underline{\lambda}$ is a sequence, $\operatorname{rev}(\underline{\lambda})$ means its reverse, $\underline{\lambda} + \lambda$ means to add $\lambda$ to each term, $\sigma\underline{\lambda}$ means to apply $\sigma$ to each term, and $-\underline{\lambda}$ means to multiply each term by $-1$.
	\begin{thm}
		Take $F_0 \otimes R$ as being generated in degrees
		\[
		-\sigma \operatorname{rev}(Q_0') + \sigma \omega_{x_{r_1}},
		\]
		$F_1 \otimes R$ generated in degrees
		\[
		-\operatorname{rev}(Q_1') + \sigma \omega_{x_{r_1}}, \text{ or equivalently } Q_1 - \omega_{z_1} + \sigma \omega_{x_{r_1}},
		\]
		$F_2 \otimes R$ generated in degrees
		\[
		\sigma Q_2 - \omega_{z_1} + \sigma \omega_{x_{r_1}}, \text{ or equivalently } -\sigma \operatorname{rev}(Q_2') + \sigma \omega_{x_1} - \omega_{z_1} + \sigma \omega_{x_{r_1}},
		\]
		and $F_3 \otimes R$ generated in degrees
		\[
		-\sigma \operatorname{rev}(Q_3') + \sigma \omega_{x_1} - \omega_{z_1} + \sigma \omega_{x_{r_1}}.
		\]
		Then the differentials of
		\[
		\mb{F} \colon F_3 \otimes R \xto{d_3} F_2 \otimes R \xto{d_2} F_1 \otimes R \xto{d_1} F_0 \otimes R
		\]
		as defined in \S\ref{sec:gen-licci-differentials} are homogeneous of degree zero.
	\end{thm}
	\begin{proof}
		This follows by stitching together the gradings in \eqref{eq:stitching-multigrading}. The overall shift by $\sigma \omega_{x_{r_1}}$ is somewhat arbitrary; it is only for the purpose of having one generator of $F_0 \otimes R$ in multidegree zero. All weights in a given highest weight representation differ from one another by elements of the root lattice $Q$, so this overall shift ensures that the multidegrees lie in the root lattice $Q$ instead of some affine translate thereof. Furthermore if $r_1 = 1$ we have $F_0 \otimes R = R$ in multidegree zero.
	\end{proof}
	To figure out the explicit representations of these weights in $Q = \bigoplus_{i\in T} \mb{Z}\alpha_i \cong \mb{Z}^n$, where $n = r_1+r_2+r_3-1$, we can use the Cartan matrix $A$. Since
	\[
	\alpha_i = \sum_j A_{i,j} \omega_j,
	\]
	writing out the multidegrees as linear combinations of the fundamental weights and then multiplying the coefficients by $A^{-1}$ yields our desired $\mb{Z}^n$-multidegrees.
	\begin{remark}
		If $T$ is one of the affine Dynkin diagrams $E_n^{(1)}$ then the Cartan matrix $A$ is not invertible. However, we can simply enlarge the diagram by increasing one of the parameters $r_i$, keeping $\sigma$ the same. Then we are no longer in affine type, and the output $\mb{F}$ is only altered by the addition of a split exact part corresponding to the parameter $r_i$ that was increased.
	\end{remark}
	
	\subsubsection{Symmetry of $x$ and $z$ arms}\label{sec:liccires-exchange-xz}
	The roles played by the left and bottom arms are symmetric in our construction. Recall from \S\ref{bg:Weyl-group-and-subgroups} that if $\sigma \in \leftindex^{z_1}W^{x_1}$, then $\sigma^{-1} \in \leftindex^{x_1} W^{z_1}$. Hence we can do the same construction of \S\ref{sec:liccires-construction} interchanging the roles of the $x$ and $z$ arms with $\sigma^{-1}$ in place of $\sigma$. Let
	\[
	\mf{n}_{\sigma^{-1}} = \bigoplus_{\alpha > 0, \sigma\alpha < 0} \mf{g}_\alpha.
	\]
	Let $R' = \Sym (\mf{n}_{\sigma^{-1}})^*$, $Y' \in \mf{n}_{\sigma^{-1}} \otimes R'$ be the generic element of $\mf{n}_{\sigma'}$, and
	\[
	\mb{F}' \colon 0 \to F_0 \otimes R' \to F_1 \otimes R' \to F_2 \otimes R' \to F_3 \otimes R'
	\]
	be the sequence of differentials produced in this setting.
	
	There is an involution $\tau$ of $\mf{g}$, the \emph{Cartan involution}, which interchanges the Lie algebra generators $e_i \leftrightarrow -f_i$ for $i \in T$ and acts by $-1$ on $\mf{h}$ (here $f_i$ refers to \S\ref{bg:lie-construction}, not the integers $f_i$ fixed throughout this section). We have an isomorphism of nilpotent subalgebras
	\[
	\mf{n}_{\sigma^{-1}} = \bigoplus_{\alpha > 0, \sigma \alpha < 0} \mf{g}_\alpha \xto{\sigma\tau} \mf{n}_{\sigma} = \bigoplus_{\alpha > 0, \sigma^{-1}\alpha < 0} \mf{g}_\alpha
	\]
	which dually induces an isomorphism $R \xto{\cong} R'$ and a map $\mf{n}_{\sigma} \otimes R \to \mf{n}_{\sigma} \otimes R'$ which sends $Y \mapsto \sigma\tau Y'$.
	
	The formula $(Xf)(v) = f((-\tau X)v)$ for $X \in\mf{g}$, $v \in L(\omega)$, and $f \in L(\omega)^\vee$ defines an action of $\mf{g}$ on $L(\omega)^\vee$ that makes it isomorphic to $L(\omega)$. We fix identifications $L(\omega) \cong L(\omega)^\vee$, and by restriction we get identifications $F_i \cong F_i^*$ as vector spaces.
	\begin{prop}\label{prop:liccires-exchange-xz}
		$\mb{F}' \cong \mb{F}^* \otimes_R R'$ via the isomorphism $R \xto{\cong} R'$ described above.
	\end{prop}
	\begin{proof}
		As an example, let us consider the differential $d_1^*$ of $\mb{F}^*$. By construction it is
		\[
		0 \to F_0^* \otimes R \to L(\omega_{x_{r_1}}) \otimes R \xto{\exp(Y)\sigma} L(\omega_{x_{r_1}}) \otimes R \to F_1^* \otimes R.
		\]
		Base-change to $R'$ amounts to replacing $Y$ with $\sigma\tau Y'$, and
		\[
		\exp(\sigma\tau Y')\sigma = \sigma\exp(\tau Y') \sigma^{-1} \sigma = (\exp(-\tau Y') \sigma^{-1})^{-1}.
		\]
		We have the commutative diagram
		\[
		\begin{tikzcd}
			0 \ar[r]& F_0^* \otimes R' \ar[r] \ar[d,"\cong"] & L(\omega_{x_{r_1}}) \otimes R' \ar[rr,"(\exp(-\tau Y')\sigma^{-1})^{-1}"]\ar[d,"\cong"] && L(\omega_{x_{r_1}}) \otimes R' \ar[r]\ar[d,"\cong"] & F_1^* \otimes R'\ar[d,"\cong"]\\
			0 \ar[r]& F_0 \otimes R' \ar[r] & L(\omega_{x_{r_1}})^\vee \otimes R' \ar[rr,"(\exp(Y')\sigma^{-1})^{-1}"] && L(\omega_{x_{r_1}})^\vee \otimes R' \ar[r] & F_1 \otimes R'
		\end{tikzcd}
		\]
		where the vertical maps are induced by the involution $\tau$. Note that the action of $(\exp(Y')\sigma^{-1})^{-1}$ on $L(\omega_{x_{r_1}})^\vee \otimes R'$ is dual to the action of $\exp(Y')\sigma^{-1}$ on $L(\omega_{x_{r_1}}) \otimes R'$, so the bottom row is the last differential of $\mb{F}'$ by definition.
		
		The situation for the other two differentials is completely analogous, so we omit it.
	\end{proof}
	The isomorphism $R \xto{\cong} R'$ is not degree-preserving. By construction, we instead have:
	\begin{cor}\label{cor:xz-interchange-grading}
		The grading on $\mb{F}'$ is obtained by applying $\sigma^{-1}$ to the grading on $\mb{F}$ and then multiplying by $-1$.
	\end{cor}
	
	\begin{example}\label{ex:1562-nonzero-multi}
		Let $r_1 = 2$, $r_2 = 4$, and $r_3 = 1$. The diagram is $T = E_6$ and we use Bourbaki numbering of the vertices so that $(1,\colorX{2},\colorZ{3},4,5,6) = (z_2,\colorX{x_1},\colorZ{z_1},u,y_1,y_2)$. Take
		\[
		\sigma = \colorZ{s_3} s_4 s_2 s_5 s_6 s_1 s_4 s_5 s_3 s_4 \colorX{s_2} \in \leftindex^{3}W^2.
		\]
		The construction applied to $\sigma^{-1} \in \leftindex^{2}W^3$ yields
		\[
		\mb{F}'\coloneqq 0 \to R' \to R'^6 \to R'^5 \to R'^2.
		\]
		Taking its dual and shifting so that the last term is generated in degree zero, we obtain the following multigrading:
		\begin{align*}
			0 \to
			\bigoplus&\begin{matrix}
				R'((1, & {5}, & 3, & 5, & 2, & 1))\\
				R'((1, & {6}, & 4, & 7, & 4, & 2))
			\end{matrix}
			\to\\[1em]\to
			\bigoplus&\begin{matrix}
				R'((1, & {4}, & 3, & 5, & 3, & 2))\\
				R'((1, & {4}, & 3, & 5, & 3, & 1))\\
				R'((1, & {4}, & 3, & 5, & 2, & 1))\\
				R'((1, & {4}, & 3, & 4, & 2, & 1))\\
				R'((1, & {4}, & 2, & 4, & 2, & 1))\\
				R'((0, & {4}, & 2, & 4, & 2, & 1))
			\end{matrix}
			\to\\[1em]\to
			\bigoplus&\begin{matrix}
				R'((0, & {2}, & 1, & 2, & 1, & 0))\\
				R'((0, & {2}, & 1, & 2, & 1, & 1))\\
				R'((1, & {3}, & 2, & 3, & 2, & 1))\\
				R'((1, & {3}, & 2, & 4, & 2, & 1))\\
				R'((1, & {3}, & 3, & 4, & 2, & 1))
			\end{matrix}
			\to\\\to
			&R'.
		\end{align*}
		Here for example $R'((1,5,3,5,2,1))$ means a copy of $R'$ generated in multidegree
		\[
		-(\alpha_1 + 5\alpha_2 + 3\alpha_3 + 5\alpha_4 + 2\alpha_5 + \alpha_6) \in Q
		\]
		As mentioned in \S\ref{sec:liccires-base-ring}, if we coarsen the multigrading to a $\mb{Z}$-grading by sending $\sum c_i \alpha_i$ to $-c_2$, the ring $R'$ is a positively graded polynomial ring. The resolution has the grading
		\[
		0 \to R'(-5) \oplus R'(-6) \to R'^6(-4) \to R'^2(-2) \oplus R'^3(-3) \to R'.
		\]
		(In fact, this is none other than the resolution in \cite[Proposition~3.6]{Brown87} with $n = 5$.)
	\end{example}
	
	\subsection{Proof that $\mb{F}$ is a resolution}\label{sec:liccires-acyclicity}
	We now prove that the differentials $d_i$ defined in \S\ref{sec:liccires-construction} actually assemble into an acyclic complex.
	\begin{lem}
		The composites $d_1 d_2$ and $d_2 d_3$ are identically zero.
	\end{lem}
	\begin{proof}
		In the following, $\otimes$ with no subscript means $\otimes_\mb{C}$. The composite $d_1 d_2$ is adjoint to the composite $\epsilon_1 (d_1^* \otimes_R d_2)$ where $\epsilon_1$ is the contraction $F_1^* \otimes F_1\to \mb{C}$ tensored with $R$. By construction the tensor product $d_1^* \otimes_R d_2$ is
		\[
		F_0^* \otimes F_2 \otimes R \hookrightarrow L(\omega_{x_{r_1}}) \otimes L(\omega_{y_{r_2-2}}) \otimes R \xto{\exp(Y)\sigma} L(\omega_{x_{r_1}}) \otimes L(\omega_{y_{r_2-2}}) \otimes R \twoheadrightarrow F_1^* \otimes F_1 \otimes R.
		\]
		The representation $F_1^* \otimes F_1$ is the sum $\mf{sl}(F_1)\oplus \mb{C}$. A highest weight vector for $\mf{sl}(F_1)$ is also a highest weight vector for $F_0^* \otimes F_2$, with weight $\omega_{x_{r_1}} + \omega_{y_{r_2-2}}$. In particular, $\mf{sl}(F_1)$ and $F_0^* \otimes F_2$ belong to the same irreducible $\mf{g}$-representation $L(\omega_{x_{r_1}} + \omega_{y_{r_2-2}})$ inside of the tensor product, whereas the $\mb{C}$ factor in $F_1^* \otimes F_1$ belongs to $L(\omega_{z_1})$. Thus the image of $d_1^* \otimes_R d_2$ lands in $\mf{sl}(F_1) \otimes R$. As $\epsilon_1$ is the projection onto the complementary $\mb{C}$ factor, the composite $\epsilon_1(d_1^* \otimes_R d_2)$ is zero as claimed.
		
		Similarly, the composite $d_2 d_3$ is adjoint to the composite $(d_2 \otimes_R d_3^*)\iota_2$ where $\iota_2$ is the canonical map $\mb{C}\to F_2 \otimes F_2^*$ tensored with $R$. The tensor product $d_2 \otimes_R d_3^*$ is
		\[
		F_2 \otimes F_2^* \otimes R \hookrightarrow L(\omega_{y_{r_2-2}}) \otimes L(\omega_{z_{r_3}}) \otimes R \xto{\exp(Y)\sigma} L(\omega_{y_{r_2-2}}) \otimes L(\omega_{z_{r_3}}) \otimes R \twoheadrightarrow F_1 \otimes F_3^* \otimes R.
		\]
		The representation $F_2 \otimes F_2^*$ is the sum $\mf{sl}(F_2) \oplus \mb{C}$. A highest weight vector for $\mf{sl}(F_2)$ is also a highest weight vector for $F_1 \otimes F_3^*$, with weight $\omega_{y_{r_2-2}}+\omega_{r_3}$. Again we have that $\mf{sl}(F_2)$ and $F_1 \otimes F_3^*$ belong to the same irreducible $\mf{g}$-representation $L(\omega_{x_{r_1}} + \omega_{y_{r_2-2}})$ inside of the tensor product, whereas the $\mb{C}$ factor in $F_2 \otimes F_2^*$ belongs to $L(\omega_{x_1})$. So the tensor product $d_1^* \otimes_R d_2$ is zero on the image of $\iota_2$.
	\end{proof}
	
	To prove the acyclicity of $\mb{F}$, we will make use of the Buchsbaum-Eisenbud acyclicity criterion. The \emph{rank} of a homomorphism $d$ between free $R$-modules is the maximum value of $r$ for which $\bigwedge^r d$ is nonzero, i.e. it is the size of the largest nonvanishing minor of $d$.
	\begin{definition}\label{def:true-grade}
		Let $R$ be a ring, $I \subset R$ an ideal, and
		\[
		\nu_n := \text{maximum length of an $R[x_1,\ldots,x_n]$-sequence in $IR[x_1,\ldots,x_n]$}.
		\]
		The \emph{(true) grade} of $I$, introduced by Northcott in \cite{Northcott76}, is $\grade I \coloneqq \sup_{n \geq 0} \nu_n$. If $R$ is Noetherian, then this recovers the usual notion of grade. 
	\end{definition}
	\begin{thm}[\cite{Buchsbaum-Eisenbud73}]\label{thm:acyclicity-criterion}
		Let $R$ be a ring. A complex
		\[
		0 \to F_n \xto{d_n} F_{n-1} \xto{d_{n-1}} \cdots \xto{d_2} F_2 \xto{d_1} F_0
		\]
		of free $R$-modules is exact if and only if
		\[
		\rank F_k = \rank d_k + \rank d_{k+1}
		\]
		and
		\[
		\grade I(d_k) \geq k
		\]
		for $k = 1,\ldots,n$, where $I(d_k)\coloneqq I_{\rank d_k}(d_k)$ is the ideal of $(\rank d_k)\times(\rank d_k)$ minors of $d_k$.
	\end{thm}
	\begin{remark}
		The original theorem statement in \cite{Buchsbaum-Eisenbud73} assumed moreover that $R$ is Noetherian, but Northcott showed in \cite{Northcott76} that it holds without this assumption provided one uses the notion of true grade. This will not be important for us at present, since our base ring is a finitely generated polynomial ring throughout this section, but it will be useful for some arguments involving potentially non-Noetherian rings in \S\ref{sec:GFR-HST}.
	\end{remark}
	As the ring $R$ is a polynomial ring, in particular Cohen-Macaulay, the grade of an ideal is the same as its codimension. Furthermore, since $R$ is a domain, the fact that $\mb{F}$ is a complex implies $\rank d_i \leq r_i$. Therefore it is sufficient to prove that $\grade I_{r_i}(d_i) \geq i$ for each $i$, from which $\rank d_i = r_i$ follows as a consequence. In fact, we will show that $\grade I_{r_i}(d_i) = 3$.
	
	For this, let $G/P = G/P_{x_1}^+$ as in \S\ref{sec:liccires-base-ring} and let $X^w \subset G/P$ be the codimension 3 opposite Schubert variety associated to $w = s_{z_1}s_u s_{x_1}$. For each cell $C_\sigma$ meeting $X^w$, it is known that the codimension of $\mc{N}_\sigma^w = X^w \cap C_\sigma$ inside of $C_\sigma$ is 3; see \cite[Lemma~7.3.10]{KumarBook}.
	\begin{lem}\label{lem:extremal-coords-on-X}
		The Pl\"ucker coordinates belonging to the bottom $z_1$-graded component $\bigwedge^{r_0} F_1 \subset L(\omega_{x_1})^\vee$ cut out $X^w$ set-theoretically in $G/P$. Thus on the cell $C_\sigma$, the ideal generated by the entries of
		\begin{equation}\label{eq:gen-licci-coords-on-KL}
			\bigwedge^{r_1} F_1 \otimes R \hookrightarrow L(\omega_{x_1})^\vee \otimes R \xto{(\exp(Y)\sigma)^{-1}} L(\omega_{x_1})^\vee \otimes R \twoheadrightarrow \mb{C} \otimes R
		\end{equation}
		has codimension 3, where the first map is inclusion of the bottom $z_1$-graded component and the last map is projection onto the lowest weight space (equivalently the bottom $x_1$-graded component).
	\end{lem}
	\begin{proof}
		Following \S\ref{subsec:Plucker-coords-Schubert-cells}, for $\rho \in W^{x_1}$, let $p_\rho = \rho p_e$ denote the corresponding extremal Pl\"ucker coordinate, where $p_e \in L(\omega_{x_1})^\vee$ is a lowest weight vector. Then $X^w$ is set-theoretically cut out by $\{p_\rho : \rho \not\geq w\}$. The condition $\rho \not \geq w$ means that a reduced word for $\rho$ does not contain $(s_{z_1},s_u, s_{x_1})$ as a subword.
		
		Let $\rho = s_{t_N} \cdots s_{t_1}$ where $(s_{t_N},\ldots,s_{t_1})$ is a reduced word. The assumption that $\rho$ is a minimal length representative of $[\rho]\in W/W_{x_1}$ means that for all $i = 1,\ldots,N$, the weight $s_{t_{i-1}}s_{t_{i-2}}\cdots s_{t_1} \omega_{x_1}$ has a positive coefficient for $\omega_{t_i}$. From this it is easy to see that $t_i = z_1$ for some $i$ implies $t_j = u$ for some $j < i$. Since $t_1 = x_1$, we conclude $\rho \geq w$ if and only if $t_i = z_1$ for some $i$.
		
		Thus the extremal Pl\"ucker coordinates defining $X^w$ are exactly the extremal Pl\"ucker coordinates belonging to the bottom $z_1$-graded component of $L(\omega_{x_1})^\vee$. This component is dual to the representation of $\mf{g}^{(z_1)} = \mf{sl}(F_1)\times \mf{sl}(F_3)$ with highest weight $\omega_{x_1}$, so it is $\bigwedge^{r_0} F_1$. All the weight spaces in $\bigwedge^{r_0} F_1$ are extremal (i.e. $\bigwedge^{r_0} F_1$ is \emph{miniscule}), so this representation is equal to the span of the extremal Pl\"ucker coordinates vanishing on $X^w$, and we are done.
		
		The other part of the lemma statement is just reiterating the fact that
		\[
		L(\omega_{x_1})^\vee \otimes R \xto{(\exp(Y)\sigma)^{-1}} L(\omega_{x_1})^\vee \otimes R \twoheadrightarrow \mb{C} \otimes R
		\]
		gives the restriction of Pl\"ucker coordinates to the affine cell $C_\sigma = \exp(Y)\sigma v$.
	\end{proof}
	
	Using this lemma, we can prove the acyclicity of $\mb{F}$ via the Buchsbaum-Eisenbud acyclicity criterion. Before we do so, it is helpful to note a few representations appearing the $z_1$-graded decompositions of $L(\omega_{y_{r_2-2}})$ and $L(\omega_{z_{r_3}})$:
	\begin{itemize}
		\item Since
		\[
		s_{z_1} s_u s_{y_1} \cdots s_{y_{r_2-2}} \omega_{y_{r_2-2}} = \omega_{x_1} + \omega_{z_2} - \omega_{z_1},
		\]
		the $\mf{g}^{(z_1)}=\mf{sl}(F_1)\times\mf{sl}(F_3)$-representation with highest weight $\omega_{x_1} + \omega_{z_2}$ appears as an extremal subrepresentation in $L(\omega_{y_{r_2-2}})$. This is $\bigwedge^{r_1} F_1^* \otimes F_3$.
		\item Since
		\[
		s_{z_1} s_{z_2} \cdots s_{z_{r_3}} \omega_{z_{r_3}} = \omega_u - \omega_{z_1},
		\]
		the $\mf{g}^{(z_1)}=\mf{sl}(F_1)\times\mf{sl}(F_3)$-representation with highest weight $\omega_{u}$ appears as an extremal subrepresentation in $L(\omega_{z_{r_3}})$. This is $\bigwedge^{r_1+1} F_1^*$.
	\end{itemize}
	
	\begin{thm}\label{thm:F-is-acyclic}
		The complex $\mb{F}$ resolves a Cohen-Macaulay $R$-module supported on the Kazhdan-Lusztig variety $\mc{N}_\sigma^w = X^w \cap C_\sigma \subset C_\sigma = \Spec R$, where $w = s_{z_1} s_u s_{x_1}$.
	\end{thm}
	\begin{proof}	
		We will prove that powers of the Pl\"ucker coordinates \eqref{eq:gen-licci-coords-on-KL} can be found in $I_{r_i}(d_i)$ for $i = 1,2,3$. From Lemma~\ref{lem:extremal-coords-on-X} it then follows that $\grade I_{r_i}(d_i) = 3$. This means $\mb{F}^*$ is also acyclic, thus $\mb{F}$ resolves a perfect module, or equivalently a Cohen-Macaulay module since $R$ is a polynomial ring.
		
		The method for exhibiting powers of the extremal Pl\"ucker coordinates inside $I_{r_i}(d_i)$ was hinted at near the end of Example~\ref{ex:1nn1-resolution}, where we demonstrated that $I_{n-1}(d_2)$ contained the squares of entries in $d_1$.
		
		More generally, we will show the following. Let $\Omega$ denote the entries of the matrix \eqref{eq:gen-licci-coords-on-KL}, i.e. the extremal Pl\"ucker coordinates set-theoretically cutting out $X^w$ on the affine cell $C_\sigma$. For each Pl\"ucker coordinate $p \in \Omega$, we will show that $p \in I_{r_1}(d_1)$, $p^{r_3+1} \in I_{r_2}(d_2)$, and $p^{r_2-1} \in I_{r_3}(d_3)$.
		
		The situation for $d_1$ is relatively straightforward. The bottom $x_1$-graded component of $L(\omega_{x_{r_1}})^\vee$ is $F_0$, so the bottom $x_1$-graded component of $\bigwedge^{r_1} L(\omega_{x_{r_1}})^\vee$ is the one-dimensional representation $\bigwedge^{r_1} F_0$. Its weight is the sum of those in $F_0 \subset L(\omega_{x_{r_1}})^\vee$, which are given by $-Q_0'$ (c.f. \S\ref{sec:liccires-multigrading-Q}). This sum is $-\omega_{x_1}$, so $\bigwedge^{r_1}F_0$ belongs to the irreducible subrepresentation $L(\omega_{x_1})^\vee$. Therefore $\bigwedge^{r_1} d_1$ factors through $L(\omega_{x_1})^\vee \otimes R$ and the maximal minors of $d_1$ are exactly the Pl\"ucker coordinates $p \in \Omega$. As such, the cokernel of $d_1$ is supported on $\mc{N}_\sigma^w$.
		
		For $d_2$, we consider the $(r_3+1)$-th symmetric power of \eqref{eq:gen-licci-coords-on-KL}:
		\begin{equation}\label{eq:acyclic-pf-d2-matrix}
			S_{r_3+1} \bigwedge^{r_2} F_1^* \otimes R\hookrightarrow S_{r_3+1} L(\omega_{x_1})^\vee \otimes R \xto{(\exp(Y)\sigma)^{-1}}  S_{r_3+1} L(\omega_{x_1})^\vee \otimes R \twoheadrightarrow \mb{C} \otimes R.
		\end{equation}
		Here the powers $p^{r_3+1}$ for $p \in \Omega$ correspond to the weights in $S_{r_3+1} \bigwedge^{r_2} F_1^*$ in the $W_{z_1}$-orbit of the lowest weight $-(r_3+1)\omega_{x_1}$. So these powers $p^{r_3+1}$ still appear when we restrict to the irreducible subrepresentation of this lowest weight, generated by the bottom $x_1$-graded component $\mb{C}\subset S_{r_3 + 1} L(\omega_{x_1})^\vee$:
		\[
		S_{(r_3+1)^{r_2}} F_1^* \hookrightarrow L((r_3+1)\omega_{x_1})^\vee \otimes R \to L((r_3+1)\omega_{x_1})^\vee \otimes R \twoheadrightarrow \mb{C} \otimes R.
		\]
		The crucial point is that $L((r_3+1)\omega_{x_1})^\vee$ also resides in $\bigwedge^{f_2} L(\omega_{y_{r_2-2}})^\vee$ as the irreducible representation generated by its bottom $x_1$-graded component $\bigwedge^{f_2} F_2^* \cong \mb{C}$. Again, this can be seen by adding the weights $-Q_2$ in $F_2^*$ (c.f. \S\ref{sec:liccires-multigrading-Q}).
		
		Inside $\bigwedge^{f_2} L(\omega_{y_{r_2-2}})^\vee$, an analysis of weights shows that $\bigwedge^{f_2} F_2^*$ resides in the $\mf{g}^{(z_1)} = \mf{sl}(F_1)\times\mf{sl}(F_3)$-representation
		\[
		\bigwedge^{r_2} F_1^* \otimes \bigwedge^{r_3} (\bigwedge^{r_1} F_1 \otimes F_3^*) \subset \bigwedge^{f_2}(F_1 \oplus \bigwedge^{r_1} F_1 \otimes F_3^*) \subset \bigwedge^{f_2} L(\omega_{y_{r_2-2}})^\vee.
		\]
		Hence by comparing \eqref{eq:acyclic-pf-d2-matrix} to
		\[
		\bigwedge^{r_2} F_1^* \otimes \bigwedge^{r_3} (\bigwedge^{r_1} F_1 \otimes F_3^*) \otimes R \hookrightarrow \bigwedge^{f_2} L(\omega_{y_{r_2-2}})^\vee \otimes R \xto{(\exp(Y)\sigma)^{-1}} \bigwedge^{f_2} L(\omega_{y_{r_2-2}})^\vee \otimes R \twoheadrightarrow \bigwedge^{f_2} F_2^* \otimes R
		\]
		we see that for each $p \in \Omega$, the power $p^{r_3 + 1}$ appears as a maximal minor of
		\[
		L(\omega_{y_{r_2-2}})^\vee \otimes R \xto{(\exp(Y)\sigma)^{-1}} L(\omega_{y_{r_2-2}})^\vee \otimes R \twoheadrightarrow F_2^* \otimes R
		\]
		involving $r_2$ columns from $F_1^* \otimes R$. But the restriction of this map to $F_1^* \otimes R$ is none other than the dual of the differential $d_2$ by construction, and thus cofactor expansion of the determinant implies $p^{r_3+1} \in I_{r_2}(d_2)$ as claimed.
		
		The situation for $d_3$ is very similar to that of $d_2$. Consider the $(r_2-1)$-th symmetric power of \eqref{eq:gen-licci-coords-on-KL}:
		\begin{equation}\label{eq:acyclic-pf-d3-matrix}
			S_{r_2-1} \bigwedge^{r_1} F_1 \otimes R \hookrightarrow S_{r_2-1} L(\omega_{x_1})^\vee \otimes R \xto{(\exp(Y)\sigma)^{-1}}  S_{r_2-1} L(\omega_{x_1})^\vee \otimes R \twoheadrightarrow \mb{C} \otimes R.
		\end{equation}
		Again the powers $p^{r_2-2}$ for $p\in \Omega$ correspond to the weights in $S_{r_2-2} \bigwedge^{r_1}F_1$ in the $W_{z_1}$-orbit of the lowest weight $-(r_2-1)\omega_{x_1}$. Restricting to the irreducible subrepresentation $L((r_2-1)\omega_{x_1})^\vee$ we have
		\[
		S_{(r_2-1)^{r_1}} F_1 \hookrightarrow L((r_2-1)\omega_{x_1})^\vee \otimes R \to L((r_2-1)\omega_{x_1})^\vee \otimes R \twoheadrightarrow \mb{C} \otimes R.
		\]
		The reason for considering this subrepresentation is that it also appears in $\bigwedge^{f_2} L(\omega_{z_{r_3}})^\vee$ as the irreducible representation generated by its bottom $x_1$-graded component $\bigwedge^{f_2} F_2 \cong \mb{C}$, seen by adding the weights $-Q_2'$ in $F_2$ (c.f. \S\ref{sec:liccires-multigrading-Q}). An analysis of weights shows that $\bigwedge^{f_2} F_2$ is contained in
		\[
		\bigwedge^{r_3} F_3 \otimes \bigwedge^{r_2} (\bigwedge^{r_1+1} F_1) \subset \bigwedge^{f_2}(F_3 \oplus \bigwedge^{r_1+1} F_1) \subset \bigwedge^{f_2} L(\omega_{z_{r_3}})^\vee
		\]
		so comparing \eqref{eq:acyclic-pf-d3-matrix} to
		\[
		\bigwedge^{r_3} F_3 \otimes \bigwedge^{r_2} (\bigwedge^{r_1+1} F_1) \otimes R \hookrightarrow \bigwedge^{f_2} L(\omega_{z_{r_3}})^\vee \otimes R \xto{(\exp(Y)\sigma)^{-1}} \bigwedge^{f_2} L(\omega_{z_{r_3}})^\vee \otimes R \twoheadrightarrow \bigwedge^{f_2} F_2 \otimes R
		\]
		we see that for each $p \in \Omega$, the power $p^{r_2-1}$ appears as a maximal minor of
		\[
		L(\omega_{z_{r_3}})^\vee \otimes R \xto{(\exp(Y)\sigma)^{-1}} L(\omega_{z_{r_3}})^\vee \otimes R \twoheadrightarrow F_2 \otimes R
		\]
		involving all $r_3$ columns from $F_3 \otimes R$. These $r_3$ columns exactly comprise the differential $d_3$ by definition, so cofactor expansion proves $p^{r_2-1} \in I_{r_3}(d_3)$.
		
		Alternatively, one could deduce $\grade I_{r_3}(d_3) \geq 3$ by applying the already established result that $\grade I_{r_1}(d_1) \geq 3$ to the resolution $\mb{F}'$ discussed in \S\ref{sec:liccires-exchange-xz} and then using Proposition~\ref{prop:liccires-exchange-xz}. Geometrically, this amounts to showing that the maximal minors of $d_3$ can be interpreted as the extremal Pl\"ucker coordinates cutting out $\mc{N}_{\sigma^{-1}}^{w^{-1}}$ in $C_{\sigma^{-1}} \subset G/P_{z_1}^+$ as opposed to $\mc{N}_{\sigma}^w$ in $C_\sigma \subset G/P_{x_1}^+$.
	\end{proof}
	
	\section{The generic ring}\label{sec:Rgen}
	Hochster formally posed the question of searching for universal free resolutions in \cite{Hochster75}. The motivation is that these universal examples should codify the ``best possible'' structure theorems for free resolutions in a certain sense. To demonstrate, let us recast the Hilbert-Burch theorem in this language. Let $R$ be a local ring, and $I \subset R$ an ideal with $\pdim R/I = 2$. The minimal free resolution of $R/I$ has the form
	\begin{equation}\label{eq:GFR-length-two}
		0 \to R^{n-1} \xto{d_2} R^n \xto{d_1} R
	\end{equation}
	and the Hilbert-Burch theorem states that $I = aJ$ where $a \in R$ is a nonzerodivisor and $J = I_{n-1}(d_2)$ is generated by the maximal minors of $d_2$. Here is an equivalent way of stating the theorem:
	\begin{thm}[\cite{Hilbert1890}, \cite{Burch68}]\label{thm:GFR-Hilbert-Burch}
		Fix an integer $n \geq 2$ and let $R_\mathrm{univ}$ be a polynomial ring in the variables $\{x_{ij}\}_{1 \leq i \leq n, 1 \leq j \leq n-1}$ and one more variable $a_1$. Let $\mb{F}^\mathrm{univ}$ denote the complex
		\[
		0 \to R_\mathrm{univ}^{n-1} \xto{d_2} R_\mathrm{univ}^n \xto{d_1} R_\mathrm{univ}
		\]
		in which $d_2$ is the generic matrix with entries $x_{ij}$ and $d_1 = a_1 \bigwedge^{n-1} d_2^*$ for some fixed identification of $R_\mathrm{univ}^n \cong \bigwedge^{n-1} (R_\mathrm{univ}^n)^*$. Then if $R$ is any ring and $\mb{F}$ is a complex of the form \eqref{eq:GFR-length-two}, then there exists a unique homomorphism $R_\mathrm{univ} \to R$ specializing $\mb{F}^\mathrm{univ}$ to $\mb{F}$.
	\end{thm}
	Furthermore, the complex $\mb{F}^\mathrm{univ}$ over $R_\mathrm{univ}$ is itself acyclic. This explains the suggestive notation used in the theorem: $\mb{F}^\mathrm{univ}$ is the \emph{universal example} of a free resolution of the form \eqref{eq:GFR-length-two}.
	\begin{definition}\label{def:GFR-UFR}
		Let $R$ be a ring and $\mb{F}$ a complex of free $R$-modules
		\[
		0 \to R^{f_n} \to \cdots \to R^{f_0}.
		\]
		We say the sequence $\underline{f} = (f_0,\ldots,f_n)$ is the \emph{format} of $\mb{F}$. Suppose that
		\begin{enumerate}
			\item $\mb{F}$ is acyclic, and
			\item for any ring $S$ and free resolution $\mb{G}$ over $S$ with the same format $(f_0,\ldots,f_n)$, there exists a unique homomorphism $R\to S$ specializing $\mb{F}$ to $\mb{G}$, in the sense that $\mb{F} \otimes_R S = \mb{G}$. Note that we require actual equality, not just an isomorphism.
		\end{enumerate}
		Then we say $(R,\mb{F})$ is the \emph{universal pair} for the format $\underline{f}$. We refer to $\mb{F}$ as a \emph{universal free resolution} and $R$ the associated \emph{universal ring}.
	\end{definition}
	The standard category theory argument regarding universal properties shows that, if a universal pair exists for a given format $\underline{f}$, then it is unique up to unique isomorphism.
	
	Conceptually, condition (2) expresses that $(R,\mb{F})$ encodes an equational structure theorem for free resolutions of format $\underline{f}$, and condition (1) expresses that this structure theorem is optimal since any other example must factor through $(R,\mb{F})$. Here the adjective ``equational'' qualifies that these structure theorems assert unique solutions to a system of equations involving the entries of $\mb{G}$. The acyclicity criterion Theorem~\ref{thm:acyclicity-criterion} is not equational, for instance.
	
	In \cite{Hochster75}, Hochster constructed the universal pair $(R_\mathrm{univ},\mb{F}^\mathrm{univ})$ for length 2 formats $\underline{f} = (f_0,f_1,f_2)$. We assume $f_0 \geq r_0$, $f_1 = r_1 + r_2$, and $f_2 = r_2$ for positive integers $r_1,r_2$ to avoid degenerate cases.
	
	Let $d_1$ and $d_2$ be matrices of indeterminates, where $d_1$ is $f_0 \times f_1$ and $d_2$ is $f_1 \times f_2$. The construction starts by taking the ring $R_c$ of generic complexes of format $\underline{f}$, which is a polynomial ring in the aforementioned indeterminates, modulo the ideal $I_1(d_1 d_2) + I_{r_1 + 1}(d_1) + I_{r_2 + 1}(d_2)$. This ring comes equipped with a tautological complex $\mb{F}^c$ whose differentials are $d_i$. The ideals $I_{r_1}(d_1)$ and $I_{r_2}(d_2)$ both have grade 1, so in view of Theorem~\ref{thm:acyclicity-criterion}, the goal is to increase $\grade I_{r_2}(d_2)$ to 2.
	
	This can be achieved by taking the \emph{ideal transform} with respect to $I = I_{r_2}(d_2)$, defined as
	\[
	R = \{h \in \operatorname{Frac} R_c : I^t h \subseteq R_c \text{ for some }t\}.
	\]
	With $\mb{F} = \mb{F}^c \otimes R$, Hochster proves that the pair $(R,\mb{F})$ is the universal pair for $\underline{f}$, and that it essentially encodes the first structure theorem of Buchsbaum and Eisenbud, which we now recall.
	\begin{definition}
		A complex $\mb{F}$ over a ring $R$ is \emph{acyclic in grade $c$} if $\mb{F} \otimes R_\mf{p}$ is acyclic for all primes $\mf{p}$ with $\grade \mf{p} \leq c$.
	\end{definition}
	Recall that by Theorem~\ref{thm:acyclicity-criterion}, acyclicity is equivalent to $\grade I_{r_i}(d_i) \geq i$ for all $i$. By contrast, acyclicity in grade $c$ amounts to the weaker requirement that $\grade I_{r_i}(d_i) \geq \min(i,c+1)$ instead.
	\begin{thm}[\cite{Buchsbaum-Eisenbud74}]\label{thm:BE1}
		Let $\mb{F}$ be a complex of free $R$-modules
		\[
		0 \to F_n \xto{d_n} F_{n-1} \xto{d_{n-1}} \cdots \xto{d_1} F_0
		\]
		that is acyclic in grade 1. Let $f_i = \rank F_i$ and $r_i = \rank d_i$. Fix identifications $\bigwedge^{f_i} F_i \cong R$. Then there exist uniquely determined maps $a_i \colon R \to \bigwedge^{r_i}F_{i-1}$ ($i = 1,\ldots,n$) so that
		\[
		\begin{tikzcd}
			\bigwedge^{r_i} F_i \ar[dr,"-\wedge a_{i+1}",swap]\ar[rr,"\bigwedge^{r_i} d_i"] && \bigwedge^{r_i} F_{i-1}\\
			& \bigwedge^{f_i} F_i \cong R \ar[ur,"a_i",swap]
		\end{tikzcd}
		\]
		commutes, where we set $a_{n+1}$ to be the identity $R \to \bigwedge^0 F_n = R$.
	\end{thm}
	The original theorem was stated under the stronger assumption that $\mb{F}$ is acyclic. The weaker hypotheses here are due to Eagon and Northcott in \cite{Eagon-Northcott73}. The maps $a_i$ are called \emph{Buchsbaum-Eisenbud multipliers}.
	
	After Hochster's construction of the universal pair $(R_\mathrm{univ},\mb{F}^\mathrm{univ})$, many authors established various properties of the ring $R_\mathrm{univ}$. Let $R_a$ be the ring obtained by adjoining the Buchsbaum-Eisenbud multipliers to $R_c$. In Hochster's original treatment, the ring $R_\mathrm{univ}$ was shown to be the integral closure of $R_a$. Huneke later showed in \cite{Huneke81} that the ring $R_a$ is already integrally closed, so $R_\mathrm{univ} = R_a$. In \cite{Pragacz-Weyman90}, Pragacz and Weyman analyzed the relations in $R_\mathrm{univ}$. They also proved that it has rational singularities (this is the only statement so far that uses Assumption~\ref{ass:base-field}). Tchernev established numerous more properties of $R_\mathrm{univ}$ using Gr\"obner bases in \cite{Tchernev01}, and Kustin determined a free resolution of $R_\mathrm{univ}$ as a quotient of a polynomial ring in \cite{Kustin07}.
	
	In \cite{Hochster75}, Hochster expressed doubt in the existence of universal free resolutions for formats of length $\geq 3$, and proposed a weakening of universality in which the specialization is no longer required to be unique.
	\begin{definition}\label{def:GFR-GFR}
		Let $R$ be a ring and $\mb{F}$ an acyclic complex of free $R$-modules. Suppose that the format of $\mb{F}$ is $\underline{f}$, and that for any ring $S$ and free resolution $\mb{G}$ over $S$ with the same format $(f_0,\ldots,f_n)$, there exists a (not necessarily unique) homomorphism $R \to S$ specializing $\mb{F}$ to $\mb{G}$. Then we say $(R,\mb{F})$ is a \emph{generic pair} for the format $\underline{f}$. We refer to $\mb{F}$ as a \emph{generic free resolution} and $R$ the associated \emph{generic ring}.
	\end{definition}
	Unfortunately, the removal of unique specialization means that such generic pairs are not unique for a fixed format. Indeed, one could simply adjoin indeterminates to an existing generic ring without affecting its genericity.
	
	However, this is a necessary concession, as Bruns confirmed Hochster's doubt in \cite{Bruns84}: universal free resolutions do not exist for formats of length $\geq 3$. In the same paper, he also showed that for any nonnegative integers $r_0,r_1,\ldots,r_n$, the format
	\begin{equation}\label{eq:GFR-rank-condition-on-format}
		\underline{f} = (f_0,\ldots,f_n) = (r_0 + r_1,r_1 + r_2,\ldots,r_{n-1} + r_n, r_n)
	\end{equation}
	admits a generic pair (this condition on $\underline{f}$ is forced by linear algebra).
	
	Bruns's construction of generic free resolutions is via a procedure he calls ``generic exactification.'' Essentially, if $\mb{F}$ is a complex over $R$ and $H_i(\mb{F})\neq 0$ for some $i > 0$, we can pick a map $Z\colon R^N \to F_i$ surjecting onto the cycles which are nonzero in homology. Then we adjoin a generic $f_{i+1} \times N$ matrix of variables $X = [x_{ij}]$ to $R$ and quotient by the relation $d_{i+1}X = Z$ to obtain a new ring $R'$, thereby killing the cycles in $H_i(\mb{F})$. While this process may introduce new cycles, we at least have that the induced map $H_i(\mb{F}) \to H_i(\mb{F}\otimes R')$ is zero. Thus, taking $R_\mathrm{gen}$ to be the direct limit of the rings produced iteratively using this procedure, the complex $\mb{F}^\mathrm{gen} = R_\mathrm{gen}$ is acyclic by construction.
	
	While Bruns's proof is constructive, it does not reveal any properties of the ring $R_\mathrm{gen}$. For length 3 specifically, a more systematic construction of generic free resolutions was carried out in \cite{Weyman89} and \cite{Weyman18}, which we now discuss.
	
	\subsection{Weyman's generic free resolutions of length three}\label{sec:GFR-Weyman}
	If $(R_\mathrm{univ},\mb{F}^\mathrm{univ})$ is a universal pair for the format $\underline{f} = (f_0,f_1,\ldots,f_n)$, then the action of $\prod \GL(f_i)$ on $\mb{F}^\mathrm{univ}$ induces an action of $\prod \GL(f_i)$ on $R_\mathrm{univ}$.
	\begin{example}
		Let $\underline{f} = (f_0,f_1,f_2) = (1,n,n-1)$ and let $F_i = \mb{C}^{f_i}$. From Theorem~\ref{thm:GFR-Hilbert-Burch} we see that, under the action of $\prod \GL(F_i)$, the ring $R_\mathrm{univ}$ is
		\[
		R_\mathrm{univ} = \Sym [(F_1^* \otimes F_2) \oplus (F_0^* \otimes \bigwedge^{f_1} F_1 \otimes \bigwedge^{f_2} F_2^*)]
		\]
		where the variables $x_{ij}$ span the representation $F_1^* \otimes F_2$ and the Buchsbaum-Eisenbud multiplier $a_1$ spans the one-dimensional representation $F_0^* \otimes \bigwedge^{f_1} F_1 \otimes \bigwedge^{f_2} F_2^*$.
	\end{example}
	More generally, the universality property of $(R_a,\mb{F}^a)$ yields an action of $\prod \GL(F_i)$ on $R_a$. The decomposition of $R_a$ under this action is described in \cite{Pragacz-Weyman90}.
	
	If $(R_\mathrm{gen},\mb{F}^\mathrm{gen})$ is an abstract generic pair, there is no natural action of $\prod \GL(F_i)$ on $R_\mathrm{gen}$. Despite this, we can at least try to construct a generic pair taking symmetry into account. Moreover, it turns out that this process can result in the presence of \emph{more} symmetry, rather than less.
	
	\subsubsection{Construction of $\Rgen$}
	As in \S\ref{sec:example-res}, we fix positive integers $r_1 \geq 1$, $r_2 \geq 2$, $r_3 \geq 1$ and define $f_0 = r_1$, $f_1 = r_1 + r_2$, $f_2 = r_2 + r_3$, $f_3 = r_3$. Let $\underline{f} = (f_0,f_1,f_2,f_3)$ and $F_i = \mb{C}^{f_i}$. Often we will abuse notation and write $F_i$ when we really mean $F_i \otimes R$ for some ring $R$.
	
	Weyman constructed a candidate generic pair $(\Rgen(\underline{f}),\Fgen(\underline{f}))$ for the format $\underline{f}$ in \cite{Weyman89}, and verified its acyclicity in \cite{Weyman18}. We will typically suppress $\underline{f}$ from the notation. As mentioned previously, generic pairs are not unique. Henceforth when we say \emph{the} generic ring, we refer to the model $\Rgen$ specifically.
	\begin{remark}
		One can more generally consider formats where $f_0 \geq r_1$, but we will not do so here. The main case of interest will actually be when $f_0 = r_1 = 1$.
	\end{remark}
	We now briefly summarize the construction of $\Rgen$. The starting point is Theorem~\ref{thm:BE1}, which we restate with two small adjustments: we state it only for $c=3$ and we avoid identifying top exterior powers with the base ring in the interest of doing things $\GL(F_i)$-equivariantly.
	\begin{thm}\label{thm:BE1'}
		Let $0 \to F_3 \xto{d_3} F_2 \xto{d_2} F_1 \xto{d_1} F_0$ be a complex of free modules, acyclic in grade 1, of format $(f_0,f_1,f_2,f_3)$ over $R$. Then there are uniquely determined maps $a_3,a_2,a_1$, constructed as follows:
		\begin{itemize}
			\item $a_3$ is the top exterior power
			\[
			a_3 \colon \bigwedge^{f_3} F_3 \to \bigwedge^{f_3} F_2.
			\]
			\item $a_2$ is the unique map making the following diagram commute:
			\[
			\begin{tikzcd}
				\bigwedge^{r_2} F_2 \ar[rr,"\bigwedge^{r_2}d_2"] \ar[dr,"-\wedge a_3",swap] && \bigwedge^{r_2} F_1\\
				& \bigwedge^{f_3} F_3^* \otimes \bigwedge^{f_2} F_2 \ar[ur,dashed,"a_2",swap]
			\end{tikzcd}
			\]
			\item $a_1$ is the unique map making the following diagram commute:
			\[
			\begin{tikzcd}
				\bigwedge^{f_0} F_1 \ar[rr,"\bigwedge^{r_1}d_1"] \ar[dr,"-\wedge a_2",swap] && \bigwedge^{f_0} F_0\\
				& \bigwedge^{f_3} F_3 \otimes \bigwedge^{f_2} F_2^* \otimes \bigwedge^{f_1} F_1 \ar[ur,dashed,"a_1",swap]
			\end{tikzcd}
			\]
		\end{itemize}
	\end{thm}
	Note that $a_1$ is just a scalar. If $\grade I_{f_0}(d_1) \geq 2$, then $a_1$ is an isomorphism.
	
	Theorem~\ref{thm:BE1} can be used to construct a free complex $\mb{F}^a$ of the given format $\underline{f}$ over a ring $R_a$, called the Buchsbaum-Eisenbud multiplier ring, such that $\mb{F}^a$ is acyclic in grade 1 and is universal with respect to this property.
	
	The pair $(R_a,\mb{F}^a)$ is a good starting point for constructing the generic pair $(\Rgen,\Fgen)$, since if $\mb{F}$ is any resolution of the same format over some ring $R$, then of course it is acyclic in grade 1 and so there is a unique map $R_a \to R$ specializing $\mb{F}^a$ to $\mb{F}$.
	
	However, $\mb{F}^a$ is not acyclic: letting $d_i$ denote the differentials of $\mb{F}^a$, we have
	\[
	\grade I_{r_3} (d_3) = 2, \quad \grade I_{r_2} (d_2) = 2, \quad \grade I_{r_1} (d_1) = 1.
	\]
	From the perspective of the acyclicity criterion Theorem~\ref{thm:acyclicity-criterion}, the failure of $\mb{F}^a$ to be acyclic can be attributed to the insufficient grade of $I_{r_3} (d_3)$.
	
	Specifically we have $H_1(\mb{F}^a) \neq 0$, so one strategy to proceed would be to try and kill $H_1(\mb{F}^a)$ following \cite{Bruns84}. The first step would be to handle the Koszul cycles. However, it is not clear how to carry out this recursive procedure systematically in an explicit manner.
	
	The alternative approach carried out in \cite{Weyman89} is to look at the Koszul complex on $\bigwedge^{r_3} d_3$ instead. Explicitly, writing $\mc{K} = \bigwedge^{f_3} F_3^* \otimes \bigwedge^{f_3}F_2$ so that $\bigwedge^{r_3}d_3$ can be viewed as a map $R_a \to \mc{K}$, the fact that $\grade I_{r_3} (d_3) = 2$ implies the existence of nonzero $H^2$ in
	\[
	0 \to \bigwedge^0 \mc{K} \to \bigwedge^1 \mc{K} \to \bigwedge^2 \mc{K} \to \bigwedge^3 \mc{K}.
	\]
	Now the promised connection to representation theory gradually appears. First, in \cite{Weyman89}, the recursive procedure of killing $H^2$ in the complex above was carried out with the aid of a graded Lie algebra $\bigoplus_{i > 0} \mb{L}_i$ where we view $\mb{L}_i$ as residing in degree $-i$. Let
	\[
	\mb{L}^\vee \coloneqq \bigoplus_{i > 0} \mb{L}_i^*, \qquad \mbf{L} \coloneqq (\mb{L}^\vee)^* = \prod_{i > 0} \mb{L}_i.
	\]
	We call $\mbf{L}$ the \emph{defect Lie algebra}.
	
	There is a diagram
	\begin{equation*}\begin{tikzcd}
			0 \ar[r] & \bigwedge^0 \mc{K} \ar[r] & \bigwedge^1 \mc{K} \ar[r] & \bigwedge^2 \mc{K} \ar[r] & \bigwedge^3 \mc{K}\\
			&& \mb{L}^\vee \ar[u, "p"] \ar[r] & \bigwedge^2 \mb{L}^\vee \ar[u, "q"]
	\end{tikzcd}\end{equation*}
	where the dual of the lower horizontal map is the bracket in $\mbf{L}$. Let $p_i$ denote the restriction of $p$ to $\mb{L}_i^* \subset \mb{L}^\vee$ and similarly $q_i$ the restriction of $q = \bigwedge^2 p$ to $(\bigwedge^2 \mb{L}^\vee)_i$. The map $p_1$ is defined to lift a cycle constructed using the second structure theorem of \cite{Buchsbaum-Eisenbud74}. Since $\mb{L}^\vee$ is \emph{strictly} positively graded, $(\bigwedge^2 \mb{L}^\vee)_i$ only involves $\mb{L}_j$ for $j < i$, which allows for recursive computation of the cycles $q_m$ and their lifts $p_m$ for $m \geq 2$.
	
	For positive integers $m$, define $R_m'$ to be the ring obtained from $R_a$ by generically adjoining variables for the entries of $p_1,\ldots,p_m$ and quotienting by all relations they would satisfy on a split exact complex (see for instance \cite[Lemma 2.4]{Weyman89}). Let $R_m$ be the ideal transform of $R_m'$ with respect to $I_{r_2}(d_2) I_{r_3}(d_3)$. The ring $\Rgen$ is defined to be the limit of the rings $R_m$, and $\Fgen \coloneqq \mb{F}^a \otimes \Rgen$.
	
	The idea behind this construction is that adjoining the lifts $p_i$ ought to kill $H^2$ in the Koszul complex, and the ideal transform with respect to $I_{r_2}(d_2) I_{r_3}(d_3)$ ensures that we do not generate homology elsewhere in our complex. The acyclicity of $\Fgen$ was reduced to the exactness of certain 3-term complexes (c.f. \cite[Theorem 3.1]{Weyman89}), and this was later proven in \cite{Weyman18}.
	
	\subsubsection{Exponential action of $\mbf{L}$}
	Given a free resolution $\mb{F}$ over some ring $R$, a choice of maps $p_1,\ldots,p_m$ making
	\begin{equation}\begin{tikzcd}\label{eq:p-lifting}
			0 \ar[r] & R \ar[r,"\bigwedge^{r_3} d_3"] & \mc{K}\otimes R \ar[r] & \bigwedge^2 \mc{K}\otimes R \ar[r] & \bigwedge^3 \mc{K}\otimes R\\
			&& \mb{L}_i^* \otimes R \ar[u, "p_i"] \ar[r] & (\bigwedge^2 \mb{L}^\vee)_i \otimes R \ar[u, "q_i"]
	\end{tikzcd}\end{equation}
	commute determines a map $R_m' \to R$. This extends uniquely to the ideal transform $R_m$ since the image of $I_{r_2}(d_2) I_{r_3}(d_3)$ in $R$ has grade at least 2.
	\begin{lem}\label{lem:GFR-p-determines-w}
		There is a natural bijection
		\[
		\begin{tikzcd}
			\{\text{maps $w\colon \Rgen \to R$ specializing $\Fgen$ to $\mb{F}$}\} \ar[d,"\simeq",leftrightarrow]\\
			\{\text{choices of $\{p_i\}_{i > 0}$ making \eqref{eq:p-lifting} commute}\}
		\end{tikzcd}
		\]
	\end{lem}
	\begin{proof}
		This follows from the preceding discussion since $\Rgen = \lim R_m$.
	\end{proof}
	Furthermore, having chosen $p_i$ for $i < m$, the diagram \eqref{eq:p-lifting} shows that the non-uniqueness of $p_m$ lifting $q_m$ is exactly $\Hom(\mb{L}_m^* \otimes R,R) = \mb{L}_m \otimes R$. In \cite{Weyman89}, the action of $\mbf{L}$ on $\Rgen$ by derivations is described. Specifically, elements $u \in \mb{L}_n$ act on $\Rgen$ by $R_{n-1}$-linear derivations. It is sufficient to describe how they affect (the entries of) $p_{n+k}$ for $k \geq 0$, and this is as follows: the derivation $D_u$ sends $p_n^*$ to
	\[
	\mc{K}^* \xto{\bigwedge^{r_3} d_3^*} \Rgen \xto{u} \mb{L}_n \otimes \Rgen 
	\]
	and $p_{n+k}^*$ to
	\[
	\mc{K}^* \xto{p_k^*} \mb{L}_k \otimes \Rgen \xto{[u,-]} \mb{L}_{n+k} \otimes \Rgen.
	\]
	These are just restatements of the formulas given in \cite[Prop. 2.11]{Weyman89} and \cite[Thm. 2.12]{Weyman89} respectively.
	
	These formulas naturally extend to an arbitrary element $X \in \mb{L} = \prod_{i > 0} \mb{L}_i$; the resulting derivation is well-defined because $\mb{L}_{>n}$ acts by zero on $R_n$. In a slight abuse of notation, we will also write $X$ for the corresponding derivation. Homomorphisms $\Rgen \to R$ correspond to $R$-algebra homomorphisms $\Rgen \otimes R \to R$, and the Lie algebra $\mb{L}\otimes R$ acts on $\Rgen \otimes R$.
	
	For $X \in \mb{L} \otimes R$, the action of $\exp X \coloneqq \sum_{i\geq 0} \frac{1}{i!}X^i$ on $\Rgen \otimes R$ is well-defined since every element of $\Rgen \otimes R$ is killed by a sufficiently high power of $X$. Since $X$ acts by an $(R_a\otimes R)$-linear derivation, it follows formally that $\exp X$ acts by an automorphism fixing $R_a\otimes R$. Such automorphisms completely describe the non-uniqueness of the map $\Rgen \to R$ given a particular resolution $(R,\mb{F})$, as the following observation from \cite{Ni-exp} shows.
	
	\begin{thm}\label{thm:parametrize}
		Let $\mb{F}$ be a resolution of length three over $R$ and let $\Rgen$ be the generic ring for the associated format. Fix a $\mb{C}$-algebra homomorphism $w\colon \Rgen \to R$ specializing $\mb{F}^\mathrm{gen}$ to $\mb{F}$. Then $w$ determines a bijection
		\[
		\mbf{L}\cotimes R \coloneqq \prod_{i > 0} (\mb{L}_i \otimes R) \simeq \{\text{$\mb{C}$-algebra homomorphisms $w'\colon \Rgen \to R$ specializing $\mb{F}^\mathrm{gen}$ to $\mb{F}$}\}.
		\]
		Note that a $\mb{C}$-algebra homomorphism $\Rgen \to R$ can be viewed as an $R$-algebra homomorphism $\Rgen \otimes R\to R$. The correspondence above identifies $X \in \mbf{L} \cotimes R$ with the map $w\exp X$ obtained by precomposing $w$ with the action of $\exp X$ on $\Rgen\otimes R$. 
	\end{thm}
	\begin{proof}
		By Lemma~\ref{lem:GFR-p-determines-w}, the homomorphism $w\colon \Rgen \otimes R \to R$ is completely determined by the choice of the structure maps $p_i$. For $X \in \mbf{L} \cotimes R$, let us write $X = \sum_{i > 0} u_i$ where $u_i \in \mb{L}_i \otimes R$, and let $X_n = \sum_{i=1}^n u_i$ denote the partial sums.
		
		Precomposing $w$ by $\exp X$ or $\exp X_n$ has the same effect on the structure maps $p_k$ for $k \leq n$. Acting by $\exp X$ on $p_1$, we get
		\[
		p_1 + (\bigwedge^{r_3} d_3)u_1^*.
		\]
		Here $u_1^*$ means the dual of $R \xto{u_1} \mb{L}_1 \otimes R$. All possible choices of the structure map $p_1$ are obtained by lifting a particular map $q_1$ in the diagram \eqref{eq:p-lifting}, so it follows that choices of $u_1 \in \mb{L}_1 \otimes R$ correspond to choices for the structure map $p_1$.
		
		Once $X_{n-1}$ has been computed, $u_n \in \mb{L}_n \otimes R$ can be similarly determined by comparing $p_n$ with $p_n'$. Acting by $\exp X$ on $p_n$ gives
		\[
		(p_n + p_{n-1}[u_1,-]^* + \cdots) + (\bigwedge^{r_3} d_3)u_n^*.
		\]
		The first part consists of terms involving $u_k$ for $k<n$, which have already been determined. Once again, \eqref{eq:p-lifting} shows that there is a unique choice of $u_n \in \mb{L}_n \otimes R$ that makes the whole expression equal to $p_n'$.
		
		Proceeding inductively in this fashion, we construct $X \in \mbf{L} \cotimes R$ with the desired property, and the uniqueness at each step is evident as well.
	\end{proof}
	If we view the generic pair as describing an equational structure theorem for $\mb{F}$, this tells us that the general solution to our system of equations is readily obtained once we have a particular solution.
	
	The task of finding a particular solution (a map $w\colon \Rgen \to R$ specializing $\Fgen$ to $\mb{F}$) may be difficult, but here is a helpful observation.
	\begin{prop}\label{prop:equivariant-p}
		Let $R$ be a ring and $\mb{F}$ a free resolution
		\[
		0 \to F_3 \otimes R \to F_2 \otimes R \to F_1 \otimes R \to F_0 \otimes R.
		\]
		\begin{enumerate}
			\item Suppose a group $G$ acts on the free modules $F_i \otimes R$ and the differentials of $\mb{F}$ are $G$-equivariant. Since $\prod \GL(F_i \otimes R)$ acts on $\Rgen \otimes R$ and on $\mb{L}_m^*$, we have an induced action of $G$ on $\Rgen \otimes R$ and $\mb{L}_m^*$. If the maps $p_m$ in \eqref{eq:p-lifting} are chosen to be $G$-equivariant, then the induced map $w\colon \Rgen \otimes R \to R$ is also $G$-equivariant.
			\item Suppose $R$ is graded and $\mb{F}$ is a graded free resolution where the differentials are homogeneous of degree zero. This induces a grading on $\Rgen \otimes R$ and $\mb{L}_m^*$ for each $m$, and it is possible to choose all $p_m$ in \eqref{eq:p-lifting} to be homogeneous of degree zero. The corresponding $w \colon \Rgen\otimes R \to R$ is then also homogeneous of degree zero.
		\end{enumerate}
	\end{prop}
	The construction of the cycle $q_1$ in terms of the differentials $d_i$ and the construction of $q_m$ in terms of $p_1,\ldots,p_{m-1}$ are both $\prod \GL(F_i)$-equivariant. So in the setting of Proposition~\ref{prop:equivariant-p}, (1) the cycle $q_m$ is $G$-equivariant provided $p_1,\ldots,p_{m-1}$ are, and similarly (2) $q_m$ is homogeneous of degree zero provided $p_1,\ldots,p_{m-1}$ are.
	\begin{proof}
		This is evident from the construction of $\Rgen = \lim R_m$ and Lemma~\ref{lem:GFR-p-determines-w}. In the graded setting, it is always possible to recursively take $p_m$ which is homogeneous of degree zero: simply take any lift of $q_m$ (which is homogeneous of degree zero by induction) and discard the components which are not homogeneous of degree zero.
	\end{proof}
	
	\subsubsection{The critical representations}
	
	Recall that we have fixed parameters $r_1 \geq 1$, $r_2 \geq 2$, $r_3 \geq 1$, from which our format $\underline{f}$ is defined as $(f_0,f_1,f_2,f_3) = (r_1,r_1+r_2,r_2+r_3,r_3)$. As in \S\ref{sec:example-res}, consider the diagram $T = T_{r_1+1,r_2-1,r_3+1}$
	\[
	\begin{tikzcd}[column sep=small, row sep=small]
		x_{r_1} \ar[r,dash] & \cdots \ar[r,dash] & \colorX{x_1} \ar[r,dash] & u\ar[r,dash]\ar[d,dash] & y_1\ar[r,dash] &\ar[r,dash] \cdots\ar[r,dash] & y_{r_2-2}\\
		&&& \colorZ{z_1}\ar[d,dash]\\
		&&& \vdots\ar[d,dash]\\
		&&& z_{r_3}
	\end{tikzcd}
	\]
	and let $\mf{g}$ denote the associated Kac-Moody Lie algebra. We define subalgebras $\mf{sl}(F_i)$ of $\mf{g}$ as in \S\ref{sec:gen-licci-differentials}, so e.g. $\mf{g}^{(z_1)} = \mf{sl}(F_1)\times \mf{sl}(F_3)$.
	
	One of the main observations driving the results of \cite{Weyman18} is that
	\begin{align*}
		\mb{L}^\vee = \mf{n}_{z_1}^+ = \bigoplus_{\alpha >_{z_1} 0} \mf{g}_\alpha, \qquad \mbf{L} = \hat{\mf{n}}_{z_1}^- = \prod_{\alpha <_{z_1} 0} \mf{g}_\alpha,
	\end{align*}
	and that the action of $\bigoplus \mb{L}_i$ on $\Rgen$ extends to an action of $\mf{g}$. We refer to \S\ref{bg:lie-grading1} for explanations regarding notation here. Using this connection, Weyman was able to
	\begin{enumerate}
		\item decompose $\Rgen$ into representations of $\mf{sl}(F_0) \times \mf{sl}(F_2) \times \mf{g}$, and
		\item prove the needed exactness of certain complexes from \cite{Weyman89} thereby proving $\Fgen$ is acyclic.
	\end{enumerate}
	Interestingly, while (2) was the original goal, we will never actually use the acyclicity of $\Fgen$. This may come as a surprise, but notice that the classical Hilbert-Burch theorem does not actually claim the universal example to be acyclic; that is a separate result! Similarly, the first structure theorem of Buchsbaum and Eisenbud does not include the statement that $(R_a,\mb{F}^a)$ is universal for finite free complexes acyclic in grade 1. In essence, the fact that these theorems can be recast as ``universal examples'' is a certification that they are the \emph{best possible} structure theorems for their respective objects. From this perspective, it is less surprising that $(\Rgen,\Fgen)$ has utility independent of the acyclicity of $\Fgen$.
	
	On the other hand, (1) is essential to our results, as it allows us to define and analyze so-called ``higher structure maps,'' which will be the topic of \S\ref{sec:GFR-HST}.
	
	We also recall Assumption~\ref{ass:base-field} that although we work over $\mb{C}$ throughout, the results remain valid over $\mb{Q}$. Indeed, the Lie algebra $\mf{g}$ may be defined using the same generators and relations as in \S\ref{bg:lie-construction} over $\mb{Q}$, which results in the \emph{split form} of $\mf{g}$. Its representation theory parallels the situation over $\mb{C}$, and the construction and decomposition of $\Rgen$ remains valid.
	
	\begin{remark}\label{rem:gl-vs-sl}
		There is one subtle point, which is that while the ring $\Rgen$ certainly has an action of $\prod \mf{gl}(F_i)$ by construction, this action does \emph{not} come from an inclusion of $\prod \mf{gl}(F_i)$ into $\mf{sl}(F_0)\times \mf{sl}(F_2) \times \mf{g}$. Nor is it correct to say that $\mf{gl}(F_0) \times \mf{gl}(F_2) \times \mf{g}$ acts on $\Rgen$, since the action of $\mf{gl}(F_2)$ does not commute with the action of $\mf{g}$.
		\begin{enumerate}
			\item Rather, if we let
			\[
			M = \bigwedge^{f_3} F_3 \otimes \bigwedge^{f_2} F_2^* \otimes \bigwedge^{f_1} F_1
			\]
			then the $\prod \mf{gl}(F_i)$-equivariant description of $\mb{L}_1$ from \cite{Weyman89} is
			\[
			\mb{L}_1 = \bigwedge^{r_1+1} F_1^* \otimes F_3 \otimes M.
			\]
			Given an irreducible lowest weight representation $L(\omega)^\vee$ of $\mf{g}$, the action of $\prod \mf{gl}(F_i)$ on $L(\omega)^\vee$ can be inferred from its action on any $z_1$-graded component, e.g. the bottom one.
			\item\label{item:identify-a1} The Buchsbaum-Eisenbud multiplier $a_1$ from Theorem~\ref{thm:BE1'} is a map $M \to \bigwedge^{f_0} F_0$. If $\grade I_{r_0}(d_0) \geq 2$, then $a_1$ is an isomorphism and we may instead view
			\[
			\mb{L}_1 = \bigwedge^{r_1+1} F_1^* \otimes F_3 \otimes \bigwedge^{f_0} F_0
			\]
			which is sometimes more convenient, especially when dealing with resolutions of cyclic modules.
		\end{enumerate}
	\end{remark}

	\begin{example}
		If $\underline{f} = (1,5,6,2)$, then the diagram is $T_{2,3,3} = E_6$. Writing $\mf{g}_i$ for the component of $\mf{g}$ in $z_1$-degree $i$ (i.e. the sum of root spaces $\mf{g}_\alpha$ where the coefficient of $\alpha_{z_1}$ in $\alpha$ is equal to $i$), we have
		\begin{align*}
			\mf{g}_2 &= \bigwedge^4 F_1 \otimes \bigwedge^2 F_3^* \otimes S_2 M^*\\
			\mf{g}_1 &= \bigwedge^2 F_1 \otimes F_3^* \otimes M^*\\
			\mf{g}_0 &= \mf{sl}(F_1) \times \mf{sl}(F_3) \times \mb{C}\\
			\mf{g}_{-1} &= \bigwedge^2 F_1^* \otimes F_3 \otimes M\\
			\mf{g}_{-2} &= \bigwedge^4 F_1^* \otimes \bigwedge^2 F_3 \otimes S_2 M
		\end{align*}
		and $\mbf{L} = \mf{g}_{-2} \oplus \mf{g}_{-1}$.
	\end{example}
	\begin{example}
		If $\underline{f} = (1,6,8,3)$, then the diagram is $T_{2,4,4} = E_7^{(1)}$. Writing $\mf{g}_i$ for the component of $\mf{g}$ in $z_1$-degree $i$, we have
		\begin{align*}
			\vdots\\
			\mf{g}_{4} &= S_{2^2,1^4} F_1 \otimes S_{2,1^2} F_3^* \otimes S_4 M^* \\
			\mf{g}_{3} &= S_{2,1,1,1,1} F_1 \otimes S_{2,1} F_3^* \otimes S_3 M^* \\
			\mf{g}_{2} &= \bigwedge^4 F_1 \otimes \bigwedge^2 F_3^* \otimes S_2 M^* \\
			\mf{g}_{1} &= \bigwedge^2 F_1 \otimes F_3^* \otimes M^* \\
			\mf{g}_0 &= \mf{sl}(F_1) \times \mf{sl}(F_1) \times \mb{C}^2\\
			\mf{g}_{-1} &= \bigwedge^2 F_1^* \otimes F_3 \otimes M \\
			\mf{g}_{-2} &= \bigwedge^4 F_1^* \otimes \bigwedge^2 F_3 \otimes S_2 M\\
			\mf{g}_{-3} &= S_{2,1,1,1,1} F_1^* \otimes S_{2,1} F_3 \otimes S_3 M\\
			\mf{g}_{-4} &= S_{2^2,1^4} F_1^* \otimes S_{2,1^2} F_3 \otimes S_4 M\\
			\vdots
		\end{align*}
		and $\mbf{L}$ is an infinite product in this case. Its decomposition is periodic, satisfying
		\[
		\mb{L}_{i+3} \cong \mb{L}_i \otimes \bigwedge^6 F_1^* \otimes \bigwedge^3 F_3 \otimes S_3 M
		\]
		for all $i > 0$, stemming from the fact that $\mf{g}$ is an affine Lie algebra. This fact is also responsible for the extra copy of $\mb{C}$ appearing in $\mf{g}_0$; c.f. \S\ref{bg:lie-construction}.
	\end{example}
	
	The decomposition of $\Rgen$ into representations for the product $\mf{sl}(F_0) \times \mf{sl}(F_2) \times \mf{g}$ is detailed in \cite{Weyman18}. Of these representations, there are a few of particular interest, which we call the \emph{critical representations}---they are the ones generated by the entries of the differentials $d_i$ and Buchsbaum-Eisenbud multipliers $a_i$ for $\Fgen$. We denote these representations by $W(d_i)$ and $W(a_i)$ respectively. Let $L(\omega)$ be the irreducible representation with highest weight $\omega$ so that $L(\omega)^\vee$ is the irreducible representation with lowest weight $-\omega$. The aforementioned representations are
	\begin{align}
		\begin{split}\label{eq:crit-reps}
			W(d_3) &= F_2^* \otimes L(\omega_{z_{r-1}})^\vee\\
			&= F_2^* \otimes [F_3 \oplus M^* \otimes\bigwedge^{f_0+1} F_1 \oplus \cdots]\\
			W(d_2) &= F_2 \otimes L(\omega_{y_{q-1}})^\vee\\
			&= F_2 \otimes [F_1^* \oplus M^* \otimes F_3^* \otimes \bigwedge^{f_0} F_1 \oplus  \cdots]\\
			W(d_1) &= F_0^* \otimes L(\omega_{x_{p-1}})^\vee\\
			&= F_0^* \otimes [F_1 \oplus M^* \otimes F_3^* \otimes \bigwedge^{f_0+2} F_1 \oplus \cdots]\\
			W(a_3) &= \bigwedge^{f_3} F_2^* \otimes L(\omega_{z_{1}})^\vee\\
			&= \bigwedge^{f_3} F_2^* \otimes [\bigwedge^{f_3} F_3 \oplus \cdots]\\
			W(a_2) &= \bigwedge^{f_2} F_2 \otimes L(\omega_{x_1})^\vee\\
			&= \bigwedge^{f_2} F_2 \otimes [\bigwedge^{r_2}F_1^* \otimes \bigwedge^{f_3} F_3^* \oplus \cdots]\\
			W(a_1) &= \bigwedge^{f_0} F_0^* \otimes \bigwedge^{f_1} F_1 \otimes \bigwedge^{f_2} F_2^* \otimes \bigwedge^{f_3} F_3
		\end{split}
	\end{align}
	where the stated decompositions are into representations of $\prod \mf{gl}(F_i)$, not just $\prod \mf{sl}(F_i)$, following Remark~\ref{rem:gl-vs-sl}.
	\begin{remark}\label{rem:a3-in-d3}
		Although we will not use this fact (so we do not include a proof of it here), the ring $\Rgen$ is generated by $W(d_3)$, $W(d_2)$, $W(d_1)$, $W(a_2)$, and $W(a_1)$. The representation $W(a_3)$ is not needed because it is contained in $S_{r_3}W(d_3)$:
		\[
		W(a_3) = \bigwedge^{f_3} F_2^* \otimes L(\omega_{z_1})^\vee \subseteq \bigwedge^{f_3} F_2^* \otimes \bigwedge^{f_3} L(\omega_{z_{r_3}})^\vee \subseteq S_{r_3} W(d_3).
		\]
	\end{remark}
	
	\subsection{Higher structure maps}\label{sec:GFR-HST}
	Given a map $w\colon \Rgen \to R$ for a complex $(R,\mb{F})$, we denote by $w^{(i)}$ the restriction of $w$ to the representation $W(d_i) \subset R_\mathrm{gen}$ and $w^{(a_i)}$ the restriction of $w$ to the representation $W(a_i)$. We typically view these maps as being $R$-linear with source $L(\omega)^\vee \otimes R$, e.g. we think of $w^{(3)}$ as an $R$-linear map
	\[
	w^{(3)}\colon L(\omega_{z_{r-1}})^\vee \otimes R = [F_3 \oplus M^* \otimes\bigwedge^{f_0+1} F_1 \oplus \cdots] \otimes R \to F_2\otimes R
	\]
	as opposed to a $\mb{C}$-linear map $F_2 \otimes L(\omega_{z_{r-1}})^\vee \to R$.
	
	We use $w^{(*)}_j$ to denote the restriction of $w^{(*)}$ to the $j$-th $z_1$-graded component of the representation, indexed so that $j=0$ corresponds to the bottom graded piece. For instance, $w^{(i)}_0 = d_i$ for $i = 1,3$ and $w^{(2)}_0 = d_2^*$. We call the maps $w^{(*)}_{>0}$ (a specific choice of) \emph{higher structure maps} for $\mb{F}$. Let us demonstrate Theorem~\ref{thm:parametrize} from this perspective.
	\begin{example}\label{ex:multiplication}
		Consider a free resolution $\mb{F}$ of format $(1,f_1,f_2,f_3)$ resolving $R/I$ where $\operatorname{depth} I \geq 2$, and make the identification described in Remark~\ref{rem:gl-vs-sl} \eqref{item:identify-a1}. The structure maps $w^{(i)}_1$ give a choice of multiplicative structure on $\mb{F}$; see \cite[Prop. 7.1]{Lee-Weyman19}. Explicitly, such a resolution has the (non-unique) structure of a commutative differential graded algebra, and the non-uniqueness is evidently seen from the fact that the multiplication $\bigwedge^2 F_1 \to F_2$ may be chosen as any lift in the diagram
		\begin{equation}\label{eq:example-w31}\begin{tikzcd}
				0 \ar[r] & F_3 \ar[r] & F_2 \ar[r] & F_1 \ar[r] & R\\
				&&& \bigwedge^2 F_1 \ar[ul, dashed] \ar[u]
		\end{tikzcd}\end{equation}
		where the map $\bigwedge^2 F_1 \to F_1$ is given by $e_1 \wedge e_2 \mapsto d_1(e_1) e_2 - d_1(e_2) e_1$. Indeed, we have that $\mb{L}_1 = F_3 \otimes \bigwedge^2 F_1^*$, which is exactly the non-uniqueness witnessed here.
		
		Now suppose that $w\colon R_\mathrm{gen} \to R$ (equivalently, $R \otimes R_\mathrm{gen} \to R$) is one choice of higher structure maps for $\mb{F}$, and take an element $X = \sum_{i > 0} u_i \in \mb{L} \otimes R$ using the same notation as before. Let $w' = w\exp(X)$, i.e.
		\[
		w' = w\left(1 + u_1 + \left(\frac{1}{2}u_1^2 + u_2\right) + \cdots\right)
		\]
		Note that $u_k$ maps $W(d_i)_{j}$ to $W(d_i)_{j-k}$. If we restrict the above equation to the representation $W(d_3)$ and expand it degree-wise, we get
		\begin{align*}
			w'^{(3)}_0 &= w^{(3)}_0\\
			w'^{(3)}_1 &= w^{(3)}_1 + w^{(3)}_0 u_1\\
			w'^{(3)}_2 &= w^{(3)}_2 + w^{(3)}_1 u_1 + w^{(3)}_0 \left(\frac{1}{2}u_1^2 + u_2\right)\\
			\vdots
		\end{align*}
		The first equation reflects that the underlying complex is still the same $\mb{F}$. The next equation shows that the new multiplication, viewed as a map $F_2^* \otimes \bigwedge^2 F_1 \to R$, was obtained from the old one by adding the composite
		\[
		F_2^* \otimes \bigwedge^2 F_1 \xto{1 \otimes u_1} F_2^* \otimes F_3 \xto{d_3} R.
		\]
		Here $u_1 \in \mb{L}_1 = F_3 \otimes \bigwedge^2 F_1^*$ could've been any map $\bigwedge^2 F_1 \to F_3$, and this exactly matches what we see in \eqref{eq:example-w31}.
	\end{example}
	We next establish some straightforward representation theory lemmas, which will be used to highlight the importance of a particular subspace in $W(a_3) \subset \Rgen$.
	\begin{lem}
		Let $b \in L(\omega_{z_1})^\vee$ be a lowest weight vector. The subspace
		\[
		V \coloneqq \{Xb : X \in \mf{g}\} \subseteq L(\omega_{z_1})^\vee
		\]
		is a representation of $\mf{n}_{z_1}^-$, and thus a $\mbf{L}$-representation.
	\end{lem}
	\begin{proof}
		Let $Y \in \mf{n}_{z_1}^-$. Then
		\[
		YX b = [Y,X]b + XYb = [Y,X]b
		\]
		since $Yb= 0$. The representation $L(\omega_{z_1})^\vee$ is a lowest weight representation, so for any $Y' \in \hat{\mf{n}}_{z_1}^-$, there exists a truncation $Y \in \mf{n}_{z_1}^-$ of $Y$ such that $Yb = Y'b$, so $\mbf{L}$ also acts on $V$.
	\end{proof}
	\begin{lem}
		The map
		\[
		\mb{C} \oplus \mb{L}^\vee \to V
		\]
		sending $1 \in \mb{C}$ to $b$ and $X \in \mb{L}^\vee = \mf{n}_{z_1}^+$ to $Xb$ is an isomorphism of vector spaces.
	\end{lem}
	\begin{proof}
		It is easy to see that the subspace
		\[
		\mf{p} \coloneqq \{ X \in \mf{g} : Xb \in \mb{C}b\} \subset \mf{g}
		\]
		is a subalgebra. Moreover it contains the maximal parabolic $\mf{p}_{z_1}^- = \bigoplus_{\alpha \leq_{z_1} 0} \mf{g}_\alpha$. Since $L(\omega_{z_1})^\vee$ is not the trivial representation, we have $\mf{p}_{z_1}^- \subseteq \mf{p} \subsetneq \mf{g}$. Therefore $\mf{p} = \mf{p}_{z_1}^-$, so $\mb{L}^\vee \cap \mf{p} = 0$ and the map is injective.
		
		Given any $X \in \mf{g}$, we may express $X$ as $X^+ + X^-$ where $X^+ \in \mf{n}_{z_1}^+$ and $X^- \in \mf{p}_{z_1}^-$. Then $Xb = X^+ b + X^- b$ where $X^- b \in \mb{C}b$, showing surjectivity.
	\end{proof}
	\begin{remark}\label{rem:p-in-a3}
		The action of $\mb{L}_i$ on $\mb{C} \oplus \mb{L}^\vee$ looks very similar to its action on the entries of $p_m$ as discussed around Theorem~\ref{thm:parametrize}. In fact, we expect that the restriction of $w\colon \Rgen \to R$ to
		\[
		\bigwedge^{r_2} F_2^* \otimes \bigwedge^{r_3} F_3 \otimes \mb{L}_m^* \subset \bigwedge^{r_2} F_2^* \otimes \bigwedge^{r_3} F_3 \otimes [\mb{C} \oplus \mb{L}^\vee] \subseteq W(a_3)
		\]
		should exactly recover the structure map $p_m$. While this should not be too difficult to prove, it would be somewhat technical and unnecessary for our purposes, so we leave it as a guess. Some of the results below should (in principle) be consequences of this statement, but of course we do not assume this statement in their proofs.
	\end{remark}
	In the following, we will make some arguments using regular sequences. If the rings involved are not Noetherian, then it may be necessary to adjoin variables in view of Definition~\ref{def:true-grade}, but this has no effect on the conclusions.
	
	\begin{lem}\label{lem:V-exp-props}
		Let $h \in R$ be a nonzerodivisor, let $\pi \colon [\mb{C} \oplus \mb{L}^\vee] \otimes R \twoheadrightarrow \mb{C} \otimes R$ be projection onto the first factor, and let $\gamma\colon \mb{L}^\vee \otimes R \to \mb{C} \otimes R$ be any map. Then
		\begin{itemize}
			\item there is a unique $X \in \mbf{L} \cotimes R_h$ such that $(h\pi+\gamma) = h\pi \exp X$, and
			\item if $S$ is a ring containing $R$ and $X' \in \mbf{L} \cotimes S$ satisfies $(h\pi+\gamma) = h\pi \exp X'$, then $X'$ must belong to $\mbf{L} \cotimes R_h$ and thus equal $X$.
		\end{itemize}
	\end{lem}
	\begin{proof}
		For $X \in \mbf{L} \cotimes S$, we have $\pi X = 0$ only when $X = 0$. Thus, the precomposition action of $\exp (\mbf{L}\cotimes S)$ on $\pi$ has trivial stabilizer, showing uniqueness. One can solve for $X$ explicitly in a manner similar to the proof of Theorem~\ref{thm:parametrize}, showing it must be an element of $\mbf{L}\cotimes R_h$. One should informally think of $X$ as $\log (\pi + \gamma/h)$; we omit the details.
	\end{proof}
	
	The next result says that we do not lose any information by only considering the structure maps $w^{(i)}$, since they uniquely determine $w$. In fact, we only need the differentials and part of $w^{(a_3)}$, which we recall can be computed from $w^{(3)}$ (c.f. Remark~\ref{rem:a3-in-d3}).
	
	\begin{prop}\label{prop:HST-determined-by-a3}
		Let $w$ and $w'$ be two maps $\Rgen \to R$ specializing $\Fgen$ to the same resolution $\mb{F}$ over some ring $R$. Viewing them as maps $\Rgen \otimes R \to R$, write $\bar{w},\bar{w}'$ for their restrictions to
		\[
		\bigwedge^{r_2} F_2^* \otimes \bigwedge^{r_3} F_3 \otimes [\mb{C} \oplus \mb{L}^\vee] \otimes R \subseteq W(a_3) \otimes R.
		\]
		Then there is a unique element $X \in \mbf{L} \cotimes R$ such that $\bar{w}' = \bar{w}\exp X$.
		
		In particular, if $\bar{w} = \bar{w}'$, then $X = 0$ and $w = w'$.
	\end{prop}
	The existence of such an element $X$ is already known by Theorem~\ref{thm:parametrize}; the substance of this statement is that $X$ is completely determined by comparing $\bar{w}$ and $\bar{w}'$. If Remark~\ref{rem:p-in-a3} were true, this would be immediate given Lemma~\ref{lem:GFR-p-determines-w}.
	\begin{proof}
		Since $\mb{F}$ is acyclic, $\grade I_{r_3}(d_3) = 3 \geq 1$ so there is some $e \in \bigwedge^{r_3}F_3\otimes \bigwedge^{r_2}F_2^*$ such that $h = a_3(e) \in R$ is a nonzerodivisor. Let $\bar{w}_e$ and $\bar{w}'_e$ denote the restrictions of $\bar{w}$ and $\bar{w}'$ to $\mb{C}e \otimes [\mb{C} \oplus \mb{L}^\vee]$, viewed as maps
		\[
		[\mb{C} \oplus \mb{L}^\vee] \otimes R \to (\mb{C}e)^* \otimes R \cong R
		\]
		From Theorem~\ref{thm:parametrize}, we know there exists $X \in \mbf{L} \cotimes R$ with the property that $\bar{w}_e = \bar{w}_e' \exp X$. The uniqueness follows from Lemma~\ref{lem:V-exp-props}.
	\end{proof}
	
	It will often be convenient to manipulate higher structure maps $w^{(i)}$ over a larger ring containing $R$, for instance a localization in which $\mb{F}$ becomes split exact. The next result ensures that, even if we work in a larger ring, it is easy to tell when $w$ factors through $R$.
	
	\begin{prop}\label{prop:localization-lemma}
		Let $R \subset S$ be two rings, $\mb{F}$ a resolution over $R$ with the property that $\mb{F}\otimes S$ is also acyclic, and $w' \colon \Rgen \to S$ a map specializing $\Fgen$ to $\mb{F} \otimes S$.
		
		Suppose the restriction of $w'$ to
		\[
		\bigwedge^{r_2} F_2^* \otimes \bigwedge^{r_3} F_3 \otimes [\mb{C} \oplus \mb{L}^\vee] \subseteq W(a_3)
		\]
		is $R$-valued. Then $w'$ factors through $R$.
	\end{prop}
	\begin{proof}
		This proof is similar to the preceding one. This time we use that $\grade I_{r_3}(d_3) = 3 \geq 2$ so there are $e_1,e_2 \in \bigwedge^{r_3}F_3\otimes \bigwedge^{r_2}F_2^*$ such that $h_1 = a_3(e_1)$ and $h_2 = a_3(e_2)$ form a regular sequence.
		
		Since $\mb{F}$ is acyclic, we may pick a $w\colon \Rgen \to R$ specializing $\Fgen$ to $\mb{F}$. By Theorem~\ref{thm:parametrize}, there exists $X \in \mbf{L} \cotimes S$ such that $w' = (w\otimes S)\exp X$.
		
		Let $\bar{w}_e$ and $\bar{w}_e'$ denote the restrictions of $w \otimes S$ and $w'$ to $(\mb{C}e_1 \oplus \mb{C}e_2) \otimes [\mb{C} \oplus \mb{L}^\vee] \subset W(a_3)$, viewed as maps
		\[
		[\mb{C} \oplus \mb{L}^\vee] \otimes S \to (\mb{C}e_1 \oplus \mb{C}e_2)^* \otimes S \cong S^2.
		\]
		We have $\bar{w}_e' = \bar{w}_e \exp X$. Applying Lemma~\ref{lem:V-exp-props} to the first row of $\bar{w}_e$ and $\bar{w}_e'$, we find that $X \in \mbf{L} \cotimes R_{h_1}$. Applying it to the second row, we find that $X \in \mbf{L} \cotimes R_{h_2}$. Since $h_1,h_2$ is a regular sequence, we have $R_{h_1} \cap R_{h_2} = R$ and $X \in \mbf{L} \cotimes R$, from which it follows that $w'$ factors through $R$ if viewed as a map from $\Rgen$.
	\end{proof}
	
	\subsubsection{Computing particular higher structure maps}
	So far we have studied relationships between different choices of $w\colon \Rgen \to R$ specializing $\Fgen$ to a given $\mb{F}$, but we have not discussed how to effectively compute a particular choice of $w$ given a resolution $\mb{F}$ in the first place.
	
	In general this is a difficult problem, but with the presence of some symmetry, Proposition~\ref{prop:equivariant-p} can often simplify the task, or at least allow us to deduce qualitative properties about particularly nice choices of $w$.
	
	\begin{example}\label{ex:max-ideal-1683}
		Let $R = \mb{C}[t_1,t_2,t_3]$ with the standard $\mb{Z}$-grading and let $I = (t_1,t_2,t_3)^2$. The minimal graded free resolution of $R/I$ is
		\[
		\mb{F} \colon 0\to R^3(-4) \to R^8(-3) \to R^6(-2) \to R.
		\]
		As per Proposition~\ref{prop:equivariant-p}, there is a choice of $w\colon\Rgen \otimes R\to R$ specializing $\Fgen$ to $\mb{F}$ that is homogeneous of degree zero. We identify $M\cong R$ as in Remark~\ref{rem:gl-vs-sl}. The structure maps have the form
		\begin{align*}
			w^{(1)} \colon [F_1 \oplus \bigwedge^3 F_1 \oplus \cdots]\otimes R &\to R\\
			w^{(2)} \colon [F_1^* \oplus F_1 \otimes F_3^* \oplus \cdots] \otimes R &\to F_2^* \otimes R\\
			w^{(3)} \colon [F_3 \oplus \bigwedge^2 F_1 \oplus \cdots]\otimes R &\to F_2 \otimes R
		\end{align*}
		In this example $\mb{L}_1 = \bigwedge^2 F_1^* \otimes F_3$ is concentrated in degree zero, causing all generators of each representation $L(\omega)^\vee$ to be in the same degree. So with a homogeneous choice of $w$, all entries of $w^{(1)}$ have degree 2, whereas all entries of $w^{(2)}$ and $w^{(3)}$ are linear. 
	\end{example}
	
	\begin{example}\label{ex:d3-split}
		Let $\mb{F}$ be a resolution where $d_3$ is a split inclusion. After a change of coordinates, we assume it has the form
		\[
		0 \to F_3 \otimes R \to (F_3 \oplus Z)\otimes R \to F_1 \otimes R \to F_0 \otimes R
		\]
		where $Z = \mb{C}^{r_2}$ and $d_3$ maps $F_3 \otimes R$ identically to itself.
		
		There is an action of $G = \GL(F_3 \otimes R)$ on this resolution, and the differentials are equivariant with respect to it. The Koszul complex on $\bigwedge^{r_3}(d_3)$ is just a split exact complex, so we can certainly pick lifts $p_m$ which are $G$-equivariant, e.g. by using a $G$-equivariant splitting.
		
		Let $w \colon \Rgen \otimes R \to R$ be the $G$-equivariant map obtained in this manner. The maps $w^{(i)}$ have the form
		\begin{align*}
			[F_3 \oplus \bigwedge^{f_0+1} F_1 \otimes M^* \oplus \cdots] \otimes R &\to (F_3 \oplus Z) \otimes R\\
			[F_1^* \oplus F_3^* \otimes \bigwedge^{f_0}F_1 \otimes M^* \oplus \cdots]\otimes R &\to (F_3^* \oplus Z^*) \otimes R\\
			[F_1 \oplus F_3^* \otimes \bigwedge^{f_0+2} F_1 \otimes M^* \oplus \cdots] \otimes R &\to F_0 \otimes R
		\end{align*}
		Note that $M = \bigwedge^{f_3} F_3 \otimes \bigwedge^{f_2}(F_3^* \oplus Z^*) \otimes \bigwedge^{f_1} F_1$ is a trivial representation of $G$. Since $\mb{L}_1^* = \bigwedge^2 F_1 \otimes F_3^* \otimes M^*$, we accumulate an additional factor of $F_3^*$ every time we go up in the $z_1$-grading in each representation. Thus by $G$-equivariance considerations the only components $w^{(i)}_j$ that have any chance of being nonzero are:
		\begin{align*}
			w^{(3)}_1 \colon \bigwedge^{f_0+1}F_1 \otimes M^* \otimes R &\to Z \otimes R \subset (F_3 \oplus Z )\otimes R\\
			w^{(2)}_1 \colon \bigwedge^{f_0} F_1 \otimes F_3^* \otimes M^* \otimes R&\to F_3^* \otimes R \subset (F_3^* \oplus Z^*) \otimes R
		\end{align*}
		in addition to the maps $w^{(i)}_0$ which are obviously nonzero since they give the differentials of the resolution.
		
		We also note that for the map $w^{(a_2)}$
		\[
		[\bigwedge^{r_2} F_1^* \otimes \bigwedge^{f_3} F_3^* \oplus \cdots] \otimes R \to \bigwedge^{r_2} Z^* \otimes \bigwedge^{r_3} F_3^* \otimes R
		\]
		only the bottom component, namely $a_2$ itself, is nonzero by the same considerations.
	\end{example}
	Hence it would be beneficial to at least understand how to compute $w^{(3)}_1$ and $w^{(2)}_1$ explicitly. In \cite[Prop. 7.1]{Lee-Weyman19}, it is described\footnote{In the referenced paper, it was assumed that $a_1\colon M \to \bigwedge^{f_0} F_0^*$ is an isomorphism.} how to compute these maps via a comparison map from a Buchsbaum-Rim complex, which we now recall. We write $F_i$ to mean $F_i \otimes R$ in the following. Theorem~\ref{thm:BE1'} gives a factorization
	\[
	\begin{tikzcd}
		\bigwedge^{r_1} F_1 \ar[rr, "\bigwedge^{r_1}d_1"]\ar[dr] &&\bigwedge^{r_1} F_0\\
		& M \ar[ur, "a_1", swap]
	\end{tikzcd}
	\]
	in particular a map $\beta\colon M^* \otimes \bigwedge^{r_1}F_1 \to R$, which is essentially $a_2^*$ after appropriate identifications. It is straightforward to check that the composite
	\[
	M^* \otimes \bigwedge^{r_1+1}F_1 \to M^* \otimes \bigwedge^{r_1} F_1 \otimes F_1 \xto{\beta \otimes 1} F_1 \xto{d_1} F_0
	\]
	is zero, thus we can lift through $d_2$ to obtain a map
	\[
	w^{(3)}_1 \colon M^*\otimes \bigwedge^{r_1+1}F_1 \to F_2.
	\]
	The difference of the two maps
	\begin{gather*}
		M^* \otimes \bigwedge^{r_1}F_1 \otimes F_2 \xto{\beta \otimes 1} F_2 \\
		M^* \otimes \bigwedge^{r_1}F_1 \otimes F_2 \xto{1 \otimes d_2} M^* \otimes \bigwedge^{r_1}F_1 \otimes F_1 \to M^* \otimes \bigwedge^{r_1+1} F_1 \xto{w^{(3)}_1} F_2
	\end{gather*}
	has image landing in $\ker d_2$, and thus it can be lifted through $d_3$ to obtain
	\[
	w^{(2)}_1 \colon M^* \otimes \bigwedge^{r_1}F_1 \otimes F_2 \to F_3.
	\]
	In the case that $r_0 = 1$, these maps can be viewed as giving a choice of multiplication on the resolution
	\[
	0 \to M^* \otimes F_3 \to M^* \otimes F_2 \to M^* \otimes F_1 \xto{\beta} R
	\]
	recovering what was illustrated in Example~\ref{ex:multiplication}.
	
	A very important special case of Example~\ref{ex:d3-split} is when the entire complex $\mb{F}$ is split exact, e.g. we take $\mb{F}$ to be the split exact complex
	\[
	\mb{F}^\ssc \colon 0 \to F_3 \to F_3 \oplus Z \to F_0 \oplus Z \to F_0
	\]
	of $\mb{C}$-vector spaces. Here $M = \bigwedge^{f_0} F_0$, and a direct computation shows that there is a unique $G = \GL(F_0) \times \GL(F_3) \times \GL(Z)$-equivariant choice of $w$. For this choice of $w$, direct computation with the explicit definitions of $w^{(3)}_1$ and $w^{(2)}_1$ above shows that
	\[
	w^{(3)}_1 \colon \bigwedge^{f_0+1} F_1 \otimes \bigwedge^{f_0} F_0^* = Z \oplus F_0^* \otimes \bigwedge^2 Z \oplus \cdots  \to F_3 \oplus Z
	\]
	maps $Z$ identically to itself and is zero on all other factors by $G$-equivariance. Similarly
	\[
	w^{(2)}_1 \colon \bigwedge^{f_0} F_1 \otimes F_3^* \otimes \bigwedge^{f_0} F_0^* = (\mb{C} \oplus F_0^* \otimes Z \oplus \bigwedge^2 F_0^* \otimes \bigwedge^2 Z \oplus \cdots)\otimes F_3^* \to F_3^* \oplus Z^* = F_2^*
	\]
	maps $F_3^*$ identically to itself and is zero on all other factors.
	
	\begin{remark}\label{rem:split-exact-revisited}
		This complex $\mb{F}^\ssc$ already appeared in Example~\ref{ex:split-exact-complex}. In fact, the higher structure maps $w_\ssc^{(i)}$ are implicit in that example, dual to the inclusions
		\begin{align*}
			F_0^* &\hookrightarrow L(\omega_{x_{r_1}})\\
			F_2 &\hookrightarrow L(\omega_{y_{r_2-2}})\\
			F_2^* &\hookrightarrow L(\omega_{z_{r_3}}).
		\end{align*}
		After we project onto $F_1^*$, $F_1$, and $F_3^*$ (which is dual to the inclusion of the bottom $z_1$-graded components), we get $(w^{(1)}_0)^* = d_1^*$, $(w^{(2)}_0)^* = d_2$, and $(w^{(3)}_0)^* = d_3^*$.
	\end{remark}
	Here is an equivalent restatement of the preceding remark. View $\mf{sl}(F_0) \times \mf{sl}(F_2)$ as the subalgebra $\mf{g}^{(x_1)}$ of $\mf{g}$ (c.f. \S\ref{sec:gen-licci-differentials}).
	\begin{thm}\label{thm:ssc}
		There exists a $\mb{C}$-algebra homomorphism $w_\mathrm{ssc}\colon \Rgen \to \mb{C}$ so that
		\begin{align*}
			w_\ssc^{(1)} \colon L(\omega_{x_{r_1}}) &\twoheadrightarrow F_0\\
			w_\ssc^{(2)} \colon L(\omega_{y_{r_2-2}}) &\twoheadrightarrow F_2^*\\
			w_\ssc^{(3)} \colon L(\omega_{z_{r_3}}) &\twoheadrightarrow F_2
		\end{align*}
		are given by projection onto the bottom $x_1$-graded component. Note that this determines $w_\ssc$ completely by Proposition~\ref{prop:HST-determined-by-a3}. Furthermore, $w^{(a_2)}$ is also projection onto its bottom $x_1$-graded component, which is its lowest weight space.
		
		The map $w_\ssc$ specializes $\Fgen$ to the \emph{standard split complex} of $\mb{C}$-vector spaces
		\[
		\mb{F}^\ssc \colon 0 \to F_3 \to F_2 \oplus Z \to F_0 \oplus Z \to F_0
		\]
		where $Z = \mb{C}^{r_2}$.
	\end{thm}
	
	The comment regarding $w^{(a_2)}$ in this theorem implies the following.
	\begin{cor}\label{cor:GP-homogeneous-ring-wa2}
		Let $w\colon \Rgen \to R$ specialize $\Fgen$ to a resolution $\mb{F}$. Then there is a unique ring homomorphism
		\[
		\bigoplus_{n \geq 0} L(n\omega_{x_1})^\vee \to R
		\]
		equal to $w^{(a_2)}$ in degree 1. Here the source is the homogeneous coordinate ring of $G/P = G/P_{x_1}$; c.f. Lemma~\ref{lem:extremal-coords-on-X}.
		
		If $w^{(a_2)}$ is surjective, then the above determines a map
		\[
		\Spec R \to G/P_{x_1} \subset \mb{P}(\widehat{L}(\omega_{x_1})).
		\]
		This map lands in the complement of $X^{s_{z_1} s_u s_{x_1}}$ if and only if $a_2$ generates the unit ideal.
	\end{cor}
	\begin{proof}
		The homogeneous coordinate ring of $G/P$ is generated in degree 1 so the uniqueness is clear; we need to check that the map is well-defined. As noted in Theorem~\ref{thm:ssc}, this is certainly true for $w = w_\ssc$, so the result for split $\mb{F}$ follows from Theorem~\ref{thm:parametrize} since $\GL(F_i)$ and $\exp \mbf{L}$ both act on the homogeneous coordinate ring of $G/P$.
		
		The ring $\bigoplus_{n \geq 0} L(n\omega_{x_1})^\vee$ is a quotient of $\Sym L(n\omega_{x_1})^\vee$ by Pl\"ucker relations. An arbitrary $\mb{F}$ is split after localization, and relations which hold over the localization must also hold over $R$.
		
		The other statement about $X^{s_{z_1} s_u s_{x_1}}$ follows from Lemma~\ref{lem:extremal-coords-on-X}.
	\end{proof}
	
	\subsubsection{Addition of a split part}\label{subsec:format-expansion}
	Let $\mb{F}$ be a resolution of format $\underline{f}$, and let $\mb{G}$ be a split exact complex. In this section, we study how one can deduce a choice of higher structure maps for $\mb{F}\oplus \mb{G}$ starting from a choice of higher structure maps for $\mb{F}$. Although this may seem like a peculiar question to consider, it is significant for a few reasons:
	\begin{itemize}
		\item We would like to use the theory of higher structure maps to define things which are intrinsic to the module resolved by $\mb{F}$. Currently Theorem~\ref{thm:parametrize} can only compare higher structure maps for resolutions of the same format.
		\item Later, we will study how higher structure maps evolve under linkage of perfect ideals as an extension of Theorem~\ref{thm:mapping-cone}. The Betti numbers can increase or decrease during this process, so it is mandatory to know how to adjust higher structure maps appropriately.
	\end{itemize}
	We will find that the answer, although simple to state and prove, is surprisingly subtle. In addition to the parameters $r_i$, $f_i$ already fixed, let $n_1,n_2,n_3 \geq 0$ be integers and let $N_i = \mb{C}^{n_i}$. Let $\mb{F}$ be a resolution of format $\underline{f}$ and let $\mb{F}'$ denote its direct sum with the split exact complex
	\[
	0 \to N_3 \otimes R \to (N_3 \oplus N_2)\otimes R \to (N_2 \oplus N_1)\otimes R \to N_1 \otimes R.
	\]
	Write $\underline{f}' = (f_0+n_1,f_1+n_1+n_2,f_2+n_2+n_3,f_3+n_3)$ for the format of $\mb{F}'$ and let $T'$ be the corresponding enlarged diagram
	\[
	\begin{tikzcd}[column sep=5pt,row sep=5pt]
		x_{r_1 + n_1} \ar[r,dash] & \cdots \ar[r,dash] & x_{r_1+1} \ar[r,dash] & x_{r_1} \ar[r,dash] & \cdots \ar[r,dash] & \colorX{x_1} \ar[r,dash] & u\ar[r,dash]\ar[d,dash] & y_1\ar[r,dash] & \ar[r,dash] \cdots\ar[r,dash] & y_{r_2-2} \ar[r,dash]& y_{r_2-1} \ar[r,dash]& \cdots\ar[r,dash] & y_{r_2+n_2-2}\\
		&&&&&& \colorZ{z_1}\ar[d,dash]\\
		&&&&&& \vdots\ar[d,dash]\\
		&&&&&& z_{r_3}\ar[d,dash]\\
		&&&&&& z_{r_3+1}\ar[d,dash]\\
		&&&&&& \vdots\ar[d,dash]\\
		&&&&&& z_{r_3+n_3}
	\end{tikzcd}
	\]
	which contains $T$ as a subdiagram. Let $\mf{g}'$ denote the Kac-Moody Lie algebra associated to $T'$; by deleting the vertices $x_{r_1+1}$, $y_{r_2-1}$, and $z_{r_3+1}$ we observe
	\[
	\mf{g}'^{(x_{r_1+1},y_{r_2-1},z_{r_3+1})} = \mf{g} \times \mf{sl}(N_1) \times \mf{sl}(N_2) \times \mf{sl}(N_3).
	\]
	For brevity we will call this $\mf{g}'_0$. In the following, $L(\omega,\mf{g}')$ denotes the irreducible highest weight representation of $\mf{g}'$ with weight $\omega$. We continue to use $L(\omega)$ to denote representations of $\mf{g}$. The bottom $x_1$-graded components of the three extremal representations of $\mf{g}'$ are:
	\begin{align*}
		F_0' = F_0 \oplus N_1 &\subset L(\omega_{x_{r_1+n_1}},\mf{g}')^\vee\\
		F_2'^* = F_2^* \oplus N_2^* \oplus N_3^* &\subset L(\omega_{y_{r_2+n_2-2}},\mf{g}')^\vee\\
		F_2' = F_2 \oplus N_2 \oplus N_3 &\subset L(\omega_{z_{r_3+n_3}},\mf{g}')^\vee.
	\end{align*}
	By examining the weights, we can see which representations of $\mf{g}'_0$ these components belong to, and these are all extremal:
	\begin{equation}\label{eq:GFR-bottom-after-expansion}
		\begin{split}
			F_0 \oplus N_1 &\subset L(\omega_{x_{r_1}})^\vee \oplus N_1\\
			F_2^* \oplus N_2^* \oplus N_3^* &\subset L(\omega_{y_{r_2-2}})^\vee \oplus N_2^* \oplus (N_3^* \otimes L(\omega_{x_1})^\vee)\\
			F_2 \oplus N_2 \oplus N_3 &\subset L(\omega_{z_{r_3}})^\vee \oplus (N_2 \otimes L(\omega_{x_1})^\vee) \oplus N_3.
		\end{split}
	\end{equation}
	For the upcoming Definition~\ref{def:I-lambda-mu} and its later uses, it will be helpful to adjust from \eqref{eq:crit-reps} and write $W(a_2)$ as
	\[
	W(a_2) = M^* \otimes L(\omega_{x_1})^\vee = \bigwedge^{f_1} F_1^* \otimes \bigwedge^{f_2} F_2 \otimes \bigwedge^{f_3} F_3^* \otimes [\bigwedge^{r_1} F_1 \oplus \cdots]
	\]
	i.e. we move a factor of $\bigwedge^{f_3} F_3^* \otimes \bigwedge^{f_1} F_1^*$ to the other side and think of $w^{(a_2)}$ as a map
	\[
	[\bigwedge^{r_1} F_1 \oplus \cdots] \otimes R \to \bigwedge^{f_1} F_1 \otimes \bigwedge^{f_2} F_2^* \otimes \bigwedge^{f_3} F_3 \otimes R = M \otimes R.
	\]
	Unlike the three extremal representations considered in \eqref{eq:GFR-bottom-after-expansion}, the one-dimensional bottom $x_1$-graded component of $L(\omega_{x_1}, \mf{g}')^\vee$ is entirely contained in a single $\mf{g}'_0$-representation, namely\footnote{The $\bigwedge^{n_1} N_1$ factor is to keep everything $\prod \GL(F_i) \times \prod \GL(N_i)$-equivariant.} $L(\omega_{x_1})^\vee \otimes \bigwedge^{n_1} N_1$. Analogously to \eqref{eq:GFR-bottom-after-expansion} we have
	\begin{equation}\label{eq:GFR-bottom-after-expansion'}
		\bigwedge^{f_1} F_1 \otimes \bigwedge^{f_2} F_2^* \otimes \bigwedge^{f_3} F_3 \otimes \bigwedge^{n_1} N_1 \subset L(\omega_{x_1})^\vee \otimes \bigwedge^{n_1} N_1 \subset L(\omega_{x_1},\mf{g}')^\vee.
	\end{equation}
	\begin{thm}\label{thm:format-expansion}
		Let $w\colon \Rgen \to R$ specialize $\Fgen$ to $\mb{F}$. Let $(\Rgen',{\Fgen}')$ denote the generic pair associated to the format $\underline{f}'$. Then there is a $w'\colon \Rgen'\to R$ specializing ${\Fgen}'$ to $\mb{F}'$ such that:
		\begin{itemize}
			\item $w'^{(1)}$ is the composite
			\[
			\begin{tikzcd}[ampersand replacement=\&,outer sep=3pt]
				L(\omega_{x_{r_1+n_1}},\mf{g}')^\vee \otimes R\ar[d,two heads]\\
				(L(\omega_{x_{r_1}})^\vee \oplus N_1)\otimes R \ar[d,"{(w^{(1)},\, \Id)}"]\\
				(F_0 \oplus N_1)\otimes R
			\end{tikzcd}
			\]
			\item $w'^{(2)}$ is the composite
			\[
			\begin{tikzcd}[ampersand replacement=\&,outer sep=3pt]
				L(\omega_{y_{r_2+n_2-2}},\mf{g}')^\vee \otimes R\ar[d,two heads]\\
				(L(\omega_{y_{r_2-2}})^\vee \oplus N_2^* \oplus (N_3^* \otimes L(\omega_{x_1})^\vee))\otimes R \ar[d,"{(w^{(2)} ,\, \Id ,\, \Id \otimes w^{(a_2)})}"]\\
				(F_2^* \oplus N_2^* \oplus N_3^*) \otimes R
			\end{tikzcd}
			\]
			\item $w'^{(3)}$ is the composite
			\[
			\begin{tikzcd}[ampersand replacement=\&,outer sep=3pt]
				L(\omega_{z_{r_3+n_3}},\mf{g}')^\vee \otimes R\ar[d,two heads]\\
				(L(\omega_{z_{r_3}})^\vee \oplus (N_2 \otimes L(\omega_{x_1})^\vee) \oplus N_3) \otimes R \ar[d,"{(w^{(2)} ,\, \Id\otimes w^{(a_2)} ,\, \Id) }"]\\
				(F_2 \oplus N_2 \oplus N_3) \otimes R
			\end{tikzcd}
			\]
			\item $w'^{(a_2)}$ is the composite
			\[
			\begin{tikzcd}[ampersand replacement=\&,outer sep=3pt]
				L(\omega_{x_1},\mf{g}')^\vee \otimes R\ar[d,two heads]\\
				L(\omega_{x_1})^\vee \otimes \bigwedge^{n_1} N_1 \otimes R \ar[d,"w^{(a_2)} \otimes \Id"]\\
				\bigwedge^{f_1} F_1 \otimes \bigwedge^{f_2} F_2^* \otimes \bigwedge^{f_3} F_3 \otimes \bigwedge^{n_1}N_1 \otimes R
			\end{tikzcd}
			\]
		\end{itemize}
		Each map denoted with a $\twoheadrightarrow$ is given by projection onto the $\mf{g}'_0$-representations identified in \eqref{eq:GFR-bottom-after-expansion} and \eqref{eq:GFR-bottom-after-expansion'}.
	\end{thm}
	We include $w'^{(a_2)}$ not because it is necessary to describe $w'$, but to point out that it remains essentially unchanged---only its source has been enlarged.
	\begin{proof}
		Since the proposed construction of $w'$ from $w$ is $\prod \GL(F_i)$-equivariant and $\mf{g}'_0$-equivariant, it suffices to prove the statement for $w_\ssc$. If $\mb{F}$ is arbitrary, then after localization it is isomorphic to a split exact complex, thus the result would follow from Theorem~\ref{thm:parametrize} and Proposition~\ref{prop:localization-lemma}.
		
		But for $w = w_\ssc$, this statement is immediate given Theorem~\ref{thm:ssc}, since by \eqref{eq:GFR-bottom-after-expansion} this construction of $w'$ simply yields $w'_\ssc\colon \Rgen' \to \mb{C}$ for the larger format $\underline{f}'$.
	\end{proof}
	
	Suppose that $\mb{F}$ resolves an $R$-module $B$. We are now finally ready to make some definitions which are intrinsic to $B$.
	
	\begin{lem}\label{lem:well-def-pt1}
		Let $R$ be a ring and let $U \subset \Rgen$ be any subspace that is closed under the actions of $\prod \GL(F_i)$ and $\mbf{L}$. Let $B$ be an $R$-module and suppose $w\colon \Rgen \to R$ specializes $\Fgen$ to $\mb{F}$ resolving $B$. Then the ideal $w(U)R$ depends only on $B$ and not on the choice of $w$.
	\end{lem}
	\begin{proof}
		Let $w'\colon \Rgen \to R$ be another map specializing $\Fgen$ to a resolution $\mb{F}'$ of $B$. To show that $w(U)R = w'(U)R$, it suffices to check after localizing at each prime of $R$, so we reduce at once to the case that $R$ is local.
		
		In this situation, the resolutions $\mb{F}$ and $\mb{F}'$ of $B$ must be isomorphic, hence related by the action of $\prod GL(F_i \otimes R)$. Different choices of $w\colon \Rgen \otimes R \to R$ specializing $\Fgen$ to a fixed $\mb{F}$ are related by $\exp(\mbf{L}\cotimes R)$ by Theorem~\ref{thm:parametrize}. As $U \otimes R$ is closed under both of these actions, the result follows.
	\end{proof}
	Note that in this lemma, we fix the format of $\mb{F}$. With the aid of Theorem~\ref{thm:format-expansion}, we can improve this, but first we make some definitions.
	\begin{definition}\label{def:I-lambda-mu}
		Let $V = S_\lambda F_1 \otimes S_\mu F_3^*$ be an irreducible representation in the $z_1$-graded decomposition of
		\[
		L(\omega_{x_1})^\vee = \bigwedge^{r_1} F_1 \oplus \cdots.
		\]
		The \emph{degree} of $V$ is $|\mu|$, and by our description of $\mb{L}_1^*$ we necessarily have $|\lambda| = r_1 + |\mu|(r_1+1)$. This degree gives the $z_1$-graded component in which $V$ appears, where we index so that $\bigwedge^{r_1} F_1$ is in degree 0. If $V$ is moreover an extremal representation, then it has multiplicity 1 inside of $L(\omega_{x_1})^\vee$, so it is well-defined to write $\mb{I}_{\leq (\lambda,\mu)}$ for the $\mbf{L}$-representation generated by
		\[
		\bigwedge^{f_2}F_2^* \otimes \bigwedge^{f_3}F_3 \otimes V \subset \bigwedge^{f_2}F_2^* \otimes \bigwedge^{f_3}F_3 \otimes L(\omega_{x_1})^\vee = W(a_2)
		\]
		in $\Rgen$, or equivalently the Demazure module generated by a highest weight vector of $V$.
	\end{definition}
	Let $b \in L(\omega_{x_1})^\vee$ be a lowest weight vector for $\mf{g}^{(z_1)}$. Then $b\in L(\omega_{x_1},\mf{g}')^\vee$ is also a lowest weight vector for $\mf{g}'^{(z_1)}$ in $L(\omega_{x_1},\mf{g}')^\vee$. Explicitly, suppose $b$ is a lowest weight vector for
	\[
	V = S_\lambda F_1 \otimes S_\mu F_3^* \subset L(\omega_{x_1})^\vee.
	\]
	If $\lambda = (\lambda_1,\ldots,\lambda_{f_1})$, define
	\[
	\lambda' = (\underset{\text{$n_1$ times}}{\underbrace{1+|\mu|,\ldots,1+|\mu|}},\lambda_1,\lambda_2,\ldots,\lambda_{f_1}).
	\]
	Then an analysis of weights\footnote{In particular, this computation requires taking the suppressed $S_{|\mu|}M^*$ factor into account.} shows that $b$ is a lowest weight vector for
	\[
	S_{\lambda'} F_1' \otimes S_\mu F_3'^* \subset L(\omega_{x_1},\mf{g}')^\vee.
	\]
	\begin{remark}
		In \S\ref{sec:HST-linkage}, we will always fix $r_1 = 1$, so the adjustment to $\lambda$ will not be needed.
	\end{remark}
	If $b$ is extremal in $L(\omega_{x_1})^\vee$ then it is also extremal in $L(\omega_{x_1},\mf{g}')^\vee$. The extremal $\mf{g}^{(z_1)}$-representations in $L(\omega_{x_1})^\vee$ are indexed by $W_{z_1} \backslash W / W_{x_1}$: if $\sigma \in \leftindex^{z_1}W^{x_1}$ and $b \in L(\omega_{x_1})^\vee$ is a lowest weight vector (for $\mf{g}$), then $\sigma b$ is a lowest weight vector for an extremal $\mf{g}^{(z_1)}$-representation.
	\begin{remark}\label{rem:limit-double-coset}
		Combinatorially, if we let $W'$ denote the Weyl group for $\mf{g}'$, we have a natural inclusion
		\[
		W_{z_1} \backslash W / W_{x_1} \hookrightarrow W'_{z_1} \backslash W' / W'_{x_1}.
		\]
		Consequently we may consider the union of all of these sets as we allow the arms of $T$ to grow arbitrarily large:
		\[
		\lim_{r_1,r_2,r_3 \to \infty} W(T)_{P_{z_1}} \backslash W(T) / W(T)_{P_{x_1}}
		\]
		We will often treat $\sigma$ as a minimal length representative of some element in this limit. The smallest $T$ for which $[\sigma] \in W(T)_{P_{z_1}} \backslash W(T) / W(T)_{P_{x_1}}$ is just the smallest diagram that contains all the reflections needed to write a reduced word for $\sigma$.
	\end{remark}
	\begin{definition}\label{def:HSI}
		Let $R$ be a ring, and let $B$ be an $R$-module which admits a free resolution
		\[
		\mb{F} \colon 0\to F_3 \to F_2 \to F_1 \to F_0
		\]
		of some format $\underline{f}$ (with $f_i < \infty$). Let $(\Rgen,\Fgen)$ be the associated generic ring and choose a $w\colon \Rgen \to R$ specializing $\Fgen$ to $\mb{F}$.
		
		If $\sigma \in \leftindex^{z_1}W^{x_1}$ and $S_\lambda F_1 \otimes S_\mu F_3^*$ is the corresponding extremal representation, we define the \emph{higher structure ideal} $\HSI_\sigma(B) \coloneqq w(\mb{I}_{\leq(\lambda,\mu)})R$.
	\end{definition}
	\begin{prop}
		The ideals $\HSI_{\sigma}(B)$ are well-defined.
	\end{prop}
	\begin{proof}
		Let $\mb{F}$ and $\mb{F}'$ be two different free resolutions of $B$. By adding split complexes to both, we may assume that they have the same format. The crucial point is that Theorem~\ref{thm:format-expansion} guarantees $\HSI_{\sigma}(B)$ does not change in this process.
		
		Once they have the same format, the result follows from Lemma~\ref{lem:well-def-pt1}. The discussion preceding Definition~\ref{def:HSI} guarantees that the notation $\HSI_\sigma$ is unambiguous.
	\end{proof}
	\begin{example}
		Suppose $I$ is a grade 3 perfect ideal in a local Noetherian ring $(R,\mf{m})$. Choose $w\colon \Rgen\to R$ specializing $\Fgen$ to a minimal free resolution of $R/I$. Then the Buchsbaum-Eisenbud multiplier $a_1$ is an isomorphism, and we identify $W(a_2)$ with $W(d_1)$. For $e\in W$ the identity, $\HSI_e(R/I) = I$.
		
		If $I$ is not the unit ideal, then it is a complete intersection exactly when the multiplication $\bigwedge^3 F_1 \otimes R\to F_3 \otimes R$ is nonzero mod $\mf{m}$. This multiplicative structure appears in the higher structure maps as
		\[
		w^{(1)}_1 \colon \bigwedge^3 F_1 \otimes F_3^* \otimes R \to R
		\]
		where the extremal representation $\bigwedge^3 F_1 \otimes F_3^* \subset L(\omega_{x_1})^\vee$ corresponds to $\sigma = s_{z_1} s_u s_{x_1}$. Hence we see that $\HSI_\sigma(R/I)$ cuts out the \emph{non-c.i. locus}, in the sense that for $\mf{p} \in \Spec R$, $\HSI_\sigma(R/I) \subseteq \mf{p}$ if and only if $I_\mf{p} \subseteq R_\mf{p}$ is a complete intersection (or the unit ideal).
		
		In the next section, we will see that the sum $\sum_\sigma \HSI_\sigma(R/I)$, i.e. the image of $w^{(1)}$, is the \emph{non-licci locus}.
	\end{example}
	\begin{example}
		For $I = (t_1,t_2,t_3)^2 \subset \mb{C}[t_1,t_2,t_3]$ from Example~\ref{ex:max-ideal-1683}, we see that
		\[
		I = \HSI_e(R/I) \subseteq \HSI_\sigma(R/I) \subseteq I
		\]
		since all entries of $w^{(1)}$ have degree 2. Thus $\HSI_\sigma(R/I) = I$ for all $\sigma$.
	\end{example}
	\begin{prop}\label{prop:classifying-cell}
		Let $B$ be a module over a local ring $(R,\mf{m},k)$ and suppose that $w^{(a_2)} \otimes k \neq 0$ for some choice of $w\colon \Rgen \to R$ specializing $\Fgen$ to a resolution of $B$.
		
		Then there is a unique $\sigma$ such that
		\[
		\HSI_\rho(B) = (1) \iff \rho \geq \sigma.
		\]
		Furthermore, there exists a choice of $w'\colon \Rgen\to R$ specializing $\Fgen$ to a resolution $\mb{F}'$ isomorphic to $\mb{F}$ such that $w'^{(a_2)}$ determines a map
		\[
		\Spec R \to C_\sigma \subset G/P_{x_1}
		\]
		and $w'^{(a_2)} \otimes k$ gives the torus-fixed $k$-point of $C_\sigma$.
	\end{prop}
	\begin{proof}
		The first statement follows from knowing that $w^{(a_2)}$ determines a ring homomorphism
		\[
		\bigoplus_{n \geq 0} L(n\omega_{x_1})^\vee \to R
		\]
		from Corollary~\ref{cor:GP-homogeneous-ring-wa2}. The desired $\sigma$ corresponds to the lowest $\GL(F_1) \times \GL(F_3)$-representation $S_\lambda F_1 \otimes S_\mu F_3^*$ on which $w^{(a_2)} \otimes k \neq 0$, which is necessarily an extremal representation. Geometrically, this amounts to saying that the $k$-point of $G/P_{x_1}$ determined by $w^{(a_2)} \otimes k \neq 0$ lies in some $\GL(F_1) \times \GL(F_3)$-translate of the Schubert cell $C_\sigma$.
		
		The other claim is that we may adjust the choice of $w$ using the action of the parabolic $P_{z_1}^-$ so that it has the desired properties. This is a purely representation-theoretic lemma, but we include an explicit proof for completeness. For brevity, let $V = L(\omega_{\colorX{x_{1}}})^\vee$. There exists $g \in \GL(F_1 \otimes R) \times \GL(F_3 \otimes R)$ such that $(w^{(a_2)} g) \otimes k$ is nonzero on the lowest weight space $V_\omega$ of $S_\lambda F_1 \otimes S_\mu F_3^*$, which has weight $\omega = \sigma(-\omega_{x_1})$. Indeed, we may take $g$ to be a pair of permutation matrices. To ease notation slightly we write $\colorX{x}$ for $\colorX{x_{1}}$.
		
		The adjusted map $w^{(a_2)} g$ describes an $R$-point of $G/P$ which lands in the big open cell $\sigma C^e$. As such, there exists a
		\[
		Y \in \prod_{\sigma^{-1}\alpha <_\colorX{x} 0} (\mf{g}_\alpha \otimes R)
		\]
		such that $w^{(a_2)} g = \pi_\omega (\exp Y)$ where $\pi_\omega \colon V \otimes R \to R$ is a projection onto the extremal weight space $V_\omega$. Recall that $\sigma^{-1}\alpha <_\colorX{x} 0$ means that the coefficient of $\alpha_\colorX{x}$ in $\sigma^{-1}\alpha$ is negative, or equivalently that $\langle \sigma^{-1}\alpha, h_\colorX{x} \rangle < 0$.
		
		Furthermore, using Lemma~\ref{lem:BCH-split-up-positively-graded-Lie-algebra}, it is possible to solve for
		\[
		Y^- \in \prod_{\alpha < 0, \sigma^{-1}\alpha <_\colorX{x} 0} (\mf{g}_\alpha\otimes R), \qquad Y^+ \in \prod_{\alpha > 0, \sigma^{-1}\alpha <_\colorX{x} 0} (\mf{g}_\alpha \otimes R)
		\]
		with the property that $\exp Y = \exp(Y^+) \exp(Y^-)$.
		
		We then define $w'$ to be the map obtained by precomposing $w$ with the action of $g\exp(-Y^-)$ on $\Rgen$, so that
		\[
		w^{(a_2)} g\exp(-Y^-) = \pi_\omega (\exp Y)\exp(-Y^-) = \pi_\omega (\exp Y^+).
		\]
		This is an $R$-point of $C_\sigma$. Finally, since $w^{(a_2)} g \otimes k$ was zero on all weight spaces lower than $V_\omega$ by assumption, we necessarily have $Y^+ \otimes k =0$. Hence $\pi_\omega (\exp Y^+) \otimes k$ yields the torus-fixed $k$-point of $C_\sigma$, as desired.
	\end{proof}
	This strongly suggests that for resolutions of grade 3 perfect ideals, where we identify $w^{(a_2)}$ with $w^{(1)}$, the condition $w^{(1)} \otimes k \neq 0$ characterizes when such an ideal is licci, and we will show this in the next section.
	
	\section{Higher structure maps for $\mb{F}^{\sigma}$}\label{sec:generic-HSM}
	In Remark~\ref{rem:split-exact-revisited}, we saw that the construction of $\mb{F}^\sigma$ for $\sigma = e \in W$ implicitly gave a choice of higher structure maps for $\mb{F}^\ssc$. This is more generally true for \emph{any} choice of $\sigma$: starting with $w_\ssc \colon \Rgen \to \mb{C}$, we first base-change to $R = \mb{C}[C_\sigma]$ to get $w\colon \Rgen \otimes R \to R$, with
	\begin{align*}
		w^{(1)}\colon L(\omega_{x_{r_1}})^\vee \otimes R &\twoheadrightarrow F_0 \otimes R\\
		w^{(2)}\colon L(\omega_{y_{r_2-2}})^\vee \otimes R &\twoheadrightarrow F_2^* \otimes R\\
		w^{(3)}\colon L(\omega_{z_{r_3}})^\vee \otimes R &\twoheadrightarrow F_2 \otimes R.
	\end{align*}
	Next we precompose by the action of $(\exp(Y)\sigma)^{-1}$ on $\Rgen \otimes R$, which is well-defined because $\Rgen$ decomposes into integrable representations of $\mf{g}$. This yields a map $w\colon \Rgen \otimes R \to R$ with
	\begin{align*}
		w^{(1)}\colon L(\omega_{x_{r_1}})^\vee \otimes R &\xto{(\exp(Y)\sigma)^{-1}} L(\omega_{x_{r_1}})^\vee \otimes R \twoheadrightarrow F_0 \otimes R\\
		w^{(2)}\colon L(\omega_{y_{r_2-2}})^\vee \otimes R &\xto{(\exp(Y)\sigma)^{-1}}L(\omega_{y_{r_2-2}})^\vee \otimes R \twoheadrightarrow F_2^* \otimes R\\
		w^{(3)}\colon L(\omega_{z_{r_3}})^\vee \otimes R &\xto{(\exp(Y)\sigma)^{-1}}L(\omega_{z_{r_3}})^\vee \otimes R \twoheadrightarrow F_2 \otimes R.
	\end{align*}
	Now we restrict to the bottom $z_1$-graded components to get the differentials $w^{(1)}_0 = d_1$, $w^{(2)}_0 = d_2^*$, and $w^{(3)}_0 = d_3$ for the complex $\Fgen$ specializes to:
	\begin{align*}
		d_1\colon F_1 \otimes R \hookrightarrow L(\omega_{x_{r_1}})^\vee \otimes R &\xto{(\exp(Y)\sigma)^{-1}} L(\omega_{x_{r_1}})^\vee \otimes R \twoheadrightarrow F_0 \otimes R\\
		d_2^*\colon F_1^* \otimes R \hookrightarrow L(\omega_{y_{r_2-2}})^\vee \otimes R &\xto{(\exp(Y)\sigma)^{-1}}L(\omega_{y_{r_2-2}})^\vee \otimes R \twoheadrightarrow F_2^* \otimes R\\
		d_3\colon F_3 \otimes R \hookrightarrow L(\omega_{z_{r_3}})^\vee \otimes R &\xto{(\exp(Y)\sigma)^{-1}}L(\omega_{z_{r_3}})^\vee \otimes R \twoheadrightarrow F_2 \otimes R.
	\end{align*}
	and we see that $w$ specializes $\Fgen$ to the resolution $\mb{F}^\sigma$.
	
	Similarly, we have that $w^{(a_2)}$ is
	\[
	L(\omega_{x_1})^\vee \otimes R \xto{(\exp(Y)\sigma)^{-1}}L(\omega_{x_1})^\vee \otimes R \twoheadrightarrow \bigwedge^{f_2} F_2^* \otimes \bigwedge^{f_3} F_3 \otimes R.
	\]
	Let $\mf{m} \subset R$ be the ideal of variables, and $\kk = R/\mf{m} = \mb{C}$. Since $Y \otimes \kk = 0$, the map $w^{(a_2)} \otimes \kk$ is nonzero only on the extremal weight space corresponding to $\sigma$.
	
	Before we analyze the maps $w^{(i)}$ further, we make a brief digression to discuss some combinatorics of weights that will be helpful in the sequel.

	\section{Higher structure maps and linkage}\label{sec:HST-linkage}
	We now apply the theory of higher structure maps to the study of linkage. Our first task is to analyze how these maps behave under linkage. Although not coming from $\Rgen$, similar ideas have been used in the past, notably in the paper \cite{AKM} of Avramov, Kustin, and Miller. In \cite{GNW-E6}, the situation is explicitly worked out for some maps $w^{(i)}_j$ where $j$ is small.
	
	\begin{example}\label{ex:split-exact-triality}
		Let us revisit Remark~\ref{rem:split-exact-revisited} regarding the situation for the standard split complex $\mb{F}^\ssc$. In Theorem~\ref{thm:ssc} we observed that $\mb{F}^\ssc$ admits a particularly simple choice of higher structure maps, where
		\begin{equation}\label{eq:HSTlink-ssc-ex}
			\begin{split}
				w^{(1)}\colon L(\omega_{x_1})^\vee &\twoheadrightarrow \mb{C} \\
				w^{(2)}\colon L(\omega_{y_{r_2-2}})^\vee &\twoheadrightarrow F_2^*\\
				w^{(3)}\colon L(\omega_{z_{r_3}})^\vee &\twoheadrightarrow F_2
			\end{split}
		\end{equation}
		are just given by projection onto the bottom $x_1$-graded component.
		
		Suppose we interchange the roles of the $y$ and $z$ arms, using the vertex $y_1$ in place of $z_1$ and defining $F_1'$, $F_2'$, and $F_3'$ for the new diagram analogously to how $F_1$, $F_2$, and $F_3$ were defined for the original diagram. If the original diagram corresponded to the format ${\underline{f}} = (1,f_1,f_2,f_3) = (1,3+d,2+d+t,t)$, then the new diagram is for the format $\underline{f}' = (1,3+t,2+d+t,d)$. Theorem~\ref{thm:ssc} tells us that
		\begin{align*}
			L(\omega_{x_{r_1}}) &\twoheadrightarrow \mb{C}\\
			L(\omega_{z_{r_3}}) &\twoheadrightarrow F_2'^*\\
			L(\omega_{y_{r_2-2}}) &\twoheadrightarrow F_2'
		\end{align*}
		describes higher structure maps for the standard split complex $(\mb{F}^\ssc)'$ of format $\underline{f}'$. But we have not changed the role of the vertex $x_1$, so if we identify $F_2^* \cong F_2'$ then these are the same maps as in \eqref{eq:HSTlink-ssc-ex}, just with the roles of $w^{(2)}$ and $w^{(3)}$ interchanged! It is only \emph{after} we restrict to the bottom graded components that we really violate the symmetry: the bottom $y_1$-graded components recover the differentials of $(\mb{F}^\ssc)'$, whereas the bottom $z_1$-graded components recover the differentials of $\mb{F}^\ssc$.
		
		Hence the \emph{same} maps from Theorem~\ref{thm:ssc}, decomposed and interpreted with respect to the vertex $y_1$ instead of $z_1$, describe structure maps for $(\mb{F}^\ssc)'$.
	\end{example}
	
	In the following, let $\underline{f} = (1,3+d,2+d+t,t)$ and $\underline{f}' = (1, 3+t,2+d+t,d)$. We write $(\Rgen,\Fgen) = (\Rgen(\underline{f}),\Fgen(\underline{f}))$ and $(\Rgen',{\Fgen}') = (\Rgen(\underline{f}'),\Fgen(\underline{f}'))$. Just as how we defined $F_i, \mbf{L}, \ldots$, we let $F_i',\mbf{L}',\ldots$, be the analogous objects for $\underline{f}'$.
	
	If $w\colon \Rgen \to R$ specializes $\Fgen$ to some split complex $\mb{F}$ over $R$, then there exists
	\[
	g\in \exp (\mbf{L} \cotimes R) \rtimes \prod \GL(F_i \otimes R)
	\]
	such that $w_\ssc g = w$ as maps $\Rgen \otimes R \to R$, where $w_\ssc g$ means to precompose $w_\ssc$ by the action of $g$ on $\Rgen$.
	
	The idea is that, since we have identified $w_\ssc$ with $w'_\ssc$ in Example~\ref{ex:split-exact-triality} (1), we perform the same action on $w'_\ssc$ to define $w' = w'_\ssc g$. This yields a map $\Rgen' \otimes R \to R$ with $w^{(i)} = w'^{(i)}$ by construction.
	
	Here there is a slight subtlety. We already know that $\mf{sl}(F_2) \times \mf{g}$ acts on $\Rgen'$, simply because this ring comes equipped with an action of $\mf{sl}(F_2') \times \mf{g}$ and we already identified $F_2' \cong F_2^*$ to have $w_\ssc$ and $w'_\ssc$ match up in Example~\ref{ex:split-exact-triality}. However, it takes slightly more effort to explain why $\mf{gl}(F_i)$ acts on $\Rgen'$. It suffices to show that this action can be seen using the actions of $\mf{gl}(F_i')$ and $\mf{g}$, which we already know to act on $\Rgen'$, and for this it is helpful to consider the $(y_1,z_1)$-bigrading.
	
	\subsection{Decomposing with respect to the $(y_1,z_1)$-bigrading}
	We let $K = \mb{C}^3$, and we view $\mf{sl}(K)$ as the subalgebra of $\mf{g}$ corresponding to the vertices $u,x_1$ in that order (c.f. Example~\ref{ex:inclusion-of-sl}). Given the decomposition of a $\mf{g}$-representation into $y_1$ or $z_1$-graded components, we can further decompose into the $(y_1,z_1)$-bigrading, in which each component is a representation of
	\[
	\mf{g}^{(y_1,z_1)} = \mf{sl}(K) \times \mf{sl}(F_3) \times \mf{sl}(F_3').
	\]
	At the level of $\SL$-representations, this amounts to writing $F_1 = F_3'^* \oplus K$ and $F_1' = F_3^* \oplus K$ (the duals are because of the order of vertices; see \S\ref{sec:gen-licci-differentials}). However, to correctly relate the actions of $\mf{gl}(F_i)$ and $\mf{gl}(F_i')$, we will instead use
	\begin{equation}\label{eq:yz-bigrading-rule}
		\begin{split}
			F_1' &= (F_3^* \otimes \bigwedge^3 K) \oplus K\\
			F_1 &= (F_3'^* \otimes \bigwedge^3 K) \oplus K\\
			F_2' &= F_2^* \otimes \bigwedge^3 K.
		\end{split}
	\end{equation}
	These formulas are motivated by Theorem~\ref{thm:mapping-cone}. Using the above to decompose into representations of $\mf{gl}(K) \times \mf{gl}(F_3) \times \mf{gl}(F_3')$, we find that we get the desired identifications
	\begin{equation}\label{eq:HSTlink-identify-linked-formats}
		\begin{split}
			\Rgen \supset W(d_1) = [F_1 \oplus \bigwedge^3 F_1 \otimes F_3^* \oplus \cdots] &= [F_1' \oplus \bigwedge^3 F_1' \otimes F_3'^* \oplus \cdots] = W(d_1)' \subset \Rgen'\\
			\Rgen \supset W(d_2) = F_2 \otimes [F_1^* \oplus F_1 \otimes F_3^* \oplus \cdots] &= F_2'^* \otimes [F_3' \oplus \bigwedge^2 F_1' \oplus \cdots] = W(d_3)' \subset \Rgen'\\
			\Rgen \supset W(d_3) = F_2^* \otimes [F_3 \oplus \bigwedge^2 F_1 \oplus \cdots] &= F_2' \otimes [F_1'^* \oplus F_1' \otimes F_3'^* \oplus \cdots] = W(d_2)' \subset \Rgen'
		\end{split}
	\end{equation}
	To see this, it is sufficient to verify the statement for scalars in $\mf{gl}(F_i)$ since the result holds automatically at the level of $\mf{sl}(F_i)$-representations. In lieu of this, we display some of the $(y_1,z_1)$-bigraded components in the representations $W(d_i)$ as it clarifies the situation, and we will need it shortly.
	
	We display the $y_1$-grading horizontally and the $z_1$-grading vertically. To the left, the $z_1$-graded components are written, as well as their corresponding $z_1$-degree. On the bottom, the same is done for the $y_1$-graded components. The decomposition of $W(d_1) = W(d_1)'$ is
	\bgroup\def\arraystretch{1.5}
	\begin{equation}\label{eq:w1-bigrading}\begin{array}{cc|cccc}
			
			& &&\iddots&\iddots\\
			
			\vdots & \vdots &  & F_3'^* \otimes \bigwedge^2 F_3^* \otimes S_{3,2,2} K & \iddots & \iddots\\
			
			F_3^* \otimes \bigwedge^3 F_1 & 1 & F_3^*\otimes \bigwedge^3 K & F_3'^* \otimes F_3^* \otimes S_{2,2,1} & \bigwedge^2 F_3'^* \otimes F_3^* \otimes S_{3,2,2} K & \iddots\\
			
			F_1 & 0 & K & F_3'^* \otimes \bigwedge^3 K \\
			
			\hline
			&& 0 & 1 & \cdots \\
			& & F_1' & F_3'^* \otimes \bigwedge^3 F_1' & \cdots
	\end{array}\end{equation}
	\egroup
	For $W(d_3) = W(d_2)'$, we have factored out the representation $\bigwedge^3 K \otimes F_2^*$, which is the same as $F_2'$. This is concentrated in bidegree $(0,0)$. On the right, we have displayed the bigrading for the remaining part using the same conventions as above.
	\bgroup\def\arraystretch{1.5}
	\begin{equation}\begin{array}{c}
			\bigwedge^3 K \otimes F_2^*\\ (= F_2')
		\end{array} \otimes \left[\begin{array}{cc|cccc}
			\vdots & \vdots & & & \iddots\\
			
			\bigwedge^3 K^* \otimes F_3^* \otimes \bigwedge^4 F_1 & 2 &  & F_3'^* \otimes F_3^* \otimes \bigwedge^3 K & \iddots & \iddots\\
			
			\bigwedge^3 K^* \otimes \bigwedge^2 F_1 & 1 & K^* & F_3'^* \otimes K & \bigwedge^2 F_3'^* \otimes \bigwedge^3 K& \\
			
			\bigwedge^3 K^* \otimes F_3 & 0 & F_3 \otimes \bigwedge^3 K^* \\
			
			\hline
			&& 0 & 1 & 2 & \cdots \\
			& & F_1'^* & F_3'^* \otimes F_1' & \bigwedge^2 F_3'^* \otimes \bigwedge^3 F_1' & \cdots
		\end{array}\right]\end{equation}
	\egroup
	Finally, for $W(d_2) = W(d_3)'$ we have
	\bgroup\def\arraystretch{1.5}
	\[\begin{array}{c}
		\bigwedge^3 K^* \otimes F_2\\ (= F_2'^*)
	\end{array} \otimes\left[\begin{array}{cc|cccc}
		\vdots & \vdots & & & \iddots\\
		
		\bigwedge^3 K \otimes \bigwedge^2 F_3^* \otimes \bigwedge^3 F_1 & 2 &  & \bigwedge^2 F_3^* \otimes S_{2,2,2}K& \iddots & \iddots\\
		
		\bigwedge^3 K \otimes F_3^* \otimes F_1 & 1 & & F_3^* \otimes S_{2,1,1}K & F_3'^* \otimes F_3^* \otimes S_{2,2,2}K & \\
		
		\bigwedge^3 K \otimes F_1^* & 0 & F_3' & \bigwedge^2 K \\
		
		\hline
		&& 0 & 1 & 2 & \cdots \\
		&& F_3' & \bigwedge^2 F_1' & F_3'^* \otimes \bigwedge^4 F_1' & \cdots
	\end{array}\right]\]
	\egroup
	
	Using the above, we obtain:
	\begin{thm}\label{thm:HST-linked-pair}
		Let $w\colon \Rgen \to R$ specialize $\Fgen$ to a resolution $\mb{F}$ of $R/I$. Then the maps $w^{(i)}$, re-interpreted using \eqref{eq:HSTlink-identify-linked-formats}, define a map $w' \colon \Rgen' \to R$ with the property that $w^{(i)} = w'^{(i)}$.
		
		Furthermore, if $w^{(1)}(K \otimes R) = (\alpha_1,\alpha_2,\alpha_3) \subset R$ is a complete intersection, where $K \subset L(\omega_{x_1})^\vee$ is the bottom $(y_1,z_1)$-bigraded component as in \eqref{eq:w1-bigrading}, then $w'$ specializes ${\Fgen}'$ to the resolution of $R/((\alpha_1,\alpha_2,\alpha_3):I)$ described in Theorem~\ref{thm:mapping-cone}.
	\end{thm}
	
	\begin{proof}
		As the first statement is about $w'$ satisfying the requisite relations in $\Rgen'$, we can perform the usual reduction to the split exact case, since any $\mb{F}$ is split on a dense open set. Now that we know $\mf{gl}(F_i)$ acts on $\Rgen'$, the claim follows from the argument sketched at the beginning of this section.
		
		We will abuse notation and write $F_i$ to mean $F_i \otimes R$ below, and similarly for $F_i'$. For the other part of the theorem, $w'$ specializes ${\Fgen}'$ to some complex
		\[
		0\to F_3' \xto{d_3'} F_2' \xto{d_2'} F_1' \xto{d_1'} R.
		\]
		To determine the differentials $d_i'$, we need only look at the bottom $y_1$-graded components of each $w'^{(i)} = w^{(i)}$.
		
		Examining the bigraded decompositions above, we find from \eqref{eq:w1-bigrading} that the differential $d_1'$ has two components in its source, being the sum of
		\begin{gather*}
			K  \hookrightarrow F_1  \xto{d_1} R\\
			F_3^* \otimes \bigwedge^3 K  \xto{w^{(1)}_1} R
		\end{gather*}
		The differential $d_2'$ has those two components in its target:
		\begin{gather*}
			F_2' = F_2^* \otimes \bigwedge^3 K \xto{d_3^*} F_3^* \otimes \bigwedge^3 K\hookrightarrow F_1'\\
			F_2' = F_2^* \otimes \bigwedge^3 K \xto{(w^{(3)}_1)^*} (\bigwedge^2 F_1)^* \otimes \bigwedge^3 K\to K  \hookrightarrow F_1'.
		\end{gather*}
		and $d_3'$ is just the part of $d_2^*$ given by
		\[
		F_3' \hookrightarrow F_1^* \otimes \bigwedge^3 K \xto{d_2^*} F_2^* \otimes \bigwedge^3 K = F_2'.
		\]
		Compare this to the comparison map from the Koszul complex on $w^{(1)}(K\otimes R)$ to $\mb{F}$ induced by the multiplicative structure $w^{(i)}_1$:
		\[
		\begin{tikzcd}
			0 \ar[r] & F_3 \ar[r] & F_2 \ar[r] & F_1 \ar[r] & R\\
			0 \ar[r] & \bigwedge^3 K \ar[r] \ar[u,"w^{(1)}_1"] & \bigwedge^2 K \ar[r] \ar[u,"w^{(3)}_1"] & K \ar[r] \ar[u,hook] & R \ar[u,equals]
		\end{tikzcd}
		\]
		By a miracle we have reconstructed the complex
		\[
		0 \to F_3' \xto{d_3'} F_2^* \otimes \bigwedge^3 K \xto{d_2'} F_3^* \otimes \bigwedge^3 K \oplus K \xto{d_1'} R
		\]
		described in Theorem~\ref{thm:mapping-cone}!
		
		The conscientious reader may worry that we have not carefully checked the coefficients for the two parts of $d_1'$ and $d_2'$ to ensure that $d_1' d_2' = 0$. But note that this is the specialization of ${\Fgen}'$ via $w'$, so it is guaranteed to be a complex. One just needs to consistently identify $F_3^* \otimes \bigwedge^3 K \oplus K$ as the standard representation of $\mf{gl}(F_1')$ in both the source of $d_1'$ and the target of $d_2'$.
	\end{proof}
	
	\subsection{Ranks of $w^{(i)} \otimes k$ and linkage}
	Throughout this subsection, we will work in a local Noetherian ring $(R,\mf{m},k)$. If $I \subset R$ is a grade $c$ perfect ideal and $\mb{F}$ is a minimal free resolution of $R/I$, we define the \emph{deviation} of $I$ to be $d(I) = \rank F_1 - c$, and the \emph{type} of $R/I$ (which we also call the type of $I$) to be $t(R/I) = \rank F_c$.
	\begin{definition}
		If $I \subset R$ is a grade 3 perfect ideal, let $\NL(I) \coloneqq \sum_\sigma \HSI_\sigma (R/I)$. In other words, it is the image of the map $w^{(1)}\colon L(\omega_{x_1})^\vee \otimes R \to R$ for any choice of $w\colon \Rgen \to R$ specializing $\Fgen$ to a resolution of $R/I$.
	\end{definition}
	
	\begin{thm}\label{thm:non-licci-locus}
		Let $(R,\mf{m},k)$ be a local Noetherian ring, $I \subset R$ a grade 3 perfect ideal, and $w\colon \Rgen \to R$ a map specializing $\Fgen$ to a minimal resolution $\mb{F}$ of $R/I$.
		\begin{enumerate}
			\item The ideal $\NL(I)$ is invariant under linkage.
			\item The ideal $I$ is licci if and only if $\NL(I) = (1)$.
			\item If $w^{(3)} \otimes k \neq 0$ then there exists $I'$ in the even linkage class of $I$ such that either $t(R/I') < t(R/I)$ or $I' = (1)$.
			\item If $w^{(2)} \otimes k \neq 0$ then there exists $I'$ in the even linkage class of $I$ such that either $d(I') < d(I)$ or $I' = (1)$.
		\end{enumerate}
	\end{thm}
	Since $\NL(I)$ is defined using structure maps for a free resolution of $R/I$, it commutes with localization. So for $R$ not necessarily local, point (2) allows us to interpret
	\[
	V(\NL(I)) = \{ \mf{p} \in \Spec R : I_\mf{p} \subset R_\mf{p} \text{ is not licci or the unit ideal}\}
	\]
	as the \emph{non-licci locus} of $I$, explaining the notation $\NL(I)$.
	
	\begin{proof}
		Statement (1) follows immediately from Theorem~\ref{thm:HST-linked-pair}. The ``only if'' implication of (2) is also immediate, since $I=(1)$ satisfies $\NL(I) = (1)$, and any licci ideal can be linked in some number of steps to the unit ideal.
		
		Now we prove the ``if'' statement. Since $w^{(1)} \otimes k \neq 0$, there exist elements $g_1,g_1',\ldots,g_N,g_N'$, where $g_i \in \GL(F_1 \otimes k)$ and $g_i' \in \GL(F_1' \otimes k)$, such that $(w^{(1)}\otimes k)g_1 g_1' \cdots g_N g_N'$ is nonzero on the lowest weight space of $L(\omega_{x_1})^\vee$. Such elements exist because any weight space that is a lowest weight space for both $\GL(F_1)$ and $\GL(F_1')$ must be a lowest weight space for $\mf{g}$, as the type $A$ subdiagrams corresponding to $\GL(F_1)$ and $\GL(F_1')$ cover the whole diagram $T$. So it is always possible to use either the action of $\GL(F_1)$ or $\GL(F_1')$ to make $L(\omega_{x_1})^\vee \otimes k \to k$ nonzero on lower weight spaces, until it is nonzero on the lowest weight space.
		
		To show that $I$ is licci, we inductively pick lifts $\tilde{g}_i \in \GL(F_1\otimes R)$ of $g_i$ and $\tilde{g}_i' \in \GL(F_1'\otimes R)$ of $g_i'$ so that
		\begin{align*}
			I &= w^{(1)} \tilde{g}_1(F_1 \otimes R)\\
			&\sim w^{(1)} \tilde{g}_1(F_1'\otimes R)\\
			&= w^{(1)} \tilde{g}_1\tilde{g}_1'(F_1'\otimes R)\\
			&\sim w^{(1)}\tilde{g}_1\tilde{g}_1'(F_1\otimes R)\\
			&= w^{(1)}\tilde{g}_1\tilde{g}_1'\tilde{g}_2(F_1\otimes R)\\
			&\quad \vdots\\
			&= w^{(1)}\tilde{g}_1 \tilde{g}_1' \cdots \tilde{g}_N \tilde{g}_N' (F_1'\otimes R) = (1).
		\end{align*}
		In more detail, since three general linear combinations of generators for a grade 3 ideal will form a regular sequence, a general lift $\tilde{g}_1$ of $g_1$ will result in $w^{(1)} \tilde{g}_1(K \otimes R)$ being a complete intersection, so that Theorem~\ref{thm:HST-linked-pair} gives the first link above. We repeat this process to choose lifts $\tilde{g}_1', \tilde{g}_2, \tilde{g}_2', \ldots$ for the subsequent links. Recall that $(w^{(1)}\otimes k)g_1 g_1' \cdots g_N g_N'$ is nonzero on the lowest weight space by construction, hence the last ideal in the sequence of links is the unit ideal.
		
		Statements (3) and (4) can be proved in the same manner. If some row of $w^{(3)} \otimes k$ is nonzero, then the exact same argument applied to that row (instead of $w^{(1)}$) shows that there exists a sequence of links $I = I_0 \sim I_1 \cdots \sim I_{2N} = I'$ and a resolution $\mb{F}'$ of $R/I'$ having the same format $(1,f_1,f_2,f_3)$ as the original resolution $\mb{F}$, but such that $w'^{(3)}_0 \otimes k \neq 0$ for $\mb{F}'$. This means either $I' = (1)$ or $t(R/I') < f_3 = t(R/I)$ as desired. The proof of (4) is completely analogous.
	\end{proof}
	
	In Theorem~\ref{thm:format-expansion}, we saw that the addition of a split part to a resolution $\mb{F}$ causes copies of $w^{(a_2)}$ (here identified with $w^{(1)}$) to appear in the new maps $w^{(3)}$ and $w^{(2)}$. Combining this with Theorem~\ref{thm:non-licci-locus}, we see that there is a dichotomy in how the ranks of $w^{(3)}$ and $w^{(3)}$ behave, depending on whether $I$ is licci.
	\begin{prop}\label{prop:licci-w-surj}
		If $I$ is licci then $w^{(3)}$ and $w^{(2)}$ are surjective.
	\end{prop}
	\begin{proof}
		This statement is true for the unit ideal by Theorem~\ref{thm:ssc}, preserved under linkage fixing a diagram $T$ by Theorem~\ref{thm:HST-linked-pair}, and preserved by addition of a split exact complex by Theorem~\ref{thm:format-expansion}. Note that this last point is reliant on $w^{(1)}$ being surjective.
	\end{proof}
	The situation is quite different for non-licci perfect ideals:
	\begin{lem}\label{lem:nonlicci-rank-invariants}
		Let $I$ be a grade 3 perfect ideal in a local Noetherian ring $R$. Suppose that $I$ is not licci. Let $w \colon \Rgen \to R$ specialize $\Fgen$ to a resolution $\mb{F}$ of $R/I$, with format $\underline{f}$. Then the quantities
		\[
		f_3 - \rank (w^{(3)} \otimes k),\quad f_1 - 3 - \rank(w^{(2)} \otimes k)
		\]
		are intrinsic to $I$, and are interchanged under linkage.
	\end{lem}
	\begin{proof}
		Fixing the format of $\mb{F}$, any two choices of $w$ will have the same quantities $\rank (w^{(i)} \otimes k)$ by Lemma~\ref{lem:well-def-pt1}. Since $I$ is not licci, $w^{(1)} \otimes k = 0$. Theorem~\ref{thm:format-expansion} implies that the addition of a split part of format $(n_1,n_1+n_2,n_2+n_3,n_3)$ to $\mb{F}$ increases both $f_3$ and $\rank (w^{(3)} \otimes k)$ by $n_3$. Similarly it increases both $f_1$ and $\rank (w^{(2)} \otimes k)$ by $n_2$. 
		
		Theorem~\ref{thm:HST-linked-pair} shows that the quantities are interchanged by a link, since the new format is $(1,f_3+3,f_2,f_1-3)$.
	\end{proof}
	
	\begin{thm}\label{thm:nonlicci-minimal-rep}
		Assume the setup of Lemma~\ref{lem:nonlicci-rank-invariants}. Then there is an ideal $I'$ in the even linkage class of $I$ such that
		\[
		t(R/I') = f_3 - \rank(w^{(3)} \otimes k), \quad d(I') = f_1 - 3 - \rank(w^{(2)}\otimes k).
		\]
		In particular, if $\mb{F}'$ is a minimal free resolution of $R/I'$ and $w'$ is a choice of higher structure maps for $\mb{F}'$, then $w'^{(i)} \otimes k = 0$ for all $i$. The ideal $I'$ has minimal deviation and type in its even linkage class.
	\end{thm}
	\begin{proof}
		This follows from repeated application of Theorem~\ref{thm:non-licci-locus} and Lemma~\ref{lem:nonlicci-rank-invariants} until both $w^{(2)} \otimes k$ and $w^{(3)} \otimes k$ are zero.
	\end{proof}
	
	\section{Classification of grade three licci ideals}\label{sec:classify}
	Fix a format $\underline{f}$ and its corresponding diagram $T$. Let $(R_\sigma,\mb{F}^\sigma)$ be the resolution constructed in \S\ref{sec:example-res} for $\sigma \in \leftindex^{z_1}W^{x_1}$, so $R_\sigma = \mb{C}[C_\sigma]$ is a polynomial ring in $\ell(\sigma)$ variables. Write $R_\sigma/I_\sigma$ for the module resolved by $\mb{F}^\sigma$. This is the coordinate ring of $\mc{N}_\sigma^w \subset C_\sigma$. Let $\mf{m}_\sigma$ denote the maximal ideal generated by the variables of $R_\sigma$.
	
	\begin{prop}\label{prop:generic-example-is-licci}
		If $\sigma \neq e$, the ideal $(I_\sigma)_{\mf{m}_\sigma}$ is a grade 3 licci ideal in $(R_\sigma)_{\mf{m}_\sigma}$. 
	\end{prop}
	\begin{proof}
		If $\sigma = e$ then $I_\sigma$ is the unit ideal. Otherwise it is a grade 3 perfect ideal as established in \S\ref{sec:liccires-acyclicity}. This proposition then follows from Theorem~\ref{thm:non-licci-locus}, and the observation that $\HSI_\sigma(R_\sigma/I_\sigma) = (1)$ follows from \S\ref{sec:generic-HSM}.
	\end{proof}
	
	Let $I$ be a perfect ideal in a local Noetherian ring $(R,\mf{m},k)$. Let $S$ be another local Noetherian ring, let $\phi\colon R \to S$ be a local homomorphism, and let $J = \phi(I)S$. Recall from \S\ref{sec:intro} that $J$ is a \emph{specialization} of $I$ if $\grade J \geq \grade I$, in which case it follows that in fact $\grade J = \grade I$ and $J$ is also perfect.
	
	In Definition~\ref{def:GFR-UFR}, we discussed the notion of specialization of complexes. We have that $J$ is a specialization of $I$ if and only if the minimal free resolution of $R/I$ specializes via $\phi$ to a minimal free resolution of $S/J$ (see \cite[Proposition~6.14]{Hochster75}).
	
	Henceforth we suppose that $\grade I = 3$ and the minimal free resolution of $R/I$ has format $\underline{f}$. If $J$ is a specialization of $I$ then $\HSI_\rho(S/J) = (\HSI_\rho(R/I))S$ for all $\rho$. As we assumed $\phi$ to be a local homomorphism, $\HSI_\rho(S/J) = (1)$ if and only if $\HSI_\rho(R/I) = (1)$.
	
	If $I$ is licci, then Theorem~\ref{thm:non-licci-locus} says
	\[
	\NL(I) = \sum_\sigma \HSI_\sigma(R/I) = (1).
	\]
	By Proposition~\ref{prop:classifying-cell}, there exists a unique minimal $\sigma$ for which $\HSI_\sigma(R/I) = (1)$. We denote this $\sigma$ by $\Psi(I)$. The above argument shows that $\Psi(I)$ is preserved under specialization.
	\begin{example}\label{ex:Psi-generic}
		From \S\ref{sec:generic-HSM} we know that $\Psi((I_\sigma)_{\mf{m}_\sigma}) = \sigma$.
	\end{example}
	
	Now we see the true significance of the resolutions constructed in \S\ref{sec:example-res}: they yield the generic resolutions for all grade 3 licci ideals.
	
	\begin{thm}\label{thm:grade-3-licci-classification}
		Let $I$ be a grade 3 licci ideal in a local Noetherian ring $(R,\mf{m},k)$. Then there exists a unique $\sigma$ for which $(I_\sigma)_{\mf{m}_\sigma}$ specializes to $I$, i.e. there is a homomorphism $\phi\colon R_\sigma \to R$ with $\phi(\mf{m}_\sigma) \subseteq \mf{m}$ specializing $\mb{F}^\sigma$ to a resolution of $R/I$. This unique $\sigma$ is equal to $\Psi(I)$.
	\end{thm}
	To be precise, the definition of $\mb{F}^\sigma$ depends on the diagram $T$. We can simply take the smallest diagram $T$ on which $\sigma$ is defined (c.f. Remark~\ref{rem:limit-double-coset}); enlarging $T$ only amounts to adding a split exact summand to $\mb{F}^\sigma$.
	\begin{proof}
		The uniqueness of $\sigma$ follows from Example~\ref{ex:Psi-generic} and the fact that $\Psi$ is preserved under specialization. From Proposition~\ref{prop:classifying-cell}, we get a map $w\colon \Rgen \to R$ such that
		\begin{enumerate}
			\item $w$ specializes $\Fgen$ to a minimal resolution of $R/I$,
			\item $w^{(1)}$ describes an $R$-point of the Schubert cell $C_\sigma$, and
			\item $w^{(1)}\otimes k$ is the torus-fixed $k$-point $\sigma v \in C_\sigma$.
		\end{enumerate}
		By (2), $w^{(1)}$ describes a homomorphism $\phi\colon R_\sigma \to R$. Furthermore, point (3) ensures that $\phi(\mf{m}_\sigma) \subseteq \mf{m}$. By construction
		\[
		\phi(I_\sigma)R = w^{(1)}(F_1 \otimes R) = I
		\]
		where $F_1 \otimes R \subset L(\omega_{x_1})^\vee \otimes R$ is the bottom $z_1$-graded component, and the claim follows.
	\end{proof}
	Note that we do \emph{not} claim $w\colon \Rgen \to R$ and $\phi\colon R_\sigma \to R$ specialize $\Fgen$ and $\mb{F}^\sigma$ to identical resolutions over $R$; by construction they have the same differential $d_1$ but the differentials $d_2,d_3$ may be different. Of course, they can be made to be equal after adjusting $w$ further using the action of $\prod \GL(F_i\otimes R)$, but that is unnecessary for the proof.
	
	As a corollary, we have a classification of grade 3 licci ideals up to specialization.
	\begin{thm}
		Let $T = T_{2,d+1,t+1}$ be the diagram associated to the format $(1,3+d,2+d+t,t)$. Consider the map
		\[
		\begin{tikzcd}
			\left\{
			\begin{matrix}
				\text{grade 3 licci ideals with deviation $\leq d$ and}\\
				\text{type $\leq t$ in local Noetherian rings}
			\end{matrix}\right\}\ar[d, "\Psi"]\\
			W_{z_1} \backslash W / W_{x_1} - [e]
		\end{tikzcd}
		\]
		sending $I \subset R$ to the minimal $\sigma$ for which $\HSI_\sigma(R/I) = (1)$.
		\begin{enumerate}
			\item $\Psi$ is surjective.
			\item If $J_1 \subset S_1$ and $J_2 \subset S_2$ are grade 3 licci ideals with deviation $\leq d$ and type $\leq t$, then $\Psi(J_1) = \Psi(J_2)$ if and only if there is an ideal specializing to both $J_1$ and $J_2$.
		\end{enumerate} 
	\end{thm}
	\begin{proof}
		The map is surjective by Proposition~\ref{prop:generic-example-is-licci} and Example~\ref{ex:Psi-generic}. The ``if'' part of (2) follows from $\Psi$ being preserved under specialization. The ``only if'' part of (2) follows from Theorem~\ref{thm:grade-3-licci-classification}.
	\end{proof}
	
	In other words, the relation $\approx$ from \S\ref{sec:intro} is an equivalence relation for grade 3 licci ideals, and $\Psi$ distinguishes the equivalence classes: $J_1 \approx J_2$ if and only if $\Psi(J_1) = \Psi(J_2)$.
	\begin{definition}\label{def:Herzog-class}
		The \emph{Herzog class} of a grade 3 licci ideal is its equivalence class under the relation $\approx$.
	\end{definition}
	Note that in \cite{Herzog80}, Herzog worked in the setting of power series rings in finitely many variables over a field $k$. He studied the relation $\approx'$ defined by having a common deformation, where $I\subset R = k\llbracket X_1,\ldots,X_m \rrbracket$ is a deformation of $J \subset S = k\llbracket Y_1,\ldots,Y_n\rrbracket$ (or more accurately $R/I$ is a deformation of $S/J$) if $S/J$ is the quotient of $R/I$ by a regular sequence on $R/I$.
	
	We finish by noting that $\approx$ and $\approx'$ are the same in this setting, and hence for grade 3 licci ideals, Definition~\ref{def:Herzog-class} generalizes the existing notion of Herzog classes in the literature. 
	
	\begin{lem}
		Let $J_1\subset S_1 = k\llbracket X_1,\ldots,X_m \rrbracket$ and $J_2 \subset S_2 = k\llbracket Y_1,\ldots,Y_n\rrbracket$ be grade 3 licci ideals. Then $J_1 \approx J_2$ if and only if $J_1 \approx' J_2$.
	\end{lem}
	\begin{proof}
		The ``if'' implication follows from the observation that if $I \subset R$ is a deformation of $J \subset S$, then the quotient map $R \to S$ makes $J$ a specialization of $I$. For the other implication, suppose that $J_1 \approx J_2$. Then there is some ideal specializing to both $J_1$ and $J_2$, and by Theorem~\ref{thm:grade-3-licci-classification} we can take this ideal to be $I_\sigma$ for some $\sigma$. Let $A$ be the $\mf{m}_\sigma$-adic completion of $R_\sigma$. Let $\phi_1\colon A \to S_1$ and $\phi_2\colon A \to S_2$ be the maps specializing $I_\sigma A$ to $J_1$ and $J_2$. Now consider $A/I_\sigma A \cotimes S_1 \cotimes S_2$. Quotienting by the variables of $S_2$ and the elements $X_i - \phi_1(X_i)$ yields $S_1/J_1$. Similarly, quotienting by the variables of $S_1$ and the elements $Y_i - \phi_2(Y_i)$ yields $S_2/J_2$. We conclude that $J_1 \approx' J_2$.
	\end{proof}
	
	\section{Conclusion}\label{sec:conclusion}
	Using the machinery of higher structure maps arising from $\Rgen$, we have given a complete classification of grade 3 licci ideals (over a field of characteristic zero). We may loosely summarize this as follows:
	\begin{thm*}
		There exists a countable list of licci ideals $I_\sigma \subset R_\sigma$, indexed by some combinatorial data $\sigma$, with the following property: for any grade 3 licci ideal $I$ in a local Noetherian ring $R$, there exists a unique $\sigma$ for which there is a local homomorphism $R_\sigma \to R$ specializing $I_\sigma$ to $I$. Furthermore, the minimal free resolutions of $I_\sigma$ are known.
	\end{thm*}
	
	The natural next step would be to extend this theory to licci ideals of grade $c \geq 4$. Unfortunately, the theory of higher structure maps is specific to the case of $c=3$. However, using a very different approach, the following was shown in \cite[Theorem~2.28]{Ni-thesis}:
	\begin{thm*}
		Let $c \geq 2$, $d \geq 0$, and $t \geq 1$ be a triple of integers satisfying the inequality
		\[
		\frac{1}{c-1} + \frac{1}{d+1} + \frac{1}{t+1} > 1.
		\]
		Then there exists a finite list of grade $c$ licci ideals $I_\sigma \subset R_\sigma$, indexed by some combinatorial data $\sigma$, with the following property: for any grade $c$ licci ideal $I$ in a local Noetherian ring $R$ satisfying $d(I) \leq d$ and $t(R/I) \leq t$, there exists a $\sigma$ for which there is a local homomorphism $R_\sigma \to R$ specializing $I_\sigma$ to $I$.
	\end{thm*}
	
	We now contrast this with what we have shown for $c=3$. First of all, there is assumption on the triple $(c,d,t)$. The ideals $I_\sigma$ are related to the representation theory of the Lie algebra associated to $T_{c-1,d+1,t+1}$, which is a Dynkin diagram (i.e. of type ADE) exactly when the inequality is satisfied. Accordingly, we call such $(c,d,t)$ ``ADE triples.''
	
	The proof of this theorem relies on the fact that, in the finite type setting, unions of opposite Schubert varieties are linearly defined. In this setting, the distinction between Schubert varieties and opposite Schubert varieties is superficial, as the positive and negative Borels are conjugate by the longest element of the Weyl group. In the Kac-Moody setting, this is no longer the case---so while it is still true that unions of Schubert varieties are linearly defined, this no longer implies the statement for unions of opposite Schubert varieties. This latter statement does not seem to be known in the representation theory literature. If it were proven, then the assumption on $(c,d,t)$ may be dropped (and we would again have a countable list of ideals $I_\sigma$ rather than a finite one).
	
	Another shortcoming is that the theorem does not assert uniqueness of $\sigma$, although we again conjecture it to be unique. In the proof of the theorem, given a licci ideal $I$, the determination of $\sigma$ and the construction of $R_\sigma \to R$ depend on a particular choice of links from $I$ to a complete intersection. A priori, a different choice of links may result in a specialization from a different $R_\sigma$. By contrast, for $c=3$, the uniqueness of $\sigma$ is guaranteed by the theory of higher structure maps---specifically, Lemma~\ref{lem:well-def-pt1}. To establish the uniqueness of $\sigma$ beyond $c=3$, it would suffice to find features preserved under deformation that distinguish $I_\sigma$ from $I_{\sigma'}$ for $\sigma \neq \sigma'$.
	
	The finite list of ideals $I_\sigma$ associated to ADE triples $(c,d,t)$ have all been implemented in Macaulay2 \cite{ADEPerfectIdeals}. As such, we can obtain their free resolutions as well. However, for some of the more complicated examples, Macaulay2 has difficulty computing the free resolutions explicitly. Furthermore, we do not have a uniform description of their differentials as we do for $c=3$ (c.f. \S\ref{sec:example-res}). In fact, we do not even have a way to determine the Betti numbers of $I_\sigma$ given $\sigma$. For $c=3$ this was a non-issue, as the deviation and type completely determine the Betti numbers.

	\printbibliography
\end{document}